\documentclass[11pt, a4paper]{amsart}
\usepackage[utf8]{inputenc}

\usepackage{amsmath, amsfonts, amssymb, amsthm}

\usepackage{graphicx}
\usepackage[inline]{enumitem}
\usepackage{hyperref}
\usepackage{stmaryrd} 
    \SetSymbolFont{stmry}{bold}{U}{stmry}{m}{n}
\usepackage{multicol}
\usepackage{mathtools}

\usepackage[cal=boondox]{mathalpha}

\theoremstyle{definition}
\newtheorem{definition}{Definition}[subsection]
\newtheorem{remark}[definition]{Remark}
\newtheorem*{acknowledgments}{Acknowledgments}
\newtheorem*{notation}{Notation}
\theoremstyle{plain}
\newtheorem{proposition}[definition]{Proposition}
\newtheorem{propdef}[definition]{Proposition-Definition}
\newtheorem{lemma}[definition]{Lemma}
\newtheorem{theorem}[definition]{Theorem}
\newtheorem{corollary}[definition]{Corollary}
\newtheorem{prop}{Proposition}[section]

\newtheorem{lem}[prop]{Lemma}

\theoremstyle{definition}
\newtheorem{defn}[prop]{Definition}
\newtheorem{thm}{Theorem} 

\usepackage{tikz, float, tikz-cd}
	\usetikzlibrary{arrows}

\usepackage[
    a4paper,
    top=2.5cm,
    bottom=2.5cm,
    left=2.5cm,
    right=2.5cm,
    includeheadfoot,
    headheight=14pt,
    footskip=1cm,
    marginparwidth=3cm,
    marginparsep=0.5cm
]{geometry}

\numberwithin{equation}{section}
\numberwithin{figure}{section}

\hypersetup{pdfborder={0 0 0}}

\newcommand{\Z}{\mathbb{Z}}
\newcommand{\Q}{\mathbb{Q}}

\newcommand{\K}{\mathbf{k}} 
\newcommand{\A}{\mathbf{A}} 
\newcommand{\B}{\mathbf{B}} 
\newcommand{\Betti}{\mathrm{B}} 
\newcommand{\DeRham}{\mathrm{DR}} 
\newcommand{\V}{\mathcal{V}} 
\newcommand{\W}{\mathcal{W}} 
\newcommand{\M}{\mathcal{M}} 

\newcommand{\Aut}{\mathrm{Aut}}
\newcommand{\Mor}{\mathrm{Mor}}
\newcommand{\alg}{\mathsf{alg}}
\newcommand{\Mod}{\mathsf{mod}}

\newcommand{\Ad}{\mathrm{Ad}}
\newcommand{\pr}{\mathrm{pr}}
\newcommand{\gr}{\mathrm{gr}}

\newcommand{\F}{\mathcal{F}} 

\newcommand{\col}{\mathrm{col}}
\newcommand{\row}{\mathrm{row}}

\newcommand{\coker}{\mathrm{coker}}
\newcommand{\im}{\mathrm{im}}
\newcommand{\diag}{\mathrm{diag}}
\newcommand{\fil}{\mathrm{fil}}
\newcommand{\bM}{\mathbf{M}}

\newcommand{\rmod}{\mathsf{rmod}}
\newcommand{\mat}{\mathsf{mat}}

\newcommand{\BFS}{\mathsf{BFS}}
\newlist{steplist}{enumerate}{1}
\setlist[steplist]{label=\textbf{Step \arabic*}, align=left, wide=5pt, leftmargin=5pt}
\newlist{termlist}{enumerate}{1}
\setlist[termlist]{label=\textbf{Term \arabic*}, align=left, wide=1pt}

\author{Benjamin Enriquez}
\address{\scriptsize Institut de Recherche Mathématique Avancée (UMR 7501), Université de Strasbourg, 7 rue René Descartes, 67000 Strasbourg, France}
\email{b.enriquez@math.unistra.fr}
\author{Khalef Yaddaden}
\address{\scriptsize Graduate School of Mathematics, Nagoya University, Furo-cho, Chikusa-ku, Nagoya, 464-8602, Japan.}
\email{khalef.yaddaden.c8@math.nagoya-u.ac.jp}
\title[A categorical formulation of the double shuffle theory]{A categorical formulation of the Deligne-Terasoma approach to double shuffle theory}
\date{June 18, 2025.}


\begin{document}
    \begin{abstract}
        In this paper, we introduce the notion of a bimodule with a factorization structure (BFS) and show that such a structure gives rise to an algebra morphism.
        We then prove that this framework offers an interpretation of the geometric construction underlying both the Betti and de Rham harmonic coproducts of the double shuffle theory developed in \cite{DeT, EF1, EF2, EF3}.
    \end{abstract}
    
    \maketitle
	
    {\footnotesize \tableofcontents}
    
    \section*{Introduction}

In \cite{EF1}, the first author and Furusho revisited the formalism of double shuffle relations among multiple zeta values set up by Racinet \cite{Rac}. There, they explained the ``de Rham'' nature of this formalism and constructed its ``Betti'' version, whose main objects are algebra ``harmonic'' coproducts $\Delta^{\W, \DeRham}$ and $\Delta^{\W, \Betti}$, respectively equipping the algebras $\W^\DeRham$ and $\W^\Betti$ with a Hopf algebra structure. The algebras $\W^\DeRham$ and $\W^\Betti$ can be seen as subalgebras of respectively $\V^\Betti$, the group algebra of the free group with two generators and $\V^\DeRham$, the enveloping algebra of the free Lie algebra with two generators.
The authors used this formalism to prove that the associator relations between the multiple zeta values imply the double shuffle relations. This is formulated as the inclusion of the torsor of associators in the double shuffle torsor \cite{EF2}. To this end, they constructed isomorphisms relating the Betti and de Rham sides and showed that any associator relates the algebra coproducts $\Delta^{\W, \DeRham}$ and $\Delta^{\W, \Betti}$ (\cite[Theorem 10.9]{EF1}). Rather than relying on Bar constructions as in \cite{Fur}, this alternative proof builds upon an interpretation of the harmonic coproducts in terms of infinitesimal braid Lie algebras for the de Rham side and braid groups for the Betti side, which is implicit in the unpublished work of Deligne and Terasoma \cite{DeT}.

The purpose of this paper is to formulate this construction of the harmonic coproducts in a categorical framework, which will be used in a later paper \cite{EY2} for the geometric interpretation of the cyclotomic version of the Betti and de Rham coproducts introduced by the second author in \cite{Yad}.
To this end, we define the category $\BFS$ of \emph{bimodules with factorization structures} whose objects are tuples $(\B, \bM, \A, \rho, e, \boldsymbol{r}, \boldsymbol{c})$ such that $(\B, \bM, \A, \rho)$ is a bimodule (i.e. $\bM$ is a right $\A$-module equipped with a compatible left $\B$-action $\rho$) and $(e, \boldsymbol{r}, \boldsymbol{c})$ satisfies the factorization identity (see Definition \ref{def:BFS})
\[
    \rho(e) = \boldsymbol{c} \circ \boldsymbol{r}.
\]
We then construct a functor $\BFS \to \mathsf{Mor}(\alg)$ which associates to any object $(\B, \bM, \A, \rho, e, \boldsymbol{r}, \boldsymbol{c})$ of $\BFS$, an algebra morphism given by (see Proposition-Definition \ref{Delta})
\begin{equation} \label{eq:def_of_Delta}
    \B \ni b \mapsto \boldsymbol{r} \circ \rho(b) \circ \boldsymbol{c}(1_\A) \in \A.
\end{equation}
On the other hand, we also define categories $\BFS_\fil$ and $\BFS_\gr$ of respectively filtered and graded objects of $\BFS$ as well as a functor $\BFS_\fil \to \BFS_\gr$ which is induced by associated graded objects. This enables the construction of functors $\BFS_\fil \to \mathsf{Mor}(\alg_\fil)$ and $\BFS_\gr \to \mathsf{Mor}(\alg_\gr)$ fitting in the following diagram
\begin{equation}\label{eq:func_diag}\begin{tikzcd}
    \BFS \ar[d, "\text{Sec. } \ref{BFS_Moralg}"'] & \BFS_\fil \ar[l, "\eqref{BFSfil_BFS}"'] \ar[r, "\text{Sec. } \ref{BFSfil_BFSgr}"] \ar[d, "\text{Sec. } \ref{BFSfil_Moralgfil}"'] & \BFS_\gr \ar[d, "\text{Sec. } \ref{BFSgr_Moralggr}"] \\
    \mathsf{Mor}(\alg) & \mathsf{Mor}(\alg_\fil) \ar[l, "\eqref{Moralgfil_Moralggr}"'] \ar[r, "\eqref{Moralgfil_Moralg}"] & \mathsf{Mor}(\alg_\gr) 
\end{tikzcd}\end{equation}
Each of the squares gives rise to two functors, and one checks that there are natural equivalences relating them.

The Betti and de Rham formalism of \cite{EF1} can be interpreted within this framework through the objects (see Proposition-Definitions \ref{propdef:facto_struct_MB} and \ref{propdef:facto_struct_MDR})
\[
    \mathcal{O}^\Betti := (\V^\Betti, \bM^\Betti, \V^\Betti \otimes \V^\Betti, \rho^\Betti, e^\Betti, \boldsymbol{r}^\Betti, \boldsymbol{c}^\Betti) \in \BFS,    
\]
and
\[
    \mathcal{O}^\DeRham := (\V^\DeRham, \bM^\DeRham, \V^\DeRham \otimes \V^\DeRham, \rho^\DeRham, e^\DeRham, \boldsymbol{r}^\DeRham, \boldsymbol{c}^\DeRham) \in \BFS_\gr.
\]
A suitable filtration on the object $\mathcal{O}^\Betti$ enables us to show that it gives rise to an object $\mathcal{O}^\Betti_\fil \in \BFS_\fil$ (see Corollary \ref{cor:OBfil_contruction}). Thanks to the functor $\BFS_\fil \to \BFS_\gr$ these objects are related by the following:

\begin{thm}[Theorem \ref{thm:grbimodB_iso_bimodDR}]
    We have $\gr(\mathcal{O}^\Betti_\fil) = \mathcal{O}^\DeRham$.
\end{thm}

The bimodule parts of these objects arise from geometric constructions, but their factorization structures $(e^\Betti, \boldsymbol{r}^\Betti, \boldsymbol{c}^\Betti)$ and $(e^\DeRham, \boldsymbol{r}^\DeRham, \boldsymbol{c}^\DeRham)$ require the use of explicit objects of the BFS categories arising from representations of the algebras $\V^\Betti$ and $\V^\DeRham$, namely (see Corollaries \ref{cor:tuple_OBmat} and \ref{cor:tuple_ODRmat})
\[
    \mathcal{O}^\Betti_\mat := (\V^\Betti, (\V^\Betti \otimes \V^\Betti)^{\oplus 3}, \V^\Betti \otimes \V^\Betti, \mathsf{r}\underline{\rho}, e^\Betti, \mathsf{r}\underline{\row}, \mathsf{r}\underline{\col}) \in \BFS,
\]
and
\[
    \mathcal{O}^\DeRham_\mat := (\V^\DeRham, (\V^\DeRham \otimes \V^\DeRham)^{\oplus 3}, \V^\DeRham \otimes \V^\DeRham, \mathsf{r}\rho, e^\DeRham, \mathsf{r}\row, \mathsf{r}\col) \in \BFS_\gr.
\]
The prefix ``$\mathsf{r}$'' signifies that the constructed objects differ from the objects with the same notation from \cite{EF1} by the fact that we consider right actions instead of left actions.
Then, once again, a suitable filtration on the object $\mathcal{O}^\Betti_\mat$ enables us to show that it gives rise to an object $\mathcal{O}^\Betti_{\mat, \fil} \in \BFS_\fil$ (see Corollary \ref{cor:tuple_OBmatfil}).

We show that the geometric objects $\mathcal{O}^\Betti_\fil$ and $\mathcal{O}^\DeRham$ and the explicit objects $\mathcal{O}^\Betti_{\mat, \fil}$ and $\mathcal{O}^\DeRham_\mat$ are respectively related thanks to bimodule isomorphisms:
\begin{thm}[Theorems \ref{thm:betti_bimod_iso_fil} and \ref{thm:derham_bimod_iso}]
    We have
    \[
        \bM^\Betti \simeq (\V^\Betti \otimes \V^\Betti)^{\oplus 3} \text{ and } \bM^\DeRham \simeq (\V^\DeRham \otimes \V^\DeRham)^{\oplus 3}.
    \]
\end{thm}

Finally, the explicit BFS objects enables the use of the functor $\BFS \to \mathsf{Mor}(\alg)$ to construct coproducts $\Delta_{\mathcal{O}^\Betti_\mat}$ and $\Delta_{\mathcal{O}^\DeRham_\mat}$, as described in \eqref{eq:def_of_Delta}; then we identify them with the coproducts $\Delta^{\W, \Betti}$ and $\Delta^{\W, \DeRham}$ thanks to the following result:
\begin{thm}[Theorems \ref{thm:DeltaWB_DeltaOB} and \ref{thm:DeltaWDR_DeltaODR}]
    For $b$ in $\V^\Betti$ (resp. $\V^\DeRham$), we have
    \[
        \Delta_{\mathcal{O}^\Betti_{\mat, \fil}}(b) = \Delta^{\W, \Betti}(b \ e^\Betti) \quad (\text{resp.} \quad \Delta_{\mathcal{O}^\DeRham_\mat}(b) = \Delta^{\W, \DeRham}(b \ e^\DeRham)).
    \]
\end{thm}

    \begin{acknowledgments}
        This project was partially supported by first author's ANR grant Project HighAGT ANR20-CE40-0016 and second author's JSPS KAKENHI Grant 23KF0230.
    \end{acknowledgments}

    \begin{notation}
        Throughout this paper, let $\K$ be a commutative $\Q$-algebra.    
    \end{notation}
    
    \section{Basic categories of filtered and graded algebras and modules}
We introduce here the basic categories and functors relating them that we shall refer to throughout this paper.
\subsection{The categories \texorpdfstring{$\K{\text-}\Mod$}{k-mod}, \texorpdfstring{$\K{\text-}\Mod_\fil$}{k-modfil} and \texorpdfstring{$\K{\text-}\Mod_\gr$}{k-modgr}}
\begin{definition}
    \begin{enumerate}[label=(\alph*), leftmargin=*, itemsep=2mm]
        \item $\K\text{-}\Mod$ is the category of $\K$-modules;
        \item $\K{\text-}\Mod_{\fil}$ is the category of \emph{filtered} $\K$-modules, that is, $\K$-modules $\bM$ equipped with a decreasing sequence of $\K$-submodules $(\F^n\bM)_{n \in \Z}$ called filtration. Morphisms are filtered $\K$-module morphisms, that is, $\K$-module morphisms $\boldsymbol{\varphi} : \bM \to \bM^\prime$ which are compatible with the filtrations on both sides. We denote by $\F^n \boldsymbol{\varphi} : \F^n\bM \to \F^n\bM^\prime$ the induced $\K$-module morphism corresponding to $n \in \Z$;
        \item $\K{\text-}\Mod_{\gr}$ is the category of $\Z$-\emph{graded} $\K$-modules, that is, $\K$-module for which there exists a sequence of $\K$-submodules $(\bM_n)_{n \in \Z}$, called grading, such that $\bM = \bigoplus_{n \in \Z} \bM_n$. Morphisms are graded $\K$-module morphisms, that is, $\K$-module morphisms $\boldsymbol{\varphi} : \bM \to \bM^\prime$ such that $\varphi(\bM_n) \subset \bM_n^\prime$ for any $n \in \Z$. We denote by $\boldsymbol{\varphi}_n : \bM_n \to \bM^\prime_n$ the induced $\K$-module morphism corresponding to $n \in \Z$.
    \end{enumerate}
\end{definition}

Recall that there is an \emph{associated graded functor} $\gr : \K{\text-}\Mod_\fil \rightarrow \K{\text-}\Mod_\gr$ which takes a filtered $\K$-module $\left(\bM, (\F^n\bM)_{n \in \Z}\right)$ to the graded $\K$-module $\gr(\bM) := \bigoplus_{n \in \Z} \gr_n(\bM)$, where $\gr_n(\bM) := \F^n\bM / \F^{n+1}\bM$ for any $n \in \Z$. Denote by $x \mapsto [x]_n$ the canonical projection $\F^n\bM \twoheadrightarrow \gr_n\bM$ for any $n \in \Z$. At the level of morphisms, one assigns to a filtered $\K$-module morphism $\boldsymbol{\varphi} : \bM \to \bM^\prime$, the graded $\K$-module morphism $\gr(\boldsymbol{\varphi}) : \gr(\bM) \to \gr(\bM^\prime)$ induced by the commutative diagram
\[\begin{tikzcd}
    \F^n\bM \ar[rr, "\F^n\boldsymbol{\varphi}"] \ar[d, two heads] && \F^n\bM^\prime \ar[d, two heads] \\
    \gr_n(\bM) \ar[rr, ""] && \gr_n(\bM^\prime)
\end{tikzcd}\]
for any $n \in \Z$ (see for example \cite[Chap. III, Sec. 2, no. 3 and no. 4]{Bbk}).

\subsection{The categories \texorpdfstring{$\K{\text-}\alg$}{k-mod}, \texorpdfstring{$\K{\text-}\alg_\fil$}{k-modfil} and \texorpdfstring{$\K{\text-}\alg_\gr$}{k-modgr}}
\begin{definition}
\begin{enumerate}[label=(\alph*), leftmargin=*, itemsep=2mm]
    \item $\K\text{-}\alg$ is the category of $\K$-algebras (with unit);
    \item $\K{\text-}\alg_{\fil}$ is the category of filtered $\K$-algebras, that is, $\K$-algebras $(\A, \cdot)$ equipped with a $\K$-module filtration $(\F^n\A)_{n \in \Z}$ that satisfies $\F^k\A \cdot \F^n\A \subset \F^{k+n}\A$, for any $k, n \in \Z$ and $1 \in \F^0\A$. Morphisms are filtered $\K$-algebra morphisms, that is, $\K$-algebra morphisms which are also filtered $\K$-module morphisms;
    \item $\K{\text-}\alg_{\gr}$ is the category of $\Z$-graded $\K$-algebras, that is, $\K$-algebras $(\A, \cdot)$ equipped with a $\K$-module grading $(\A_n)_{n \in \Z}$ that satisfies $\A_k \cdot \A_n \subset \A_{k+n}$ for any $k, n \in \Z$ and $1 \in \A_0$. Morphisms are graded $\K$-algebra morphisms, that is, $\K$-algebra morphisms which are also graded $\K$-module morphisms.
\end{enumerate}
\end{definition}

One immediately checks that if $\A$ is an object of $\K{\text-}\alg_\fil$, then $\gr(\A)$ is an object of $\K{\text-}\alg_\gr$. Moreover, for a filtered $\K$-algebra morphism $\boldsymbol{\varphi} : \A \to \A^\prime$, the map $\gr(\boldsymbol{\varphi}) : \gr(\A) \to \gr(\A^\prime)$ is a graded $\K$-algebra morphism (see for example \cite[Chap. III, Sec. 2, no. 3 and no. 4]{Bbk}). One then defines a functor $\K{\text-}\alg_\fil \to \K{\text-}\alg_\gr$, which we also denote $\gr$, such that we have a natural equivalence that we summarize in the following diagram
\[\begin{tikzcd}
    \K{\text-}\alg_\fil \ar[rr] \ar[d, "\gr"'] && \K{\text-}\Mod_\fil \ar[d, "\gr"] \\
    \K{\text-}\alg_\gr \ar[rr] && \K{\text-}\Mod_\gr
\end{tikzcd}\]
where the horizontal arrows are forgetful functors.

\subsection{The categories \texorpdfstring{$\A{\text-}\rmod$}{k-mod}, \texorpdfstring{$\A{\text-}\rmod_\fil$}{k-modfil} and \texorpdfstring{$\A{\text-}\rmod_\gr$}{k-modgr}}
\begin{definition}
\begin{enumerate}[label=(\alph*), leftmargin=*, itemsep=2mm]
    \item For $\A \in \K{\text-}\alg$, denote by $\A\text{-}\rmod$ the category of right $\A$-modules. In particular, one has $\A \in \A\text{-}\rmod$.
    \item For $\A \in \K{\text-}\alg_\fil$, denote by $\A\text{-}\rmod_\fil$ the category of \emph{filtered right modules over the filtered algebra $\A$}, that is, right $\A$-modules $\bM$ which are equipped with a $\K$-module filtration $(\F^n\bM)_{n \in \Z}$ that satisfies $\F^k\bM \cdot \F^n\A \subset \F^{k+n}\bM$ for any $k, n \in \Z$. Morphisms are filtered right $\A$-module morphisms. In particular, one has $\A \in \A{\text-}\rmod_\fil$;
    \item For $\A \in \K{\text-}\alg_\gr$, denote by$\A\text{-}\rmod_\gr$ the category of \emph{graded right modules over the graded algebra $\A$}, that is, right $\A$-modules $\bM$ which are equipped with a $\K$-module grading $(\bM_n)_{n \in \Z}$ that satisfies $\bM_k \cdot \A_n \subset \bM_{k+n}$ for any $k, n \in \Z$. Morphisms are graded right $\A$-module morphisms. In particular, one has $\A \in \A{\text-}\rmod_\gr$.
\end{enumerate}
\end{definition}

Let $\A$ be an object of $\K{\text-}\alg_\fil$. One immediately checks that if $\bM$ is an object of $\A{\text-}\rmod_\fil$, then $\gr(\bM)$ is an object of $\gr(\A){\text-}\rmod_\gr$, the structure maps of which arise from the vertical cokernels of the collection of commutative diagrams 
\[\begin{tikzcd}
    \F^k\bM \otimes \F^{n+1}\A + \F^{k+1}\bM \otimes \F^n\A \ar[rr] \ar[d, hook] && \F^{k+n+1}\bM \ar[d, hook'] \\
    \F^k\bM \otimes \F^n\A \ar[rr] && \F^{k+n}\bM
\end{tikzcd}\]
indexed by $k,n \in \Z$. Moreover, if $\varphi : \bM \to \bM^\prime$ is a morphism of $\A{\text-}\rmod_\fil$, then the map $\gr(\varphi) : \gr(\bM) \to \gr(\bM^\prime)$ is a morphism of $\gr(\A){\text-}\rmod_\gr$. \newline
This defines a functor $\gr : \A{\text-}\rmod_\fil \to \gr(\A){\text-}\rmod_\gr$.
    \section{The bimodule with factorization structure categories} 

In this section, we introduce bimodules with factorization structures (BFS), a central concept for our framework (Definition \ref{def:BFS}). We further define their filtered and graded counterparts, extending the construction to settings where the underlying algebraic structures are equipped with filtrations or gradings (Definitions \ref{def:BFSfil} and \ref{def:BFSgr}). We show that any BFS gives rise to an algebra morphism (Proposition-Definition \ref{Delta}). From this, it will follow that a filtered (resp. graded) BFS yields a filtered (resp. graded) algebra morphism (Corollary \ref{cor:Deltafil}, resp. Corollary \ref{cor:Deltagr}).

\subsection{The categories \texorpdfstring{$\K{\text-}\BFS$}{k-BFS}, \texorpdfstring{$\K{\text-}\BFS_\fil$}{k-BFSfil} and \texorpdfstring{$\K{\text-}\BFS_\gr$}{k-BFSgr}}
\subsubsection{The category \texorpdfstring{$\K{\text-}\BFS$}{k-BFS}}
\begin{definition} \label{def:BFS}
    \begin{enumerate}[label=(\alph*), leftmargin=*]
        \item A \emph{$\K$-bimodule} is a tuple $(\B, \bM, \A, \rho)$ where $\A$ and $\B$ are objects of $\K{\text-}\alg$, $\bM$ is an object of $\A{\text-}\rmod$ and $\rho : \B \to \mathrm{End}_{\A{\text-}\rmod}(\bM)$ is a morphism of $\K{\text-}\alg$. The $\K$-module $\bM$ is said to have a $(\B, \A)$-bimodule structure.
        \item A \emph{factorization structure} on a $\K$-bimodule $(\B, \bM, \A, \rho)$ is a triple $(e, \boldsymbol{r}, \boldsymbol{c})$ where $e \in B$, $\boldsymbol{r} \in \Mor_{\A{\text-}\rmod}(\bM, \A)$ and $\boldsymbol{c} \in \Mor_{\A{\text-}\rmod}(\A, \bM)$ such that (equality in $\mathrm{End}_{\A{\text-}\rmod}(\bM)$)
        \[
            \rho(e) = \boldsymbol{c} \circ \boldsymbol{r}.
        \]
        \item A \emph{$\K$-bimodule with factorization structure} is a tuple $(\B, \bM, \A, \rho, e, \boldsymbol{r}, \boldsymbol{c})$ such that $(\B, \bM, \A, \rho)$ is a $\K$-bimodule and $(e, \boldsymbol{r}, \boldsymbol{c})$ is a factorization structure on it.  
    \end{enumerate}
\end{definition}

\begin{definition} \label{def:BFS_Mor}
    Let $(\B, \bM, \A, \rho)$ and $(\B^\prime, \bM^\prime, \A^\prime, \rho^\prime)$ be two $\K$-bimodules. 
    \begin{enumerate}[label=(\alph*), leftmargin=*]
        \item \label{Bimod_mor} A \emph{$\K$-bimodule morphism} between $(\B, \bM, \A, \rho)$ and $(\B^\prime, \bM^\prime, \A^\prime, \rho^\prime)$ is a triple $(\boldsymbol{g}, \boldsymbol{\varphi}, \boldsymbol{f})$ where $\boldsymbol{g} : \B \to \B^\prime$ and $\boldsymbol{f} : \A \to \A^\prime$ are morphisms of $\K{\text-}\alg$ and $\boldsymbol{\varphi} : \bM \to \bM^\prime$ a morphism of $\K{\text-}\Mod$ such that:
        \begin{enumerate}[label=(\roman*), leftmargin=*]
            \item For $a \in \A$ and $m \in \bM$,
            \[
                \boldsymbol{\varphi}(m \cdot a) = \boldsymbol{\varphi}(m) \cdot \boldsymbol{f}(a);
            \]
            \item For $b \in \B$ and $m \in \bM$
            \[
                \boldsymbol{\varphi}(\rho(b)(m)) = \rho^\prime(\boldsymbol{g}(b))(\boldsymbol{\varphi}(m)).
            \]
        \end{enumerate}
        \item If $(e, \boldsymbol{r}, \boldsymbol{c})$ and $(e^\prime, \boldsymbol{r}^\prime, \boldsymbol{c}^\prime)$ are factorization structures on $(\B, \bM, \A, \rho)$ and $(\B^\prime, \bM^\prime, \A^\prime, \rho^\prime)$ respectively, then a $\K$-bimodule morphism $(\boldsymbol{g}, \boldsymbol{\varphi}, \boldsymbol{f}) : (\B, \bM, \A, \rho) \to (\B^\prime, \bM^\prime, \A^\prime, \rho^\prime)$ is said to be \emph{compatible with the factorization structures} if:
        \begin{enumerate}[label=(\roman*), leftmargin=*]
            \item $\boldsymbol{g}(e) = e^\prime$;
            \item $\boldsymbol{f} \circ \boldsymbol{r} = \boldsymbol{r}^\prime \circ \boldsymbol{\varphi}$ (equality in $\Mor_{\A{\text-}\rmod}(\bM, \A^\prime)$);
            \item $\boldsymbol{c}^\prime \circ \boldsymbol{f} = \boldsymbol{\varphi} \circ \boldsymbol{c}$ (equality in $\Mor_{\A{\text-}\rmod}(\A, \bM^\prime)$).
        \end{enumerate}
    \end{enumerate}
\end{definition}

\begin{propdef}
    A category $\K{\text-}\BFS$ can be defined such that objects are $\K$-bimodules with factorization structures and morphisms are $\K$-bimodule morphisms compatible with factorization structures.
\end{propdef}
\begin{proof}
    Immediate.
\end{proof}

\noindent In the subsequent sections, we shall make use of the following general properties:

\begin{proposition} \label{prop:bimodtens}
    Let $(\B, \bM, \A, \rho)$ and $(\A, \bM^\prime, \B^\prime, \rho^\prime)$ be two $\K$-bimodules. Then the tuple
    \[
        (\B, \bM \otimes_\A \bM^\prime, \B^\prime, \rho_\otimes)
    \]
    is a $\K$-bimodule, where $\rho_\otimes : \B \to \mathrm{End}_{\B^\prime{\text-}\rmod}(\bM \otimes_\A \bM^\prime)$ is the $\K$-algebra morphism given by
    \[
        b \mapsto \rho(b) \otimes \mathrm{id}_{\bM^\prime}.
    \]
\end{proposition}
\begin{proof}
    Immediate verification.
\end{proof}

\begin{proposition} \label{prop:bimods_as_ker}
    Let $\A, \B \in \K{\text-}\alg$ and $\boldsymbol{\varphi} : \B \to \A$ be a surjective morphism of $\K{\text-}\alg$. Let $\bM$ be a $(\B,\B)$-bimodule. We have
    \begin{enumerate}[label=(\alph*), leftmargin=*]
        \item \label{MMker_Armod} The $\K$-module $\bM \big/ \bM \cdot \ker(\boldsymbol{\varphi})$ is a $(\B, \A)$-bimodule;
        \item There is a $(\B, \A)$-bimodule isomorphism
        \[
            \bM \otimes_\B \A \simeq \bM \big/ \bM \cdot \ker(\boldsymbol{\varphi}).
        \]
    \end{enumerate}
\end{proposition}
\begin{proof}
    \begin{enumerate}[label=(\alph*), leftmargin=*]
        \item The $\K$-module morphism $\bM \otimes \B \to \bM$ given by
        \[
            m \otimes b \mapsto m \cdot b, \text{ for } m \in \bM \text{ and } b \in \B,
        \]
        takes the $\K$-submodule $\bM \otimes \ker(\boldsymbol{\varphi}) + \bM \cdot \ker(\boldsymbol{\varphi}) \otimes \B$ of its source to the $\K$-submodule $\bM \cdot \ker(\boldsymbol{\varphi})$ of its target. This results in a commutative diagram
        \[\begin{tikzcd}
            \bM \otimes \ker(\boldsymbol{\varphi}) + \bM \cdot \ker(\boldsymbol{\varphi}) \otimes \B \ar[rr]\ar[d, hook] && \bM \cdot \ker(\boldsymbol{\varphi}) \ar[d, hook'] \\
            \bM \otimes \B \ar[rr]&& \bM
        \end{tikzcd}\]
        whose vertical cokernel is a $\K$-module morphism
        \[
            \bM \big/ \bM \cdot \ker(\boldsymbol{\varphi}) \otimes \A \to \bM \big/ \bM \cdot \ker(\boldsymbol{\varphi}),
        \]
        thus giving the right $\A$-module structure of $\bM \big/ \bM \cdot \ker(\boldsymbol{\varphi})$. The left $\B$-module structure follows from that of $\bM$ and of $\bM \cdot \ker(\boldsymbol{\varphi})$.
        \item By assumption, the morphism $\boldsymbol{\varphi}$ gives rise to the following exact sequence of left $\B$-modules
        \[
            \ker(\boldsymbol{\varphi}) \to \B \xrightarrow{\boldsymbol{\varphi}} \A \to \{0\}.
        \]
        Applying the right exact functor $\bM \otimes_\B -$ we obtain the following exact sequence of left $\B$-modules
        \[
            \bM \otimes_\B \ker(\boldsymbol{\varphi}) \to \bM \otimes_\B \B \xrightarrow{\mathrm{id}_{\bM} \otimes \boldsymbol{\varphi}} \bM \otimes_\B \A \to \{0\}.
        \]
        One derives the following left $\B$-module isomorphism
        \begin{equation} \label{eq:iso_bimod_coker}
            \bM \otimes_\B \A \simeq \coker\big(\bM \otimes_\B \ker(\boldsymbol{\varphi}) \to \bM \otimes_\B \B\big). 
        \end{equation}
        On the other hand, consider the left $\B$-module morphism $\bM \otimes \ker(\boldsymbol{\varphi}) \to \bM$ given by $m \otimes k \mapsto m \cdot k$. One checks that
        \begin{equation} \label{eq:coker_quotient}
            \coker(\bM \otimes \ker(\boldsymbol{\varphi}) \to \bM) = \bM \big/ \bM \cdot \ker(\boldsymbol{\varphi}).
        \end{equation}
        Next, we have the following commutative diagram of left $\B$-module morphisms
        \[\begin{tikzcd}
            \bM \otimes \ker(\boldsymbol{\varphi)} \ar[rr] \ar[d, two heads] && \bM \ar[d, "\simeq"] \\
            \bM \otimes_\B \ker(\boldsymbol{\varphi)} \ar[rr] && \bM \otimes_\B \B
        \end{tikzcd}\]
        Taking the horizontal cokernels we have an isomorphism
        \begin{equation} \label{eq:coker_iso}
            \coker(\bM \otimes \ker(\boldsymbol{\varphi}) \to \bM) \simeq \coker(\bM \otimes \ker(\boldsymbol{\varphi}) \to \bM \otimes_\B \B). 
        \end{equation}
        We obtain the following chain of left $\B$-module isomorphisms
        \begin{equation} \label{eq:isobimod}
            \mbox{\footnotesize$\bM \otimes_\B \A \underset{\eqref{eq:iso_bimod_coker}}{\simeq} \coker\big(\bM \otimes_\B \ker(\boldsymbol{\varphi}) \to \bM \otimes_\B \B\big) \underset{\eqref{eq:coker_iso}}{\simeq} \coker(\bM \otimes \ker(\boldsymbol{\varphi}) \to \bM) \underset{\eqref{eq:coker_quotient}}{=} \bM \big/ \bM \cdot \ker(\boldsymbol{\varphi})$}.
        \end{equation}
        It remains to prove that it is in fact a right $\A$-module morphism. Indeed, for $m \in \bM$, $a \in \A$, and $b \in \B$ such that $\boldsymbol{\varphi}(b) = a$, recall the right $\A$-module structure on $\bM \otimes_\B \A$ given by
        \[
           (m \otimes 1) \cdot a = m b \otimes 1,
        \]
        and the right $\A$-module structure on $\bM \big/ \bM \cdot \ker(\boldsymbol{\varphi})$ given by
        \[
            \big(m + \bM \cdot \ker(\boldsymbol{\varphi})\big) \cdot a = m \ b + \bM \cdot \ker(\boldsymbol{\varphi}),
        \]
        thanks to \ref{MMker_Armod}. Finally, one checks that the image of $m b \otimes 1 \in \bM \otimes_\B \A$ by the morphism \eqref{eq:isobimod} is given by $m \ b + \bM \cdot \ker(\boldsymbol{\varphi}) \in \bM \big/ \bM \cdot \ker(\boldsymbol{\varphi})$. This completes the proof.
    \end{enumerate}
\end{proof}

\begin{proposition} \label{prop:pullback}
    Let $(\B, \bM, \A, \rho, e, \boldsymbol{r}, \boldsymbol{c})$ be an object of $\K{\text-}\BFS$ and $(\B^\prime, \bM^\prime, \A^\prime, \rho^\prime)$ be a $\K$-bimodule. Assume that there exists a $\K$-bimodule isomorphism
    \[
        (\boldsymbol{g}, \boldsymbol{\varphi}, \boldsymbol{f}) : (\B, \bM, \A, \rho) \to (\B^\prime, \bM^\prime, \A^\prime, \rho^\prime).
    \]
    Set $e^\prime := \boldsymbol{g}(e) \in \B^\prime$, $\boldsymbol{r}^\prime := \boldsymbol{f} \circ \boldsymbol{r} \circ \boldsymbol{\varphi}^{-1} \in \Mor_{\A^\prime{\text-}\rmod}(\bM^\prime, \A^\prime)$ and $\boldsymbol{c}^\prime := \boldsymbol{\varphi} \circ \boldsymbol{c} \circ \boldsymbol{f}^{-1} \in \Mor_{\A^\prime{\text-}\rmod}(\A^\prime, \bM^\prime)$. Then, the tuple
    \[
        (\B^\prime, \bM^\prime, \A^\prime, \rho^\prime, e^\prime, \boldsymbol{r}^\prime, \boldsymbol{c}^\prime)
    \]
    is an object of $\K{\text-}\BFS$. 
\end{proposition}
\begin{proof}
    We have
    \[
        \rho^\prime(e^\prime) = \rho^\prime(\boldsymbol{g}(e)) =  \boldsymbol{\varphi} \circ \rho(e) \circ \boldsymbol{\varphi}^{-1} = \boldsymbol{\varphi} \circ \boldsymbol{c} \circ \boldsymbol{r} \circ \boldsymbol{\varphi}^{-1} = \boldsymbol{\varphi} \circ \boldsymbol{c} \circ \boldsymbol{f}^{-1} \circ \boldsymbol{f} \circ \boldsymbol{r} \circ \boldsymbol{\varphi}^{-1} = \boldsymbol{c}^\prime \circ \boldsymbol{r}^\prime,  
    \]
    where the second equality follows from Definition \ref{def:BFS_Mor} \ref{Bimod_mor} and the third one from the equality $\rho(e) = \boldsymbol{c} \circ \boldsymbol{r}$.
\end{proof}
\subsubsection{The category \texorpdfstring{$\K{\text-}\BFS_\fil$}{k-BFSfil}}

\begin{definition} \label{def:FnMor}
    Let $\A \in \K{\text-}\alg_\fil$ and $\bM$ and $\bM^\prime \in \A{\text-}\rmod_\fil$. For $n \in \Z$, define \linebreak $\F^n \Mor_{\A{\text-}\rmod}(\bM, \bM^\prime)$ to be the set of right $\A$-module morphisms $\boldsymbol{\varphi} : \bM \to \bM^\prime$ such that $\boldsymbol{\varphi}(\F^k\bM) \subset \F^{k+n}\bM^\prime$ for any $k \in \Z$.
\end{definition}

\begin{lemma}
    \label{MorArmod_filtration}
    Let $\A \in \K{\text-}\alg_\fil$ and $\bM$ and $\bM^\prime \in \A{\text-}\rmod_\fil$. The sequence of right \linebreak $\A$-modules $(\F^n \Mor_{\A{\text-}\rmod}(\bM, \bM^\prime))_{n \in \Z}$ is decreasing and compatible with composition, that is, for $\boldsymbol{\varphi} \in \F^n \Mor_{\A{\text-}\rmod}(\bM, \bM^\prime)$ and $\boldsymbol{\varphi^\prime} \in \F^{n^\prime} \Mor_{\A{\text-}\rmod}(\bM^\prime, \bM^{\prime\prime})$, we have 
    \[
        \boldsymbol{\varphi^\prime} \circ \boldsymbol{\varphi} \in \F^{n+n^\prime} \Mor_{\A{\text-}\rmod}(\bM, \bM^{\prime\prime}).
    \]
    In particular, the $\K$-module $\mathrm{End}_{\A{\text-}\rmod}(\bM)$ equipped with the filtration $(\F^n \mathrm{End}_{\A{\text-}\rmod}(\bM))_{n \in \Z}$ is an object of $\K{\text-}\alg_\fil$. 
\end{lemma}
\begin{proof}
    Immediate.
\end{proof}

\noindent Thanks to Lemma \ref{MorArmod_filtration}, we may define the following:
\begin{definition} \label{def:BFSfil}
    \begin{enumerate}[label=(\alph*), leftmargin=*]
        \item A \emph{filtered $\K$-bimodule} is a $\K$-bimodule $(\B, \bM, \A, \rho)$ such that $\A$ and $\B$ are objects of $\K{\text-}\alg_\fil$, $\bM$ is an object of $\A{\text-}\rmod_\fil$ and $\rho : \B \to \mathrm{End}_{\A{\text-}\rmod}(\bM)$ is a morphism of $\K{\text-}\alg_\fil$. 
        \item A \emph{filtered factorization structure} on a filtered $\K$-bimodule $(\B, \bM, \A, \rho)$ is a factorization structure $(e, \boldsymbol{r}, \boldsymbol{c})$ on the $\K$-bimodule $(\B, \bM, \A, \rho)$ such that
        \[
            e \in \F^1\B, \, \boldsymbol{r} \in \F^0\Mor_{\A{\text-}\rmod}(\bM, \A) \text{ and } \boldsymbol{c} \in \F^1\Mor_{\A{\text-}\rmod}(\A, \bM).
        \]
        \item A \emph{filtered $\K$-bimodule with factorization structure} is a filtered $\K$-bimodule equipped with a filtered factorization structure.  
    \end{enumerate}
\end{definition}

\begin{remark}
    For a filtered factorization structure $(e, \boldsymbol{r}, \boldsymbol{c})$ on a filtered $\K$-bimodule $(\B, \bM, \A, \rho)$, the identity $\rho(e) = \boldsymbol{c} \circ \boldsymbol{r}$ is an equality in $\F^1\mathrm{End}_{\A{\text-}\rmod}(\bM)$.
\end{remark}

\begin{lemma}\label{lem:morAM_M_isofil}
    If $\A$ is an object of $\K{\text-}\alg_\fil$ and $\bM$ an object of $\A{\text-}\rmod_\fil$, then the filtered $\K$-modules $\Mor_{\A{\text-}\rmod}(\A, \bM)$ and $\bM$ are isomorphic.     
\end{lemma}
\begin{proof}
    The maps
    \[
        \Mor_{\A{\text-}\rmod}(\A, \bM) \to \bM, \ \phi \mapsto \phi(1) \text{ and } \bM\to\Mor_{\A{\text-}\rmod}(\A, \bM), \ m \mapsto (a \mapsto ma)
    \]
    are compatible with the filtrations, and are inverse isomorphisms. 
\end{proof}

\begin{definition}
    A \emph{filtered $\K$-bimodule morphism compatible with factorization structures} is a morphism $(\boldsymbol{g}, \boldsymbol{\varphi}, \boldsymbol{f}) : (\B, \bM, \A, \rho, e, \boldsymbol{r}, \boldsymbol{c}) \to (\B^\prime, \bM^\prime, \A^\prime, \rho^\prime, e^\prime, \boldsymbol{r}^\prime, \boldsymbol{c}^\prime)$ of $\K{\text-}\BFS$ such that $\boldsymbol{f}$, $\boldsymbol{g}$ are morphisms of $\K{\text-}\alg_\fil$ and $\boldsymbol{\varphi}$ is a morphism of $\K{\text-}\Mod_\fil$.  
\end{definition}

\begin{lemma}
    A category $\K{\text-}\BFS_\fil$ can be defined such that objects are filtered $\K$-bimodules with factorization structures and morphisms are filtered $\K$-bimodule morphisms compatible with factorization structures.
\end{lemma}
\begin{proof}
    Immediate.
\end{proof}

\begin{remark}
    The forgetful functors $\K{\text-}\Mod_\fil \to \K{\text-}\Mod$ and $\K{\text-}\alg_\fil \to \K{\text-}\alg$ induce a functor
    \begin{equation} \label{BFSfil_BFS}
        \K{\text-}\BFS_\fil \to \K{\text-}\BFS.
    \end{equation}
\end{remark}

\noindent In the subsequent sections, we shall make use of the following pullback property:
\begin{proposition} \label{prop:pullback_fil}
    Let $(\B, \bM, \A, \rho, e, \boldsymbol{r}, \boldsymbol{c})$ be an object of $\K{\text-}\BFS_\fil$ and $(\B^\prime, \bM^\prime, \A^\prime, \rho^\prime)$ be a filtered $\K$-bimodule. Assume that there exists a filtered $\K$-bimodule isomorphism
    \[
        (\boldsymbol{g}, \boldsymbol{\varphi}, \boldsymbol{f}) : (\B, \bM, \A, \rho) \to (\B^\prime, \bM^\prime, \A^\prime, \rho^\prime).
    \]
    Set $e^\prime := \boldsymbol{g}(e) \in \B^\prime$, $\boldsymbol{r}^\prime := \boldsymbol{f} \circ \boldsymbol{r} \circ \boldsymbol{\varphi}^{-1} \in \Mor_{\A^\prime{\text-}\rmod}(\bM^\prime, \A^\prime)$ and $\boldsymbol{c}^\prime := \boldsymbol{\varphi} \circ \boldsymbol{c} \circ \boldsymbol{f}^{-1} \in \Mor_{\A^\prime{\text-}\rmod}(\A^\prime, \bM^\prime)$. Then, the tuple
    \[
        (\B^\prime, \bM^\prime, \A^\prime, \rho^\prime, e^\prime, \boldsymbol{r}^\prime, \boldsymbol{c}^\prime)
    \]
    is an object of $\K{\text-}\BFS_\fil$. 
\end{proposition}
\begin{proof}
    Recall from Proposition \ref{prop:pullback} that the tuple
    \((\B^\prime, \bM^\prime, \A^\prime, \rho^\prime, e^\prime, \boldsymbol{r}^\prime, \boldsymbol{c}^\prime)\)
    is an object of $\K{\text-}\BFS$. Moreover, since $e \in \F^1\B$ (resp. $\boldsymbol{r} \in \F^0\Mor_{\A{\text-}\rmod}(\bM, \A)$ and $\boldsymbol{c} \in \F^1\Mor_{\A{\text-}\rmod}(\A, \bM)$) and $\boldsymbol{g} : \B \to \B^\prime$ (resp $\boldsymbol{\varphi} : \bM \to \bM^\prime$ and $\boldsymbol{f} : \A \to \A^\prime$) preserve filtrations, it follows that $e^\prime = \boldsymbol{g}(e) \in \F^1\B^\prime$ (resp. $\boldsymbol{r}^\prime = \boldsymbol{f} \circ \boldsymbol{r} \circ \boldsymbol{\varphi}^{-1} \in \F^0\Mor_{\A^\prime{\text-}\rmod}(\bM^\prime, \A^\prime)$ and $\boldsymbol{c}^\prime =  \boldsymbol{\varphi} \circ \boldsymbol{c} \circ \boldsymbol{f}^{-1} \in \F^1\Mor_{\A^\prime{\text-}\rmod}(\A^\prime, \bM^\prime)$).
\end{proof}

\subsubsection{The category \texorpdfstring{$\K{\text-}\BFS_\gr$}{k-BFSgr}}
\begin{definition}
    \label{grading_Mor}
    Let $\A \in \K{\text-}\alg_\gr$ and $\bM$ and $\bM^\prime \in \A{\text-}\rmod_\gr$. For $n \in \Z$, define \linebreak $\Mor_{\A{\text-}\rmod}(\bM, \bM^\prime)_n$ the set of right $\A$-module morphisms $\varphi : \bM \to \bM^\prime$ such that \linebreak $\varphi(\bM_k) \subset \bM^\prime_{k+n}$ for any $k \in \Z$.
\end{definition}

\begin{lemma}
    \label{MorArmod_grading}
    Let $\A \in \K{\text-}\alg_\gr$ and $\bM$ and $\bM^\prime \in \A{\text-}\rmod_\gr$. The sequence of right \linebreak $\A$-modules $(\Mor_{\A{\text-}\rmod}(\bM, \bM^\prime)_n)_{n \in \Z}$ defines a grading on $\Mor_{\A{\text-}\rmod}(\bM, \bM^\prime)$ and is \linebreak compatible with composition, that is, for $\boldsymbol{\varphi} \in \Mor_{\A{\text-}\rmod}(\bM, \bM^\prime)_n$ and $\boldsymbol{\varphi^\prime} \in \Mor_{\A{\text-}\rmod}(\bM^\prime, \bM^{\prime\prime})_{n^\prime}$, we have 
    \[
        \boldsymbol{\varphi^\prime} \circ \boldsymbol{\varphi} \in \Mor_{\A{\text-}\rmod}(\bM, \bM^{\prime\prime})_{n+n^\prime}.
    \]
    In particular, the $\K$-module $\mathrm{End}_{\A{\text-}\rmod}(\bM)$ equipped with the grading $(\mathrm{End}_{\A{\text-}\rmod}(\bM)_n)_{n \in \Z}$ is an object of $\K{\text-}\alg_\gr$. 
\end{lemma}
\begin{proof}
    Immediate.
\end{proof}

\noindent Thanks to Lemma \ref{MorArmod_grading}, we may define the following:
\begin{definition} \label{def:BFSgr}
    \begin{enumerate}[label=(\alph*), leftmargin=*]
        \item A \emph{graded $\K$-bimodule} is a $\K$-bimodule $(\B, \bM, \A, \rho)$ such that $\A$ and $\B$ are objects of $\K{\text-}\alg_\gr$, $\bM$ is an object of $\A{\text-}\rmod_\gr$ and $\rho : \B \to \mathrm{End}_{\A{\text-}\rmod}(\bM)$ is a morphism of $\K{\text-}\alg_\gr$. 
        \item A \emph{graded factorization structure} on a graded $\K$-bimodule $(\B, \bM, \A, \rho)$ is a factorization structure $(e, \boldsymbol{r}, \boldsymbol{c})$ on the $\K$-bimodule $(\B, \bM, \A, \rho)$ such that
        \[
            e \in \B_1, \, \boldsymbol{r} \in \Mor_{\A{\text-}\rmod}(\bM, \A)_0 \text{ and } \boldsymbol{c} \in \Mor_{\A{\text-}\rmod}(\A, \bM)_1.
        \]
        \item A \emph{graded $\K$-bimodule with factorization structure} is a graded $\K$-bimodule equipped with a graded factorization structure.  
    \end{enumerate}
\end{definition}

\begin{remark}
    For a graded factorization structure $(e, \boldsymbol{r}, \boldsymbol{c})$ on a graded $\K$-bimodule $(\B, \bM, \A, \rho)$, the identity $\rho(e) = \boldsymbol{c} \circ \boldsymbol{r}$ is an equality in $\mathrm{End}_{\A{\text-}\rmod}(\bM)_1$.
\end{remark}

\begin{definition}
    A \emph{graded $\K$-bimodule morphism compatible with factorization structures} is a morphism $(\boldsymbol{g}, \boldsymbol{\varphi}, \boldsymbol{f}) : (\B, \bM, \A, \rho, e, \boldsymbol{r}, \boldsymbol{c}) \to (\B^\prime, \bM^\prime, \A^\prime, \rho^\prime, e^\prime, \boldsymbol{r}^\prime, \boldsymbol{c}^\prime)$ of $\K{\text-}\BFS$ such that $\boldsymbol{f}$, $\boldsymbol{g}$ are morphisms of $\K{\text-}\alg_\gr$ and $\boldsymbol{\varphi}$ is a morphism of $\K{\text-}\Mod_\gr$.  
\end{definition}

\begin{lemma}
    A category $\K{\text-}\BFS_\gr$ can be defined such that objects are graded $\K$-bimodules with factorization structures and morphisms are graded $\K$-bimodule morphisms compatible with factorization structures.
\end{lemma}
\begin{proof}
    Immediate.
\end{proof}

\noindent In the subsequent sections, we shall make use of the following pullback property:
\begin{proposition} \label{prop:pullback_gr}
    Let $(\B, \bM, \A, \rho, e, \boldsymbol{r}, \boldsymbol{c})$ be an object of $\K{\text-}\BFS_\gr$ and $(\B^\prime, \bM^\prime, \A^\prime, \rho^\prime)$ be a graded $\K$-bimodule. Assume that there exists a graded $\K$-bimodule isomorphism
    \[
        (\boldsymbol{g}, \boldsymbol{\varphi}, \boldsymbol{f}) : (\B, \bM, \A, \rho) \to (\B^\prime, \bM^\prime, \A^\prime, \rho^\prime).
    \]
    Set $e^\prime := \boldsymbol{g}(e) \in \B^\prime$, $\boldsymbol{r}^\prime := \boldsymbol{f} \circ \boldsymbol{r} \circ \boldsymbol{\varphi}^{-1} \in \Mor_{\A^\prime{\text-}\rmod}(\bM^\prime, \A^\prime)$ and $\boldsymbol{c}^\prime := \boldsymbol{\varphi} \circ \boldsymbol{c} \circ \boldsymbol{f}^{-1} \in \Mor_{\A^\prime{\text-}\rmod}(\A^\prime, \bM^\prime)$. Then, the tuple
    \[
        (\B^\prime, \bM^\prime, \A^\prime, \rho^\prime, e^\prime, \boldsymbol{r}^\prime, \boldsymbol{c}^\prime)
    \]
    is an object of $\K{\text-}\BFS_\gr$. 
\end{proposition}
\begin{proof}
    Recall from Proposition \ref{prop:pullback} that the tuple
    \((\B^\prime, \bM^\prime, \A^\prime, \rho^\prime, e^\prime, \boldsymbol{r}^\prime, \boldsymbol{c}^\prime)\)
    is an object of $\K{\text-}\BFS$. Moreover, since $e \in \B_1$ (resp. $\boldsymbol{r} \in \Mor_{\A{\text-}\rmod}(\bM, \A)_0$ and $\boldsymbol{c} \in \Mor_{\A{\text-}\rmod}(\A, \bM)_0$) and \linebreak $\boldsymbol{g} : \B \to \B^\prime$ (resp. $\boldsymbol{\varphi} : \bM \to \bM^\prime$ and $\boldsymbol{f} : \A \to \A^\prime$) preserve gradings, it follows that $e^\prime = \boldsymbol{g}(e) \in \B^\prime_1$ (resp. $\boldsymbol{r}^\prime = \boldsymbol{f} \circ \boldsymbol{r} \circ \boldsymbol{\varphi}^{-1} \in \Mor_{\A^\prime{\text-}\rmod}(\bM^\prime, \A^\prime)_0$ and $\boldsymbol{c}^\prime =  \boldsymbol{\varphi} \circ \boldsymbol{c} \circ \boldsymbol{f}^{-1} \in \Mor_{\A^\prime{\text-}\rmod}(\A^\prime, \bM^\prime)_0$).
\end{proof}

\subsection{The functor \texorpdfstring{$\K{\text-}\BFS_\fil \to \K{\text-}\BFS_\gr$}{k-BFSfil->k-BFSgr}} \label{BFSfil_BFSgr}

Let $\A$ be an object of $\K{\text-}\alg_\fil$ and $\bM$ and $\bM^\prime$ be two objects of $\A{\text-}\rmod_\fil$. Let $n \in \Z$. Thanks to Definition \ref{def:FnMor}, we have the following $\K$-module morphism
\[
    \F^n\Mor_{\A{\text-}\rmod}(\bM, \bM^\prime) \otimes \F^k\bM \to \F^{n+k}\bM^\prime 
\]
for any $k \in \Z$. One immediately checks that it induces a $\K$-module morphism
\[
    \gr_n\left(\Mor_{\A{\text-}\rmod}(\bM, \bM^\prime)\right) \otimes \gr_k(\bM) \to \gr_{n+k}(\bM^\prime) 
\]
and then the $\K$-module morphism
\[
    \gr_n\left(\Mor_{\A{\text-}\rmod}(\bM, \bM^\prime)\right) \to \Mor_{\A{\text-}\rmod}(\gr_k(\bM), \gr_{n+k}(\bM^\prime)).
\]
which enables us to define the $\K$-module morphism
\[
    \gr_n\left(\Mor_{\A{\text-}\rmod}(\bM, \bM^\prime)\right) \to \Mor_{\gr_m(\A){\text-}\rmod}(\gr_k(\bM) \otimes \gr_m(\A), \gr_{n+k}(\bM^\prime) \otimes \gr_m(\A)),
\]
for any $m \in \Z$. Thanks to Definition \ref{grading_Mor}, one then defines the $\K$-module morphism
\[
    \mathcal{g}_n^{\bM,\bM^\prime} : \gr_n\left(\Mor_{\A{\text-}\rmod}(\bM, \bM^\prime)\right) \to \Mor_{\gr(\A){\text-}\rmod}(\gr(\bM), \gr(\bM^\prime))_n 
\]
to be the direct sum over $k,m \in \Z$ of these $\K$-module morphisms. In particular, we set
\[
    \mathcal{g}_n^\bM := \mathcal{g}_n^{\bM,\bM} : \gr_n\left(\mathrm{End}_{\gr(\A){\text-}\rmod}(\bM)\right) \to \mathrm{End}_{\A{\text-}\rmod}(\gr(\bM))_n
\]
\begin{lemma}
    \label{grMor_compo}
    Let $\A$ be an object of $\K{\text-}\alg_\fil$ and $\bM$ and $\bM^\prime$ be two objects of $\A{\text-}\rmod_\fil$.
    For any $n \in \Z$, the $\K$-module morphism
    \[
        \mathcal{g}_n^{\bM,\bM^\prime} : \gr_n\left(\Mor_{\A{\text-}\rmod}(\bM, \bM^\prime)\right) \to \Mor_{\gr(\A){\text-}\rmod}(\gr(\bM), \gr(\bM^\prime))_n 
    \]
    is such that the following diagram
    \[\begin{tikzcd}
        \mbox{\footnotesize $\gr_n\left(\Mor_{\A{\text-}\rmod}(\bM, \bM^\prime)\right) \otimes \gr_{n'}\left(\Mor_{\A{\text-}\rmod}(\bM^\prime, \bM^{\prime\prime})\right)$} \ar[r] \ar[d, "\mathcal{g}_n^{\bM,\bM^\prime} \otimes \mathcal{g}_{n^\prime}^{\bM^\prime,\bM^{\prime\prime}}"'] & \mbox{\footnotesize $\gr_{n+n^\prime}\left(\Mor_{\A{\text-}\rmod}(\bM, \bM^{\prime\prime})\right)$} \ar[d, "\mathcal{g}_{n+n^\prime}^{\bM,\bM^{\prime\prime}}"] \\
        \mbox{\footnotesize $\Mor_{\gr(\A){\text-}\rmod}(\gr(\bM), \gr(\bM^\prime))_n \otimes \Mor_{\gr(\A){\text-}\rmod}(\gr(\bM^\prime), \gr(\bM^{\prime\prime}))_{n^\prime}$} \ar[r] & \mbox{\footnotesize $\Mor_{\gr(\A){\text-}\rmod}(\gr(\bM), \gr(\bM^{\prime\prime}))_{n+n^\prime}$}
    \end{tikzcd}\]
    commutes. In particular, the map
    \[
        \mathcal{g}^\bM := \bigoplus_{n \in \Z} \mathcal{g}_n^{\bM} : \gr\left(\mathrm{End}_{\A{\text-}\rmod}(\bM)\right) \to \mathrm{End}_{\gr(\A){\text-}\rmod}(\gr(\bM)) 
    \]
    is a morphism of $\K{\text-}\alg_\gr$.
\end{lemma}
\begin{proof}
    Immediate.
\end{proof}

\begin{lemma}
    \begin{enumerate}[label=(\alph*), leftmargin=*]
        \item \label{BFS_ob_gr_fil} If $(\B, \bM, \A, \rho, e, \boldsymbol{r}, \boldsymbol{c})$ is an object of $\K{\text-}\BFS_\fil$, then 
        \[
            \left(\gr(\B), \gr(\bM), \gr(\A), \mathcal{g}^\bM \circ \gr(\rho), [e]_1, \mathcal{g}_0^{\bM, \A}([\boldsymbol{r}]_0), \mathcal{g}_1^{\A, \bM}([\boldsymbol{c}]_1)\right)
        \]
        is an object of $\K{\text-}\BFS_\gr$.
        \item \label{BFS_ar_gr_fil} If $(\boldsymbol{g}, \boldsymbol{\varphi}, \boldsymbol{f})$ is a morphism of $\K{\text-}\BFS_\fil$, then the triple
        \[
            \big(\gr(\boldsymbol{g}), \gr(\boldsymbol{\varphi}), \gr(\boldsymbol{f})\big)
        \]
        is a morphism of $\K{\text-}\BFS_\gr$.
    \end{enumerate}
    \label{BFS_gr_fil}
\end{lemma} 
\begin{proof}
    \begin{enumerate}[label=(\alph*), leftmargin=*]
        \item Through a direct verification, one checks that the tuple 
        \[
            \left(\gr(\B), \gr(\bM), \gr(\A), \mathcal{g}^\bM \circ \gr(\rho)\right)
        \]
        is a graded $\K$-bimodule. Regarding the factorization structure, since $e \in \F^1\B$, \linebreak $\boldsymbol{r} \in \F^0\Mor_{\A{\text-}\rmod}(\bM, \A)$ and $\boldsymbol{c} \in \F^1\Mor_{\A{\text-}\rmod}(\A, \bM)$, one then obtains $[e]_1 \in \gr_1(\B)$, $\mathcal{g}_0^{\bM, \A}([\boldsymbol{r}]_0) \in \Mor_{\gr(\A){\text-}\rmod}(\gr(\bM), \gr(\A))_0$ and $\mathcal{g}_1^{\A, \bM}([\boldsymbol{c}]_1) \in \Mor_{\gr(\A){\text-}\rmod}(\gr(\A), \gr(\bM))_1$, respectively. Next, let us consider the diagram
        \begin{equation}\label{diag:Fil_gr_Mor_End}\begin{tikzcd}
            \mbox{\footnotesize$\F^0\Mor_{\A{\text-}\rmod}(\bM, \A) \otimes \F^1\Mor_{\A{\text-}\rmod}(\A, \bM)$} \ar[r] \ar[d, "{[-]_0 \otimes [-]_1}"', two heads] & \mbox{\footnotesize$\F^1\mathrm{End}_{\A{\text-}\rmod}(\bM)$} \ar[d, "{[-]_1}", two heads] & \F^1\B \ar[l, "\F^1\rho"'] \ar[d, "{[-]_1}", two heads] \\
            \mbox{\footnotesize$\gr_0\Mor_{\A{\text-}\rmod}(\bM, \A) \otimes \gr_1\Mor_{\A{\text-}\rmod}(\A, \bM)$} \ar[r] \ar[d, "{\mathcal{g}_0^{\bM, \A} \otimes \mathcal{g}_1^{\bM, \A}}"', two heads] & \mbox{\footnotesize$\gr_1\mathrm{End}_{\A{\text-}\rmod}(\bM)$} \ar[d, "{\mathcal{g}_1^\bM}", two heads] & \gr_1(\B) \ar[l, "\gr_1\rho"'] \ar[d, equal] \\
            \mbox{\footnotesize$\Mor_{\gr(\A){\text-}\rmod}(\gr(\bM), \gr(\A))_0 \otimes \Mor_{\gr(\A){\text-}\rmod}(\gr(\A), \gr(\bM))_1$} \ar[r] & \mbox{\footnotesize$\mathrm{End}_{\gr(\A){\text-}\rmod}(\gr(\bM))_1$} & \gr_1(\B) \ar[l, "\gr_1\rho"']
        \end{tikzcd}\end{equation}
        The top squares commute by the definitions of associated graded morphisms and the bottom squares commute thanks to Lemma \ref{grMor_compo}. Recall that $\rho(e) = \boldsymbol{c} \circ \boldsymbol{r}$ (equality in $\F^1\mathrm{End}_{\A{\text-}\rmod}(\bM)$). The image of this equality by $[-]_1$ is, using the commutativity of the top squares of Diagram \eqref{diag:Fil_gr_Mor_End}, the following equality
        \[
            \gr\rho([e]_1) = [\boldsymbol{c}]_1 \circ [\boldsymbol{r}]_0 \in \gr_1\mathrm{End}_{\A{\text-}\rmod}(\bM).
        \]
        Using the commutativity of the bottom square of diagram \eqref{diag:Fil_gr_Mor_End}, the image of this equality by $\mathcal{g}^\bM_1$ is the following equality
        \[
            \mathcal{g}^\bM_1 \circ \gr\rho([e]_1) = \mathcal{g}^{\A, \bM}_1([\boldsymbol{c}]_1) \circ \mathcal{g}^{\bM, \A}_0([\boldsymbol{r}]_0) \in \mathrm{End}_{\gr(\A){\text-}\rmod}(\gr(\bM))_1   
        \]
        \item Direct verification.
    \end{enumerate}
\end{proof}

\begin{corollary}
    The assignment given by \ref{BFS_ob_gr_fil} and \ref{BFS_ar_gr_fil} of Lemma \ref{BFS_gr_fil} defines a functor $\K{\text-}\BFS_\fil \to \BFS_\gr$.
\end{corollary}
\begin{proof}
    Immediate.
\end{proof}

\subsection{The functors \texorpdfstring{$\K{\text-}\BFS \to \mathsf{Mor}(\K{\text-}\alg)$}{BFS->Mor(k-alg)}, \texorpdfstring{$\K{\text-}\BFS_\fil \to \mathsf{Mor}(\K{\text-}\alg_\fil)$}{BFSfil->Mor(k-algfil)} and \texorpdfstring{$\K{\text-}\BFS_\gr \to \mathsf{Mor}(\K{\text-}\alg_\gr)$}{BFSgr->Mor(k-alggr)}}
For any category $\mathcal{C}$, define the category $\mathsf{Mor}(\mathcal{C})$ whose objects are morphisms of $\mathcal{C}$ and whose morphisms are commutative diagrams. \newline
The forgetful functor $\K{\text-}\alg_\fil \to \K{\text-}\alg$ induces a functor
\begin{equation}
    \label{Moralgfil_Moralg}
    \mathsf{Mor}(\K{\text-}\alg_\fil) \to \mathsf{Mor}(\K{\text-}\alg).
\end{equation}
Moreover, the functor $\gr : \K{\text-}\alg_\fil \to \K{\text-}\alg_\gr$ induces a functor
\begin{equation}
    \label{Moralgfil_Moralggr}
    \mathsf{Mor}(\K{\text-}\alg_\fil) \to \mathsf{Mor}(\K{\text-}\alg_\gr).
\end{equation}

\subsubsection{The functor \texorpdfstring{$\K{\text-}\BFS \to \mathsf{Mor}(\K{\text-}\alg)$}{BFS->Mor(k-alg)}} \label{BFS_Moralg}
\begin{lemma}
    Let $\B$ be an object of $\K{\text-}\alg$ and $e \in \B$. For $b, b^\prime \in \B$, denote $b \cdot_e b^\prime := b e b^\prime$. Then $\K \oplus (\B, \cdot_e)$ is an object of $\K{\text-}\alg$ whose product\footnote{which will be abusively denoted ``$\cdot_e$" as well.} is given explicitly by
    \[
        (\lambda, b) \cdot_e (\lambda^\prime, b^\prime) := (\lambda \lambda^\prime, \lambda b^\prime + \lambda^\prime b + b \cdot_e b^\prime).
    \]
\end{lemma}
\begin{proof}
    Direct verification.
\end{proof}

\begin{propdef}
    \label{Delta}
    Let $(\B, \bM, \A, \rho, e, \boldsymbol{r}, \boldsymbol{c})$ be an object of $\K{\text-}\BFS$ and consider the evaluation map $\mathrm{End}_{\A{\text-}\rmod}(\A) \to \A$, $u \mapsto u(1_\A)$.
    Then, the map $\Delta : \K \oplus (\B, \cdot_e) \to \A$ given by $1 \mapsto 1_\A$ and for $b \in \B$,
    \[
        b \mapsto \boldsymbol{r} \circ \rho(b) \circ \boldsymbol{c} \, (1_\A),
    \]
    is a morphism of $\K{\text-}\alg$.
\end{propdef}
\begin{proof}
    Since the evaluation map $\mathrm{End}_{\A{\text-}\rmod}(\A) \to \A$, $u \mapsto u(1_\A)$ is a morphism of $\K{\text-}\alg$ (actually an isomorphism), it suffices to show that $\widetilde{\rho} : \K \oplus (\B, \cdot_e) \to \mathrm{End}_{\A{\text-}\rmod}(\A)$ given by $1 \mapsto 1_\A$ and for $b \in \B$,
    \[
        b \mapsto \boldsymbol{r} \circ \rho(b) \circ \boldsymbol{c},
    \]
    is a morphism of $\K{\text-}\alg$. Since, by definition $\widetilde{\rho}(1) = 1_\A$, it remains to show that $\widetilde{\rho}(b \cdot_e b^\prime) = \widetilde{\rho}(b) \circ \widetilde{\rho}(b^\prime)$ for $b, b^\prime \in \B$. Indeed, we have
    \begin{align*}
        \widetilde{\rho}(b \cdot_e b^\prime) & = \widetilde{\rho}(b e b^\prime) = \boldsymbol{r} \circ \rho(b e b^\prime) \circ \boldsymbol{c} = \boldsymbol{r} \circ \rho(b) \circ \rho(e) \circ \rho(b^\prime) \circ \boldsymbol{c} \\
        & = \boldsymbol{r} \circ \rho(b) \circ \boldsymbol{c} \circ \boldsymbol{r} \circ \rho(b^\prime) \circ \boldsymbol{c} = \widetilde{\rho}(b) \circ \widetilde{\rho}(b^\prime),   
    \end{align*}
    where the third equality comes from the fact that $\rho : \B \to \mathrm{End}_{\A{\text-}\rmod}(\bM)$ is an algebra morphism and the fourth one from the identity $\rho(e) = \boldsymbol{c} \circ \boldsymbol{r}$.
\end{proof}

\begin{corollary}
    The assignment
    \begin{equation}
        \label{func:BFS_MorAlg}
        (\B, \bM, \A, \rho, e, \boldsymbol{r}, \boldsymbol{c}) \mapsto (\Delta : \K \oplus (\B, \cdot_e) \to \A)
    \end{equation}
    defines a functor $\K{\text-}\BFS \to \mathsf{Mor}(\K{\text-}\alg)$.
\end{corollary}
\begin{proof}
    Immediate verification.
\end{proof}

\begin{remark}
    Considering the functors $F, G : \K{\text-}\BFS \to \K{\text-}\alg$ given by $(\B, \bM, \A, \rho, e, \boldsymbol{r}, \boldsymbol{c}) \mapsto \K \oplus (\B, \cdot_e)$ and $(\B, \bM, \A, \rho, e, \boldsymbol{r}, \boldsymbol{c}) \mapsto \A$ respectively, the assignment 
    \[
        (\B, \bM, \A, \rho, e, \boldsymbol{r}, \boldsymbol{c}) \mapsto (\Delta : \K \oplus (\B, \cdot_e) \to \A)
    \]
    is a natural transformation from $F$ to $G$. 
\end{remark}

\subsubsection{The functor \texorpdfstring{$\K{\text-}\BFS_\fil \to \mathsf{Mor}(\K{\text-}\alg_\fil)$}{BFSfil->Mor(k-algfil)}} \label{BFSfil_Moralgfil}
\begin{lemma}
    Let $\B$ be an object of $\K{\text-}\alg_\fil$ and $e \in \F^1\B$. Then $\K \oplus (\B, \cdot_e)$ is an object of $\K{\text-}\alg_\fil$ with an algebra filtration given by
    \[
        \F^0(\K \oplus (\B, \cdot_e)) = \K \oplus (\B, \cdot_e) \text{ and } \F^n (\K \oplus (\B, \cdot_e)) = \F^{n-1}\B \text{ for } n \geq 1
    \]
\end{lemma}
\begin{proof}
    Follows from the fact that $e \in \F^1\B$.
\end{proof}

\begin{corollary} \label{cor:Deltafil}
    If $(\B, \bM, \A, \rho, e, \boldsymbol{r}, \boldsymbol{c})$ is an object of $\K{\text-}\BFS_\fil$, then the map $\Delta : \K \oplus (\B, \cdot_e) \to \A$ defined in Proposition-Definition \ref{Delta} is a morphism of $\K{\text-}\alg_\fil$. Moreover, the assignment
    \begin{equation}
        \label{func:BFS_MorAlgfil}
        (\B, \bM, \A, \rho, e, \boldsymbol{r}, \boldsymbol{c}) \mapsto (\Delta : \K \oplus (\B, \cdot_e) \to \A)
    \end{equation}
    is a functor $\K{\text-}\BFS_\fil \to \mathsf{Mor}(\K{\text-}\alg_\fil)$.
\end{corollary}
\begin{proof}
    Follows from the fact that $\rho$ is compatible with filtrations, from $\boldsymbol{r} \in \F^0\Mor_{\A{\text-}\rmod}(\bM, \A)$ and $\boldsymbol{c} \in \F^1\Mor_{\A{\text-}\rmod}(\A, \bM)$ and from the compatibility of the composition with filtrations.
\end{proof}

\subsubsection{The functor \texorpdfstring{$\K{\text-}\BFS_\gr \to \mathsf{Mor}(\K{\text-}\alg_\gr)$}{BFSgr->Mor(k-alggr)}} \label{BFSgr_Moralggr}
\begin{lemma}
    Let $\B$ be an object of $\K{\text-}\alg_\gr$ and $e \in \B_1$. Then $\K \oplus (\B, \cdot_e)$ is an object of $\K{\text-}\alg_\gr$ with an algebra grading given by
    \[
        (\K \oplus (\B, \cdot_e))_0 = \K \text{ and } (\K \oplus (\B, \cdot_e))_n = \B_{n-1} \text{ for } n \geq 1
    \]
\end{lemma}
\begin{proof}
    Follows from the fact that $e \in \B_1$.
\end{proof}

\begin{corollary} \label{cor:Deltagr}
    If $(\B, \bM, \A, \rho, e, \boldsymbol{r}, \boldsymbol{c})$ is an object of $\K{\text-}\BFS_\gr$, then the map $\Delta : \K \oplus (\B, \cdot_e) \to \A$ defined in Proposition-Definition \ref{Delta} is a morphism of $\K{\text-}\alg_\gr$. Moreover, the assignment
    \begin{equation*} 
        (\B, \bM, \A, \rho, e, \boldsymbol{r}, \boldsymbol{c}) \mapsto (\Delta : \K \oplus (\B, \cdot_e) \to \A)
    \end{equation*}
    is a functor $\K{\text-}\BFS_\gr \to \mathsf{Mor}(\K{\text-}\alg_\gr)$.
\end{corollary}
\begin{proof}
    Follows from the fact that $\rho$ is compatible with gradings, from $\boldsymbol{r} \in \Mor_{\A{\text-}\rmod}(\bM, \A)_0$ and $\boldsymbol{c} \in \Mor_{\A{\text-}\rmod}(\A, \bM)_1$, and from the compatibility of the composition with gradings.
\end{proof}
    \section{The objects \texorpdfstring{$\mathcal{O}^\Betti_\mat$}{OBmat}, \texorpdfstring{$\mathcal{O}^\Betti_{\mat, \fil}$}{OBmatfil} and \texorpdfstring{$\mathcal{O}^\DeRham_\mat$}{ODRmat} and the harmonic coproducts} \label{sec:OBmat_OBmatfil_ODRmat}
In this section, we construct explicit bimodules with factorization structures associated to the Betti and de Rham realizations of the double shuffle theory developed in \cite{EF1} (Corollary \ref{cor:tuple_OBmat}). Building on the functor $\K{\text-}\BFS \to \mathsf{Mor}(\K{\text-}\alg)$ introduced in the previous section, we then define algebra morphisms \eqref{eq:Delta_OB} and \eqref{eq:Delta_ODR} that we demonstrate to correspond precisely to the harmonic coproducts described in \cite{EF1} (Theorems \ref{thm:DeltaWB_DeltaOB} and \ref{thm:DeltaWDR_DeltaODR}).

\subsection{The object \texorpdfstring{$\mathcal{O}^\Betti_\mat$}{OBmat} of \texorpdfstring{$\K{\text-}\BFS$}{k-BFS}}
Let $F_2$ be the free group with generators $X_0$ and $X_1$ and let $\V^\Betti := \K{F_2}$ be its group algebra. For $i \in \{0,1\}$, we will abusively denote
\[
    X_i := X_i \otimes 1 \in \V^\Betti \otimes \V^\Betti \text{ and } Y_i := 1 \otimes X_i \in \V^\Betti \otimes \V^\Betti. 
\]
Recall from \cite[Sec. 7.2.3]{EF1} the algebra morphism $\underline{\rho} : \V^\Betti \to \M_3(\V^\Betti \otimes \V^\Betti)$ given by
\[
    \mbox{\small$\underline{\rho}(X_0) = \begin{pmatrix} X_0 & 0 & 0 \\ 0 & (1 - X_1) X_0 + Y_1^{-1} Y_0 Y_1 & (1 - X_1) X_0 X_1 \\ 0 & X_0 - Y_1^{-1} Y_0 Y_1 X_1^{-1} & X_0 X_1 \end{pmatrix} \text{ and } \underline{\rho}(X_1) = \begin{pmatrix} (X_1 - 1) Y_1 + 1 & Y_1 (1 - Y_1) & 0 \\ 1 -X_1 & Y_1 & 0 \\ 0 & 0 & 1 \end{pmatrix}$}
\]

\begin{propdef} \label{propdef:Betti_bimod}
    The tuple $(\V^\Betti, (\V^\Betti \otimes \V^\Betti)^{\oplus 3}, \V^\Betti \otimes \V^\Betti, \mathsf{r}\underline{\rho})$ is a $\K$-bimodule where $\mathsf{r}\underline{\rho} : \V^\Betti \to \M_3(\V^\Betti \otimes \V^\Betti)$ is the $\K$-algebra morphism given by the composition
    \[
        \mathsf{r}\underline{\rho} := \Ad_{\diag(Y_1, X_1, (X_0X_1)^{-1} Y_1^{-1} Y_0 Y_1)^{-1}} \circ \M_3(\mathrm{op}_{F_2^2}) \circ {^t}(-) \circ \underline{\rho} \circ \mathrm{op}_{F_2}.
    \]
\end{propdef}
\begin{proof}
    This follows from the algebra morphism status of $\underline{\rho}$ and from the $\K$-algebra isomorphism $\M_3(\V^\Betti \otimes \V^\Betti) \simeq \mathrm{End}_{(\V^\Betti \otimes \V^\Betti){\text-}\rmod}\left((\V^\Betti \otimes \V^\Betti)^{\oplus 3}\right)$. 
\end{proof}

\begin{definition} \label{def:rurow_rucol}
    Set
    \begin{equation*}
        \mathsf{r}\underline{\row} := \begin{pmatrix} 1 & -X_1 & 0\end{pmatrix} \ Y_1 \in \M_{1,3}(\V^\Betti \otimes \V^\Betti) \text{ and } \mathsf{r}\underline{\col} := Y_1^{-1} \begin{pmatrix}X_1 - 1 \\ 1 - Y_1 \\ 0\end{pmatrix} \in \M_{3,1}(\V^\Betti \otimes \V^\Betti).
    \end{equation*}
\end{definition}

\begin{proposition}\label{prop:rurho_rucolcircrurow}
    We have (equality in $\M_3(\V^\Betti \otimes \V^\Betti)$)
    \[
        \mathsf{r}\underline{\rho}(X_1 - 1) = \mathsf{r}\underline{\col} \cdot \mathsf{r}\underline{\row}.
    \]
\end{proposition}
\begin{proof}
    By definition we have
    \begin{align*}
        \mathsf{r}\underline{\rho}(X_1 - 1) = \diag(Y_1, X_1, (X_0X_1)^{-1} Y_1^{-1} Y_0 Y_1)^{-1} \ \M_3(\mathrm{op}_{F_2^2}) & \left({^t}(\underline{\rho}(\mathrm{op}_{F_2}(X_1 - 1)))\right) \\
        & \diag(Y_1, X_1, (X_0X_1)^{-1} Y_1^{-1} Y_0 Y_1).
    \end{align*}
    On the other hand, from \cite[page 46]{EF1}, we have
    \begin{align*}
        \underline{\rho}(\mathrm{op}_{F_2}(X_1 - 1)) & = \underline{\rho}(X_1^{-1} - 1) = -\begin{pmatrix} Y_1 \\ -1 \\ 0\end{pmatrix} (X_1 Y_1)^{-1} \begin{pmatrix} X_1 - 1 & 1- Y_1 & 0 \end{pmatrix} \\ 
        & = \begin{pmatrix} -X_1^{-1} \\ (X_1 Y_1)^{-1} \\ 0\end{pmatrix} \begin{pmatrix} X_1 - 1 & 1 - Y_1 & 0 \end{pmatrix} \\. 
    \end{align*}
    Moreover, applying identity \eqref{tMf_identity} for $(t,n,s) = (3,1,3)$, $E = \mathbf kF_2^2$ and $f = \mathrm{op}_{F_2^2}$, it follows that
    \begin{align*}
        \M_3(\mathrm{op}_{F_2^2})\left({}^t(\underline{\rho}(\mathrm{op}_{F_2}(X_1 - 1)))\right) & = 
        \M_{1,3}(\mathrm{op}_{F_2^2})\left({}\prescript{t}{}{\begin{pmatrix} X_1 - 1 & 1- Y_1 & 0 \end{pmatrix}}\right) \ 
        \M_{3,1}(\mathrm{op}_{F_2^2})\left({}\prescript{t}{}{\begin{pmatrix} -X_1^{-1} \\ (X_1 Y_1)^{-1} \\ 0\end{pmatrix}}\right) \\
        & = \begin{pmatrix} X_1^{-1} - 1 \\ 1- Y_1^{-1} \\ 0 \end{pmatrix} \ \begin{pmatrix} -X_1 & X_1 Y_1 & 0\end{pmatrix} 
    \end{align*}
    Then
    \begin{align*}
        & \begin{aligned}\mathsf{r}\underline{\rho}(X_1 - 1) = \diag(Y_1, X_1, (X_0X_1)^{-1} Y_1^{-1} Y_0 Y_1)^{-1} \ \begin{pmatrix} X_1^{-1} - 1 \\ 1- Y_1^{-1} \\ 0 \end{pmatrix} & \begin{pmatrix} -X_1 & X_1 Y_1 & 0\end{pmatrix} \\ & \diag(Y_1, X_1, (X_0X_1)^{-1} Y_1^{-1} Y_0 Y_1) \end{aligned} \\
        & = \begin{pmatrix} (X_1 Y_1)^{-1} - Y_1^{-1} \\ X_1^{-1} - (X_1 Y_1)^{-1} \\ 0 \end{pmatrix} \ \begin{pmatrix} -X_1 Y_1 & X_1 Y_1 X_1 & 0\end{pmatrix} = \begin{pmatrix} X_1 Y_1^{-1} - Y_1^{-1} \\ Y_1^{-1} - 1 \\ 0 \end{pmatrix} \ \begin{pmatrix} Y_1 & - X_1 Y_1 & 0 \end{pmatrix},
    \end{align*}
    which is the announced result.
\end{proof}

\begin{corollary} \label{cor:tuple_OBmat}
    The tuple
    \begin{equation} \label{tuple_OBmat}
        \mathcal{O}^\Betti_\mat := (\V^\Betti, (\V^\Betti \otimes \V^\Betti)^{\oplus 3}, \V^\Betti \otimes \V^\Betti, \mathsf{r}\underline{\rho}, X_1 - 1, \mathsf{r}\underline{\row}, \mathsf{r}\underline{\col}) 
    \end{equation}
    is an object of $\K{\text-}\BFS$, where $\mathsf{r}\underline{\row}$ (resp. $\mathsf{r}\underline{\col}$) is identified with its corresponding right \linebreak $(\V^\Betti \otimes \V^\Betti)$-module morphism $(\V^\Betti \otimes \V^\Betti)^{\oplus 3} \to \V^\Betti \otimes \V^\Betti$ (resp. $\V^\Betti \otimes \V^\Betti \to (\V^\Betti \otimes \V^\Betti)^{\oplus 3}$).
\end{corollary}
\begin{proof}
    This follows from Proposition-Definition \ref{propdef:Betti_bimod} and Proposition \ref{prop:rurho_rucolcircrurow}.
\end{proof}

\subsection{The image of \texorpdfstring{$\mathcal{O}^\Betti_\mat$}{OBmat} in \texorpdfstring{$\mathsf{Mor}(\K{\text-}\alg)$}{Mor(k-alg)} and the coproduct \texorpdfstring{$\Delta^{\W,\Betti}$}{DeltaWB}}
Applying the functor $\K{\text-}\BFS \to \mathsf{Mor}(\K{\text-}\alg)$ given in \eqref{func:BFS_MorAlg} to the object $\mathcal{O}^\Betti_\mat$ given in \eqref{tuple_OBmat}, one defines the algebra morphism
\begin{equation} \label{eq:Delta_OB}
    \Delta_{\mathcal{O}^\Betti_\mat} : \K \oplus (\V^\Betti, \cdot_{(X_1-1)}) \to \V^\Betti \otimes \V^\Betti.
\end{equation}
Explicitly, for $b \in \V^\Betti$ we have
\[
    \Delta_{\mathcal{O}^\Betti_\mat}(b) = \mathsf{r}\underline{\row} \cdot \mathsf{r}\underline{\rho}(b) \cdot \mathsf{r}\underline{\col},
\]
On the other hand, recall from \cite[Sec. 2.1]{EF1} the subalgebra $\W^\Betti$ of $\V^\Betti$ given by
\[
    \W^\Betti := \K \oplus \V^\Betti (X_1-1).
\]
There is an algebra morphism $\K \oplus (\V^\Betti, \cdot_{(X_1-1)}) \to \W^\Betti$ given by $v \mapsto v \cdot (X_1-1)$. It is obviously surjective, and it is injective since right multiplication by $X_1-1$ is an injective endomorphism of $\V^\Betti$. 
On the other hand, it follows from \cite[Proposition 2.3]{EF1} that the algebra $\W^\Betti$ is generated by 
\begin{equation}
    \label{eq:gen_of_WB}
    X_1^{-1} \text{ and } X_0^n(X_1-1) \text{ for } n \in \Z. 
\end{equation}

\noindent An algebra morphism $\Delta^{\W, \Betti} : \W^{\Betti} \to \W^{\Betti} \otimes \W^{\Betti}$ is given by (see \cite[Lemma 2.11]{EF1})
\[
    \Delta^{\W, \Betti}(X_1^{-1}) = X_1^{-1} \otimes X_1^{-1},
\]
and for $n \in \Z$,
\[
    \Delta^{\W, \Betti}(X_0^n(X_1-1)) = X_0^n(X_1-1) \otimes 1 + 1 \otimes X_0^n(X_1-1) - \sum_{k=1}^{n-1} X_0^k(X_1-1) \otimes X_0^{n-k}(X_1-1),
\]
using the convention that for a map $f$ from $\Z$ to an abelian group and $p, q \in \Z$,
\begin{equation}
    \label{convention_sum}
    \sum_{k=p}^q f(k) :=
    \begin{cases}
        f(p) + \cdots + f(q) & \text{if } q > p - 1 \\
        0 & \text{if } q = p - 1 \\
        -f(p-1) - \cdots - f(q+1) & \text{if } q < p - 1
    \end{cases}
\end{equation}

\begin{theorem} \label{thm:DeltaWB_DeltaOB}
    The following diagram
    \begin{equation}\label{diag:DeltaWB_DeltaOB}\begin{tikzcd}
        \W^\Betti \ar[rrr, "\Delta^{\W, \Betti}"] \ar[d, "\simeq"'] &&& \W^\Betti \otimes \W^\Betti \ar[d, hook'] \\
        \K \oplus (\V^\Betti, \cdot_{X_1 - 1}) \ar[rrr, "\Delta_{\mathcal{O}^\Betti_\mat}"] &&& \V^\Betti \otimes \V^\Betti 
    \end{tikzcd}\end{equation}
    commutes.
\end{theorem}
\begin{proof}
    Since all arrows of diagram \eqref{diag:DeltaWB_DeltaOB} are algebra morphisms, it suffices to establish the commutativity through evaluation on a system of generators of $\W^\Betti$, which we take to be the ones given in \eqref{eq:gen_of_WB}. \newline
    The image of $X_1^{-1}$ by the composition $\W^\Betti \to (\W^\Betti \otimes \W^\Betti) \hookrightarrow (\V^\Betti \otimes \V^\Betti)$ is given by
    \[
        X_1^{-1} \mapsto X_1^{-1} \otimes X_1^{-1}.
    \]
    On the other hand, recall that $X_1^{-1} = - X_1^{-1} (X_1 - 1) + 1$. The image of $X_1^{-1}$ by the composition $\W^\Betti \simeq \left(\K \oplus (\V^\Betti, \cdot_{X_1-1})\right) \hookrightarrow (\V^\Betti \otimes \V^\Betti)$ is then given by
    \[
        X_1^{-1} \mapsto - \mathsf{r}\underline{\row} \cdot \mathsf{r}\underline{\rho}(X_1^{-1}) \cdot \mathsf{r}\underline{\col} + 1.
    \]
    We have
    \begin{align*}
        \mathsf{r}\underline{\rho}(X_1^{-1}) & = \Ad_{\diag(Y_1, X_1, (X_0X_1)^{-1} Y_1^{-1} Y_0 Y_1)^{-1}} \circ \M_3(\mathrm{op}_{F_2^2}) \left({^t}\underline{\rho}(X_1) \right) \\
        & = \Ad_{\diag(Y_1, X_1, (X_0X_1)^{-1} Y_1^{-1} Y_0 Y_1)^{-1}} \circ \M_3(\mathrm{op}_{F_2^2}) \left(\begin{pmatrix} Y_1 X_1 - Y_1 + 1 & -X_1 + 1 & 0 \\ Y_1 - Y_1^2 & Y_1 & 0 \\ 0 & 0 & 1\end{pmatrix} \right) \\
        & = \Ad_{\diag(Y_1, X_1, (X_0X_1)^{-1} Y_1^{-1} Y_0 Y_1)^{-1}} \left(\begin{pmatrix} (Y_1 X_1)^{-1} - Y_1^{-1} + 1 & -X_1^{-1} + 1 & 0 \\ Y_1^{-1} - Y_1^{-2} & Y_1^{-1} & 0 \\ 0 & 0 & 1\end{pmatrix} \right) \\
        & = \begin{pmatrix} (Y_1 X_1)^{-1} - Y_1^{-1} + 1 & -Y_1^{-1} + Y_1^{-1} X_1 & 0 \\ X_1^{-1} - X_1^{-1} Y_1^{-1} & Y_1^{-1} & 0 \\ 0 & 0 & 1\end{pmatrix},
    \end{align*}
    where the second equality comes from \cite[Lemma 7.11]{EF1}. Therefore,
    \begin{align*}
        - \mathsf{r}\underline{\row} \cdot \mathsf{r}\underline{\rho}(X_1^{-1}) \cdot \mathsf{r}\underline{\col} + 1 & = - 1 + (X_1 Y_1)^{-1} + 1 = X_1^{-1} \otimes X_1^{-1},  
    \end{align*}
    thus proving the equality of the images of $X_1^{-1}$. \newline
    Next, for $n \in \Z$, the image of $X_0^n (X_1 - 1)$ by the composition $\W^\Betti \to (\W^\Betti \otimes \W^\Betti) \hookrightarrow (\V^\Betti \otimes \V^\Betti)$ is given by
    \[
        X_0^n (X_1 - 1) \mapsto X_0^n(X_1-1) \otimes 1 + 1 \otimes X_0^n (X_1-1) - \sum_{k=1}^{n-1} X_0^k(X_1-1) \otimes X_0^{n-k} (X_1-1).
    \]
    On the other hand, the image of $X_0^n (X_1 - 1)$ by the composition $\W^\Betti \simeq \left(\K \oplus (\V^\Betti, \cdot_{X_1-1})\right) \hookrightarrow (\V^\Betti \otimes \V^\Betti)$ is given by 
    \[
        X_0^n (X_1 - 1) \mapsto \Delta_{\mathcal{O}^\Betti_\mat}(X_0^n) = \mathsf{r}\underline{\row} \cdot \mathsf{r}\underline{\rho}(X_0^n) \cdot \mathsf{r}\underline{\col}.
    \]
    We have
    \[
        \mathsf{r}\underline{\rho}(X_0^n) = \mbox{\small$\diag(Y_1, X_1, (X_0X_1)^{-1} Y_1^{-1} Y_0 Y_1)$}^{-1} \cdot \M_3(\mathrm{op}_{F_2^2}) \left({^t}\underline{\rho}(X_0^{-n})\right) \cdot \mbox{\small$\diag(Y_1, X_1, (X_0X_1)^{-1} Y_1^{-1} Y_0 Y_1)$}.
    \]
    Notice that
    \[
        \mathsf{r}\underline{\row} \cdot \mbox{\small$\diag(Y_1, X_1, (X_0X_1)^{-1} Y_1^{-1} Y_0 Y_1)$}^{-1} = \begin{pmatrix} 1 & -Y_1 & 0 \end{pmatrix},
    \]
    and
    \[
        \mbox{\small$\diag(Y_1, X_1, (X_0X_1)^{-1} Y_1^{-1} Y_0 Y_1)$}^{-1} \cdot \mathsf{r}\underline{\col} = \begin{pmatrix} X_1 - 1 \\ X_1 Y_1^{-1} - X_1 \\ 0 \end{pmatrix}. 
    \]
    Then
    \begin{align*}
        & \Delta_{\mathcal{O}^\Betti_\mat}(X_0^n) = \begin{pmatrix} 1 & -Y_1 & 0 \end{pmatrix} \cdot \M_3(\mathrm{op}_{F_2^2}) \circ {}^t(-)\left(\underline{\rho}(X_0^{-n})\right) \cdot \begin{pmatrix} X_1 - 1 \\ X_1 Y_1^{-1} - X_1 \\ 0 \end{pmatrix} \\
        & \begin{aligned} = \M_{1,3}(\mathrm{op}_{F_2^2}) \circ {}^t(-) \left(\begin{pmatrix} 1 \\ -Y_1^{-1} \\ 0 \end{pmatrix}\right) & \cdot \M_3(\mathrm{op}_{F_2^2}) \circ {}^t(-)\left(\underline{\rho}(X_0^{-n})\right) \\
        & \cdot \M_{3,1}(\mathrm{op}_{F_2^2}) \circ {}^t(-) \left(\begin{pmatrix} X_1^{-1} - 1 & X_1^{-1} Y_1 - X_1^{-1} & 0 \end{pmatrix}\right)\end{aligned} \\
        & = \mathrm{op}_{F_2^2}\left(-X_1^{-1} \begin{pmatrix} -1 + X_1 & -Y_1 + 1 & 0 \end{pmatrix} \underline{\rho}(X_0^{-n}) \begin{pmatrix} Y_1 \\ -1 \\ 0 \end{pmatrix} Y_1^{-1}\right) \\
        & \begin{aligned} = \mathrm{op}_{F_2^2}\Bigg(- X_1^{-1} \otimes 1 \Big((X_1 - 1) X_0^{-n} \otimes 1 + & 1 \otimes (1 - X_1^{-1}) X_0^{-n} X_1 \\
        & - \sum_{k=1}^{-n-1} (X_1 - 1) X_0^k \otimes (1 - X_1^{-1}) X_0^{-n-k} X_1\Big) 1 \otimes X_1^{-1}\Bigg)\end{aligned} \\
        & = X_0^n (X_1 - 1) \otimes X_1 + X_1 \otimes X_0^n (X_1 - 1) + \sum_{k=1}^{-n-1} X_0^{-k} (X_1 - 1) \otimes X_0^{n+k} (X_1 - 1) \\
        & = X_0^n (X_1 - 1) \otimes X_1 + X_1 \otimes X_0^n (X_1 - 1) - \sum_{k=0}^{n} X_0^{k} (X_1 - 1) \otimes X_0^{n-k} (X_1 - 1) \\
        & = X_0^n (X_1 - 1) \otimes X_1 + X_1 \otimes X_0^n (X_1 - 1) - (X_1 - 1) \otimes X_0^n (X_1 - 1) - X_0^n (X_1 - 1) \otimes (X_1 - 1) \\
        & - \sum_{k=1}^{n-1} X_0^k (X_1 - 1) \otimes X_0^{n-k} (X_1 - 1) \\
        & = X_0^n (X_1 - 1) \otimes 1 + 1 \otimes X_0^n (X_1 - 1) - \sum_{k=1}^{n-1} X_0^{k} (X_1 - 1) \otimes X_0^{n-k} (X_1 - 1) 
    \end{align*}
    where the third equality follows by applying identities \eqref{tMM_MtM} and \eqref{tMf_identity}, the fourth one from \cite[Lemma 7.12]{EF1}, and the sixth one by using the identity (coming from \eqref{convention_sum})
    \[
        \sum_{k=1}^{-n-1}f(k)=-\sum_{k=0}^n f(-k),
    \]
    for any map $f$ from $\Z$ to an abelian group. This concludes the proof of the equality of the images of $X_0^n (X_1 - 1)$.
\end{proof}

\begin{remark}
    Under the identification $\W^\Betti \simeq \K \oplus (\V^\Betti, \cdot_{X_1 - 1})$, the commutativity of diagram \eqref{diag:DeltaWB_DeltaOB} enables us to obtain that the image of $\Delta_{\mathcal{O}^\Betti_\mat}$ lies in $\W^\Betti \otimes \W^\Betti$ as well as the identity
    \[
        \Delta_{\mathcal{O}^\Betti_\mat} = \Delta^{\W, \Betti}.
    \]
\end{remark}

\subsection{The object \texorpdfstring{$\mathcal{O}^\Betti_{\mat, \fil}$}{OBmatfil} of \texorpdfstring{$\K{\text-}\BFS_\fil$}{k-BFSfil} and its image in \texorpdfstring{$\mathsf{Mor}(\K{\text-}\alg_\fil)$}{Mor(k-algfil)}}
Recall from Proposition-Definition \ref{propdef:Betti_bimod}, Definition \ref{def:rurow_rucol} and \eqref{tuple_OBmat} the object $\mathcal{O}^\Betti_\mat$ of $\K{\text-}\BFS$ given by
\begin{equation*}
    \mathcal{O}^\Betti_\mat := (\V^\Betti, (\V^\Betti \otimes \V^\Betti)^{\oplus 3}, \V^\Betti \otimes \V^\Betti, \mathsf{r}\underline{\rho}, X_1 - 1, \mathsf{r}\underline{\row}, \mathsf{r}\underline{\col}) 
\end{equation*}
\noindent The group algebra $\V^\Betti = \K{F_2}$ is naturally equipped with a filtration
\begin{equation}
    \label{eq:filVB}
    \F^0\K{F_2} = \K{F_2} \text{ and for } n \geq 1, \ \F^n\K{F_2} = I_{F_2}^n,
\end{equation}
where $I_{F_2}$ denotes the augmentation ideal of the group algebra $\K{F_2}$, which is the $\K$-submodule of $\K{F_2}$ generated by the elements $g - 1$, where $g \in F_2$ and is also the right (or left) $\K{F_2}$-submodule of $\K{F_2}$ generated by $X_0 - 1$ and $X_1 - 1$.
The pair $(\V^\Betti, (\F^n\V^\Betti)_{n \in \Z})$ is an object of $\K{\text-}\alg_\fil$.

\begin{lemma}
    \label{lem:filVB2}
    A filtration $(\F^n(\V^\Betti \otimes \V^\Betti))_{n \in \Z}$ of the $\K$-algebra $\V^\Betti \otimes \V^\Betti$ is defined, for $n \geq 0$, by
    \[
        \F^n(\V^\Betti \otimes \V^\Betti) := \sum_{i+j=n} \F^i\V^\Betti \otimes \F^j\V^\Betti.
    \]
    The pair $(\V^\Betti \otimes \V^\Betti, (\F^n(\V^\Betti \otimes \V^\Betti))_{n \in \Z})$ is an object of $\K{\text-}\alg_\fil$.
\end{lemma}
\begin{proof}
    Immediate verification.
\end{proof}

\begin{lemma}
    \label{lem:filVB23}
    A filtration $(\F^n(\V^\Betti \otimes \V^\Betti)^{\oplus 3})_{n \in \Z}$ of the right $(\V^\Betti \otimes \V^\Betti)$-module $(\V^\Betti \otimes \V^\Betti)^{\oplus 3}$ is given by
    \[
        \F^n(\V^\Betti \otimes \V^\Betti)^{\oplus 3} := \left(\F^n(\V^\Betti \otimes \V^\Betti)\right)^{\oplus 3},
    \]
    for $n \geq 0$. The pair $\left((\V^\Betti \otimes \V^\Betti)^{\oplus 3}, (\F^n(\V^\Betti \otimes \V^\Betti)^{\oplus 3})_{n \in \Z}\right)$ is a filtered bimodule over the pair of filtered algebras $(\V^\Betti, \V^\Betti \otimes \V^\Betti)$. 
\end{lemma}
\begin{proof}
    The first statement can be proved by a direct verification. For the second statement, we have that
    \begin{align*}
        \mathsf{r}\underline{\rho}(X_0 - 1) & = \mbox{\small$\begin{pmatrix} X_0 - 1 & 0 & 0 \\ 0 & (X_1 Y_1)^{-1} Y_0 Y_1 - 1 & (X_1 Y_1)^{-1} (1 - X_1^{-1} X_0^{-1} Y_0) Y_0Y_1 \\ 0 & (Y_0 Y_1)^{-1} X_0 (1 - X_1) Y_0 Y_1 & X_0 - 1 + (1 - X_0 X_1^{-1} X_0^{-1}) Y_1^{-1} Y_0 Y_1 \end{pmatrix}$} \\
        & \in \M_3(\F^1(\V^\Betti \otimes \V^\Betti)) 
    \end{align*}
    and
    \[
        \mathsf{r}\underline{\rho}(X_1 - 1) = \begin{pmatrix} X_1 (X_1 - 1) Y_1^{-1} & X_1 (1 - X_1) & 0 \\ 1 - Y_1 & X_1 (Y_1 - 1) & 0 \\ 0 & 0 & 0 \end{pmatrix} \in \M_3(\F^1(\V^\Betti \otimes \V^\Betti)),
    \]
    which implies that $\mathsf{r}\underline{\rho}(I_{F_2}) \subset \M_3(\F^1(\V^\Betti \otimes \V^\Betti))$ and therefore for any $n \geq 1$, we obtain $\mathsf{r}\underline{\rho}(I_{F_2}^n) \subset \M_3(\F^n(\V^\Betti \otimes \V^\Betti))$.
\end{proof}

\begin{lemma}
    \label{lem:rucolinF1}
    We have \(\mathsf{r}\underline{\col} \in \F^1 (\V^\Betti \otimes \V^\Betti)^{\oplus 3}\).
\end{lemma}
\begin{proof}
    This follows by definition of $\mathsf{r}\underline{\col}$.
\end{proof}

\begin{corollary} \label{cor:tuple_OBmatfil}
    The filtrations given in \eqref{eq:filVB}, Lemmas \ref{lem:filVB2} and \ref{lem:filVB23} define a filtered structure on $\mathcal{O}^\Betti_\mat$, which defines an object $\mathcal{O}^\Betti_{\mat, \fil}$ of the category $\K{\text-}\BFS_\fil$. 
\end{corollary}
\begin{proof}
    Follows from Lemmas \ref{lem:filVB2}, \ref{lem:filVB23} and \ref{lem:rucolinF1}, where in the latter lemma, we identify $\mathsf{r}\underline{\col}$ with its corresponding right $(\V^\Betti \otimes \V^\Betti)$-module morphism $\V^\Betti \otimes \V^\Betti \to (\V^\Betti \otimes \V^\Betti)^{\oplus 3}$ and use the isomorphism of filtered $\K$-modules $\Mor_{(\V^\Betti \otimes \V^\Betti){\text-}\rmod}(\V^\Betti \otimes \V^\Betti, (\V^\Betti \otimes \V^\Betti)^{\oplus 3}) \simeq (\V^\Betti \otimes \V^\Betti)^{\oplus 3}$ given by applying Lemma \ref{lem:morAM_M_isofil} to $\A = \V^\Betti \otimes \V^\Betti$ and $\bM = (\V^\Betti \otimes \V^\Betti)^{\oplus 3}$.
\end{proof}

\begin{corollary} \label{cor:DeltaOBfil}
    The image of the object $\mathcal{O}^\Betti_{\mat, \fil}$ by the functor $\K{\text-}\BFS_\fil \to \mathsf{Mor}(\K{\text-}\alg_\fil)$ given in \eqref{func:BFS_MorAlgfil} is a filtered algebra morphism
    \begin{equation*} 
        \Delta_{\mathcal{O}^\Betti_{\mat, \fil}} : \K \oplus (\V^\Betti, \cdot_{(X_1-1)}) \to \V^\Betti \otimes \V^\Betti,
    \end{equation*}
    which, as an algebra morphism, is equal to $\Delta_{\mathcal{O}^\Betti_\mat} : \K \oplus (\V^\Betti, \cdot_{(X_1-1)}) \to \V^\Betti \otimes \V^\Betti$ given in \eqref{eq:Delta_OB}.
\end{corollary}
\begin{proof}
    The left hand side of \eqref{eq:func_diag} is a commutative diagram of functors. The result follows by evaluating the images of the object $\mathcal{O}^\Betti$ by the functors $\K{\text-}\BFS_\fil \to \K{\text-}\BFS \to \mathsf{Mor}(\K{\text-}\alg)$ and $\K{\text-}\BFS_\fil \to \mathsf{Mor}(\K{\text-}\alg) \to \mathsf{Mor}(\K{\text-}\alg_\fil)$.
\end{proof}

\subsection{The objects \texorpdfstring{$\mathcal{O}^\DeRham_\mat$}{ODRmat} and \texorpdfstring{$\gr(\mathcal{O}^\Betti_{\mat, \fil})$}{grOBmatfil} of \texorpdfstring{$\K{\text-}\BFS_\gr$}{k-BFSgr}} \label{sec:ODRmat_grOBmat}

Let $\mathfrak{f}_2$ be the free graded $\K$-Lie algebra with generators $e_0$ and $e_1$ of degree $1$ and let $\V^{\DeRham} := U(\mathfrak{f}_2)$ be its universal enveloping algebra, which is an object of $\K{\text-}\alg_\gr$ thanks to \cite[Sec. 2.1]{EF1}.
Set $e_\infty := - e_0 - e_1$. For $i \in \{0, 1, \infty\}$, we will abusively denote
\[
    e_i := e_i \otimes 1 \in \V^\DeRham \otimes \V^\DeRham \text{ and } f_i := 1 \otimes e_i \in \V^\DeRham \otimes \V^\DeRham. 
\]
Recall from \cite[Sec. 5.2.3]{EF1} the algebra morphism $\rho : \V^\DeRham \to \M_3(\V^\DeRham \otimes \V^\DeRham)$ given by
\[
    \rho(e_0) = \begin{pmatrix} e_0 & 0 & 0 \\ 0 & - e_1 + f_0 & - e_1 \\ 0 & - e_\infty - f_0 & - e_\infty \end{pmatrix} \text{ and } \rho(e_1) = \begin{pmatrix} e_1 & - f_1 & 0 \\ - e_1 & f_1 & 0 \\ 0 & 0 & 0 \end{pmatrix}.
\]
Moreover, the algebra $\V^\DeRham \otimes \V^\DeRham$ is equipped with the grading given by
\begin{equation*} 
    (\V^\DeRham \otimes \V^\DeRham)_n := \sum_{i+j=n} \V^\DeRham_i \otimes \V^\DeRham_j,
\end{equation*}
for any $n \geq 0$. Additionally, the right $(\V^\DeRham \otimes \V^\DeRham)$-module $(\V^\DeRham \otimes \V^\DeRham)^{\oplus 3}$ is equipped with the grading given by
\[
    ((\V^\DeRham \otimes \V^\DeRham)^{\oplus 3})_n := ((\V^\DeRham \otimes \V^\DeRham)_n)^{\oplus 3},
\]
for any $n \geq 0$, which we denote by $(\V^\DeRham \otimes \V^\DeRham)^{\oplus 3}_n$.

\begin{propdef} \label{propdef:DeRham_bimod}
    The tuple $(\V^\DeRham, (\V^\DeRham \otimes \V^\DeRham)^{\oplus 3}, \V^\DeRham \otimes \V^\DeRham, \mathsf{r}\rho)$ is a graded $\K$-bimodule where $\mathsf{r}\rho : \V^\DeRham \to \M_3(\V^\DeRham \otimes \V^\DeRham)$ is the $\K$-algebra morphism given by
    \[
        \mathsf{r}\rho := \M_3(\mathrm{S}_{\mathfrak{f}_2^{\oplus 2}}) \circ {^t}(-) \circ \rho \circ \mathrm{S}_{\mathfrak{f}_2}.
    \]
\end{propdef}
\begin{proof}
    The bimodule structure follows from the algebra morphism status of $\rho$ and from the $\K$-algebra isomorphism $\M_3(\V^\DeRham \otimes \V^\DeRham) \simeq \mathrm{End}_{(\V^\DeRham \otimes \V^\DeRham){\text-}\rmod}\left((\V^\DeRham \otimes \V^\DeRham)^{\oplus 3}\right)$. The graded status of the bimodule $(\V^\DeRham, (\V^\DeRham \otimes \V^\DeRham)^{\oplus 3}, \V^\DeRham \otimes \V^\DeRham, \mathsf{r}\rho)$ follows from the fact that
    \[
        \mathsf{r}\rho(e_0) = \begin{pmatrix} e_0 & 0 & 0 \\ 0 & - e_1 + f_0 & - e_\infty - f_0 \\ 0 & - e_1 & - e_\infty \end{pmatrix} \in \M_3((\V^\DeRham \otimes \V^\DeRham)_1)
    \]
    and
    \[
        \mathsf{r}\rho(e_1) = \begin{pmatrix} e_1 & - e_1 & 0 \\ - f_1 & f_1 & 0 \\ 0 & 0 & 0 \end{pmatrix} \in \M_3((\V^\DeRham \otimes \V^\DeRham)_1).
    \]
\end{proof}

\begin{definition} \label{def:rrow_rcol}
    Set
    \begin{equation*}
        \mathsf{r}\row := \begin{pmatrix} 1 & -1 & 0\end{pmatrix} \in \M_{1,3}(\V^\DeRham \otimes \V^\DeRham) \text{ and } \mathsf{r}\col := \begin{pmatrix}e_1 \\ -f_1 \\ 0\end{pmatrix} \in \M_{3,1}(\V^\DeRham \otimes \V^\DeRham).
    \end{equation*}
\end{definition}

\begin{lemma}\label{lem:rrho_rcolcircrrow}
    We have (equality in $\M_3(\V^\DeRham \otimes \V^\DeRham)$)
    \[
        \mathsf{r}\rho(e_1) = \mathsf{r}\col \cdot \mathsf{r}\row.
    \]
\end{lemma}
\begin{proof}
    Immediate verification.
\end{proof}

\begin{lemma}
    \label{lem:rcolin1}
    We have \(\mathsf{r}\col \in (\V^\Betti \otimes \V^\Betti)^{\oplus 3}_1\).
\end{lemma}
\begin{proof}
    This follows by definition of $\mathsf{r}\col$ given in Definition \ref{def:rrow_rcol}.
\end{proof}

\begin{corollary} \label{cor:tuple_ODRmat}
    The tuple
    \begin{equation} \label{tuple_ODRmat}
        \mathcal{O}^\DeRham_\mat := (\V^\DeRham, (\V^\DeRham \otimes \V^\DeRham)^{\oplus 3}, \V^\DeRham \otimes \V^\DeRham, \mathsf{r}\rho, e_1, \mathsf{r}\row, \mathsf{r}\col) 
    \end{equation}
    is an object of $\K{\text-}\BFS_\gr$, where $\mathsf{r}\row$ (resp. $\mathsf{r}\col$) is identified with its corresponding right \linebreak $(\V^\DeRham \otimes \V^\DeRham)$-module morphism $(\V^\DeRham \otimes \V^\DeRham)^{\oplus 3} \to \V^\DeRham \otimes \V^\DeRham$ (resp. $\V^\DeRham \otimes \V^\DeRham \to (\V^\DeRham \otimes \V^\DeRham)^{\oplus 3}$).
\end{corollary}
\begin{proof}
    This follows from Proposition-Definition \ref{propdef:DeRham_bimod} and Lemmas \ref{lem:rrho_rcolcircrrow} and \ref{lem:rcolin1}.
\end{proof}

\begin{proposition} \label{prop:iso_grOB_ODR}
    The tuple
    \[
        \gr(\mathcal{O}^\Betti_{\mat, \fil}) := (\gr \V^\Betti, \gr((\V^\Betti \otimes \V^\Betti)^{\oplus 3}), \gr(\V^\Betti \otimes \V^\Betti), \gr(\mathsf{r}\underline{\rho}), [X_1 - 1]_1, \gr_0(\mathsf{r}\underline{\row}), \gr_1(\mathsf{r}\underline{\col}))
    \]
    is an object of $\K{\text-}\BFS_\gr$. Moreover, the objects $\gr(\mathcal{O}^\Betti_{\mat, \fil})$ and $\mathcal{O}^\DeRham_\mat$ are isomorphic.
\end{proposition}
\begin{proof}
    The first statement follows by applying the functor $\K{\text-}\BFS_\fil \to \K{\text-}\BFS_\gr$ defined in \linebreak Sec. \ref{BFSfil_BFSgr}. For the second statement, recall from \cite[Sec. 2.4.1]{EF1} that there is a graded algebra isomorphism $\gr\V^\Betti \simeq \V^\DeRham$ given by $[X_i - 1]_1 \mapsto e_i$ ($i \in \{0, 1\}$). This induces a graded algebra isomorphism $\gr(\V^\Betti \otimes \V^\Betti) \simeq \V^\DeRham \otimes \V^\DeRham$, therefore a graded right module isomorphism $\gr((\V^\Betti \otimes \V^\Betti)^{\oplus 3}) \simeq (\V^\DeRham \otimes \V^\DeRham)^{\oplus 3}$ over the graded algebra isomorphism $\gr(\V^\Betti \otimes \V^\Betti) \simeq \V^\DeRham \otimes \V^\DeRham$. Finally, the following diagrams are commutative
    \[\begin{tikzcd}
        \gr_0(\V^\Betti \otimes \V^\Betti)^{\oplus 3} \ar[rr, "\gr_0(\mathsf{r}\underline{\row})"] \ar[d, "\simeq"'] && \gr_0(\V^\Betti) \ar[d, "\simeq"] \\
        (\V^\DeRham \otimes \V^\DeRham)^{\oplus 3} \ar[rr, "\mathsf{r}\row"] && \V^\DeRham
    \end{tikzcd}
    \hspace{10mm}
    \begin{tikzcd}
        \gr_1(\V^\Betti) \ar[rr, "\gr_1(\mathsf{r}\underline{\col})"] \ar[d, "\simeq"'] && \gr_1(\V^\Betti \otimes \V^\Betti)^{\oplus 3} \ar[d, "\simeq"] \\
        \V^\DeRham \ar[rr, "\mathsf{r}\col"] && (\V^\DeRham \otimes \V^\DeRham)^{\oplus 3}
    \end{tikzcd}\]
    and
    \[\begin{tikzcd}
        \gr(\V^\Betti) \ar[r, "\gr(\mathsf{r}\underline{\rho})"] \ar[d, "\simeq"'] & \M_3(\gr(\V^\Betti \otimes \V^\Betti)) \ar[d, "\simeq"] \\
        \V^\DeRham \ar[r, "\mathsf{r}\rho"] & \M_3(\V^\DeRham \otimes \V^\DeRham)
    \end{tikzcd}\]
    Indeed, for the first two diagrams, it suffices to prove that $\gr(\mathsf{r}\underline{\row}) = \mathsf{r}\row$ and $\gr(\mathsf{r}\underline{\col}) = \mathsf{r}\col$, under the isomorphism $\gr\V^\Betti \simeq \V^\DeRham$.
    For $g, h, k \in F_2^2$, we recall the following identities
    \begin{align}
        [g]_0 &= 1 \label{g0}\\
        [g(h-1)k]_1 & = [h-1]_1 \label{gh-1k} \\
        [g^{-1}-1]_1 & = -[g-1]_1 \label{g-1} \\
        [gh-1]_1 & = [g-1]_1 + [h-1]_1 \label{gh-1}
    \end{align}
    We have
    \begin{align*}
        \gr_0(\mathsf{r}\underline{\row}) & = \begin{pmatrix} [Y_1]_0 & -[X_1Y_1]_0 & 0 \end{pmatrix} \\
        & = \begin{pmatrix} 1 & -1 & 0\end{pmatrix} = \mathsf{r}\row,
    \end{align*}
    where the second equality follows from identity \eqref{g0} for applied to $g = Y_1$ and to $g = X_1 Y_1$. On the other hand,
    \begin{align*}
        \gr_1(\mathsf{r}\underline{\col}) = \begin{pmatrix} [Y_1^{-1}(X_1 - 1)]_1 \\ [Y_1^{-1} - 1]_1 \\ 0\end{pmatrix} = \begin{pmatrix} [X_1 - 1]_1 \\ -[Y_1 - 1]_1 \\ 0\end{pmatrix} = \begin{pmatrix}e_1 \\ -f_1 \\ 0\end{pmatrix} = \mathsf{r}\col,
    \end{align*}
    where the second equality follows from identity \eqref{gh-1k} applied to $(g, h, k)=(Y_1^{-1}, X_1, 1)$ and from identity \eqref{g-1} applied to $g = Y_1$. \newline
    For the last diagram, it suffices to prove that $\gr_1(\mathsf{r}\underline{\rho})(X_i - 1) = \mathsf{r}\rho(e_i)$ ($i \in \{0, 1\}$), under the isomorphism $\gr\V^\Betti \simeq \V^\DeRham$.
    We have
    \begin{align*}
        \gr_1(\mathsf{r}\underline{\rho})(X_1 - 1) & = \begin{pmatrix} [X_1 Y_1^{-1} (X_1 - 1)]_1 & [-X_1 (X_1 - 1)]_1 & 0 \\ [-(Y_1 - 1)]_1 & [X_1 (Y_1 - 1)]_1 & 0 \\ 0 & 0 & 0 \end{pmatrix} \\
        & = \begin{pmatrix} [X_1 - 1]_1 & -[X_1 - 1]_1 & 0 \\ -[Y_1 - 1]_1 & [Y_1 - 1]_1 & 0 \\ 0 & 0 & 0 \end{pmatrix} \\
        & = \begin{pmatrix} e_1 & - e_1 & 0 \\ -f_1 & f_1 & 0 \\ 0 & 0 & 0 \end{pmatrix} = \mathsf{r}\rho(e_i),
    \end{align*}
    where the second equality follows from identity \eqref{gh-1k} applied to $(g, h, k)=(X_1Y_1^{-1}, X_1, 1)$, $(g, h, k)=(X_1, X_1, 1)$ and to $(g, h, k)=(X_1, Y_1, 1)$.
    On the other hand, we have
    \begin{align*}
        \gr_1(\mathsf{r}\underline{\rho})(X_0 - 1) & = \begin{pmatrix} [X_0 - 1]_1 & 0 & 0 \\ 0 & [(X_1 Y_1)^{-1} Y_0 Y_1 - 1]_1 & [(X_1 Y_1)^{-1} (1 - X_1^{-1} X_0^{-1} Y_0) Y_0Y_1]_1 \\ 0 & [(Y_0 Y_1)^{-1} X_0 (1 - X_1) Y_0 Y_1]_1 & [X_0 - 1 + (1 - X_0 X_1^{-1} X_0^{-1}) Y_1^{-1} Y_0 Y_1]_1 \end{pmatrix} \\
        & = \begin{pmatrix} [X_0 - 1]_1 & 0 & 0 \\ 0 & - [X_1 - 1]_1 + [Y_0 - 1]_1 &  [X_0 - 1]_1 + [X_1 - 1]_1 - [Y_0 - 1]_1 \\ 0 & - [X_1 - 1]_1 & [X_0 - 1]_1 + [X_1 - 1]_1 \end{pmatrix} \\
        & = \begin{pmatrix} e_0 & 0 & 0 \\ 0 & - e_1 + f_0 & - e_\infty - f_0 \\ 0 & - e_1 & - e_\infty \end{pmatrix},
    \end{align*}
    where the second equality follows from these identities in $I_{F_2} / I_{F_2}^2$:
    \begin{align*}
        [(X_1 Y_1)^{-1} Y_0 Y_1 - 1]_1 = [Y_1^{-1}(X_1^{-1} Y_0 - 1)Y_1]_1 & \overset{\eqref{gh-1k}}{=} [X_1^{-1} Y_0 - 1]_1 \overset{\eqref{gh-1}}{=} [X_1^{-1} - 1]_1 + [Y_0 - 1]_1 \\
        & \overset{\eqref{g-1}}{=} - [X_1 - 1]_1 + [Y_0 - 1]_1,
    \end{align*}
    and
    \begin{align*}
        [(X_1 Y_1)^{-1} (1 - X_1^{-1} X_0^{-1} Y_0) Y_0Y_1]_1 & \overset{\eqref{gh-1k}}{=} [1 - X_1^{-1} X_0^{-1} Y_0)]_1 \overset{\eqref{gh-1}}{=} - [X_1^{-1} - 1]_1 - [X_0^{-1} - 1]_1 - [Y_0 - 1]_1 \\
        & \overset{\eqref{g-1}}{=} [X_0 - 1]_1 + [X_1 - 1]_1 - [Y_0 - 1]_1,
    \end{align*}
    and
    \begin{align*}
        [(Y_0 Y_1)^{-1} X_0 (1 - X_1) Y_0 Y_1]_1 & \overset{\eqref{gh-1k}}{=} - [X_1 - 1]_1,
    \end{align*}
    and
    \begin{align*}
        [X_0 - 1 + (1 - X_0 X_1^{-1} X_0^{-1}) Y_1^{-1} Y_0 Y_1]_1 & \overset{\eqref{gh-1k}}{=} [X_0 - 1]_1 - [X_0 X_1^{-1} X_0^{-1} - 1]_1 \\
        & \overset{\eqref{gh-1}}{=} [X_0 - 1]_1 - [X_0 - 1]_1 - [X_1^{-1} - 1]_1 - [X_0^{-1} - 1]_1 \\
        & \overset{\eqref{g-1}}{=} [X_0 - 1]_1 + [X_1 - 1]_1.
    \end{align*}
    This concludes the proof of the isomorphism claim between the two objects of $\K{\text-}\BFS_\gr$.
\end{proof}

\subsection{The image of \texorpdfstring{$\mathcal{O}^\DeRham_\mat$}{ODRmat} in \texorpdfstring{$\mathsf{Mor}(\K{\text-}\alg_\gr)$}{Mor(k-alg)} and the coproduct \texorpdfstring{$\Delta^{\W,\DeRham}$}{DeltaWDR}}

Applying the functor $\K{\text-}\BFS_\gr \to \mathsf{Mor}(\K{\text-}\alg_\gr)$ given in Sec. \ref{BFSgr_Moralggr} to the object $\mathcal{O}^\DeRham_\mat$ given in \eqref{tuple_ODRmat}, one defines the algebra morphism
\begin{equation} \label{eq:Delta_ODR}
    \Delta_{\mathcal{O}^\DeRham_\mat} : \K \oplus (\V^\DeRham, \cdot_{e_1}) \to \V^\DeRham \otimes \V^\DeRham.
\end{equation}
Explicitly, for $v \in \V^\DeRham$ we have
\[
    \Delta_{\mathcal{O}^\DeRham_\mat}(v) = \mathsf{r}\row \cdot \mathsf{r}\rho(v) \cdot \mathsf{r}\col,
\]
On the other hand, recall from \cite[Sec. 1.1]{EF1} the subalgebra $\W^\DeRham$ of $\V^\DeRham$ given by
\[
    \W^\DeRham := \K \oplus \V^\DeRham e_1.
\]
There is an algebra morphism $\K \oplus (\V^\DeRham, \cdot_{e_1}) \to \W^\DeRham$ given by $v \mapsto v \cdot e_1$. It is obviously surjective, and it is injective since right multiplication by $e_1$ is an injective endomorphism of $\V^\DeRham$.
On the other hand, it follows from \cite[Sec. 1.2]{EF1} that the algebra $\W^\DeRham$ is freely generated by 
\begin{equation}
    \label{eq:gen_of_WDR}
    e_0^n e_1 \text{ for } n \in \Z_{\geq 0}. 
\end{equation}

\noindent An algebra morphism $\Delta^{\W, \DeRham} : \W^{\DeRham} \to \W^{\DeRham} \otimes \W^{\DeRham}$ is given by (see \cite[(1.2.1)]{EF1})
\[
    \Delta^{\W, \DeRham}(e_0^n e_1) = e_0^n e_1 \otimes 1 + 1 \otimes e_0^n e_1 - \sum_{k=0}^{n-1} e_0^k e_1 \otimes e_0^{n-k-1} e_1,
\]
for $n \in \Z_{\geq 0}$.

\begin{theorem} \label{thm:DeltaWDR_DeltaODR}
    The following diagram
    \begin{equation}\label{diag:DeltaWDR_DeltaODR}\begin{tikzcd}
        \W^\DeRham \ar[rrr, "\Delta^{\W, \DeRham}"] \ar[d, "\simeq"'] &&& \W^\DeRham \otimes \W^\DeRham \ar[d, hook'] \\
        \K \oplus (\V^\DeRham, \cdot_{e_1}) \ar[rrr, "\Delta_{\mathcal{O}^\DeRham_\mat}"] &&& \V^\DeRham \otimes \V^\DeRham 
    \end{tikzcd}\end{equation}
    commutes.
\end{theorem}
\begin{proof}
    Since all arrows of diagram \eqref{diag:DeltaWDR_DeltaODR} are algebra morphisms, it suffices to establish the commutativity through evaluation on a system of generators of $\W^\DeRham$, which we take to be the ones given in \eqref{eq:gen_of_WDR}. \newline
    Let $n \in \Z_{\geq 0}$. The image of $e_0^n e_1$ by the composition $\W^\DeRham \to (\W^\DeRham \otimes \W^\DeRham) \hookrightarrow (\V^\DeRham \otimes \V^\DeRham)$ is given by
    \[
        e_0^n e_1 \mapsto e_0^n e_1 \otimes 1 + 1 \otimes e_0^n e_1 - \sum_{k=0}^{n-1} e_0^k e_1 \otimes e_0^{n-k-1} e_1.
    \]
    On the other hand, the image of $e_0^n e_1$ by the composition $\mbox{\small$\W^\DeRham \simeq \left(\K \oplus (\V^\DeRham, \cdot_{e_1})\right) \hookrightarrow (\V^\DeRham \otimes \V^\DeRham)$}$ is given by
    \[
        e_0^n e_1 \mapsto \mathsf{r}\row \cdot \mathsf{r}\rho(e_0^n) \cdot \mathsf{r}\col.
    \]
    We have
    \begin{align*}
        \mathsf{r}\row \cdot \mathsf{r}\rho(e_0^n) \cdot \mathsf{r}\col & = \prescript{t}{}{\Bigg(}{}^t\mathsf{r}\col \cdot \rho(e_0^n) \cdot {}^t\mathsf{r}\row \Bigg) \\
        & = e_0^n e_1 \otimes 1 + 1 \otimes e_0^n e_1 - \sum_{k=0}^{n-1} e_0^k e_1 \otimes e_0^{n-k-1} e_1,
    \end{align*}
    where the last equality comes from \cite[Lemma 5.7]{EF1}.
\end{proof}

\begin{remark}
    Under the identification $\W^\DeRham \simeq \K \oplus (\V^\DeRham, \cdot_{e_1})$, the commutativity of diagram \eqref{diag:DeltaWDR_DeltaODR} enables us to obtain that the image of $\Delta_{\mathcal{O}^\DeRham_\mat}$ lies in $\W^\DeRham \otimes \W^\DeRham$ as well as the identity
    \[
        \Delta_{\mathcal{O}^\DeRham_\mat} = \Delta^{\W, \DeRham}.
    \]
\end{remark}
\noindent Thanks to \cite[Proposition 2.8]{EF1}, the graded algebra isomorphism $\gr(\V^\Betti) \simeq \V^\DeRham$ induces a graded algebra isomorphism $\gr(\W^\Betti) \simeq \W^\DeRham$. An alternate proof of \cite[Proposition 2.16]{EF1} is enabled by the setting introduced in this paper, and is displayed in the following result:
\begin{corollary}
    The following diagram of $\K{\text-}\alg_\gr$ morphisms
    \begin{equation*}
        \begin{tikzcd}
            \gr(\W^\Betti) \ar[rrr, "\gr(\Delta^{\W, \Betti})"] \ar[d, "\simeq"'] &&& \gr(\W^\Betti \otimes \W^\Betti) \ar[d, "\simeq"] \\
            \W^\DeRham \ar[rrr, "\Delta^{\W, \DeRham}"] &&& \W^\DeRham \otimes \W^\DeRham
        \end{tikzcd}
    \end{equation*}
    commutes.
\end{corollary}
\begin{proof}
    Let us consider the following cube
    \[\begin{tikzcd}[row sep={40,between origins}, column sep={40,between origins}]
        &  & \gr(\K \oplus (\V^\Betti, \cdot_{X_1-1})) \arrow[ddd, "\simeq"] \arrow[rrrr, "\gr(\Delta_{\mathcal{O}^\Betti_{\mat, \fil}})"] & & & & \gr(\V^\Betti \otimes \V^\Betti) \arrow[ddd, "\simeq"] \\
        \gr(\W^{\Betti}) \arrow[rrrr, "\hspace{1.9cm} \gr(\Delta^{\W, \Betti})"] \arrow[ddd, "\simeq"'] \arrow[rru, "\simeq"] & & & & \gr(\W^\Betti \otimes \W^\Betti) \arrow[ddd, "\simeq"] \arrow[rru, hook] & & \\
        & & & & & & \\
        & & \K \oplus (\V^\DeRham, \cdot_{e_1}) \arrow[rrrr, "\hspace{-1.8cm}\Delta_{\mathcal{O}^\DeRham_\mat}"'] & & & & \V^\DeRham \otimes \V^\DeRham \\
        \W^\DeRham \arrow[rrrr, "\Delta^{\W, \DeRham}"] \arrow[rru, "\simeq"] & & & & \W^\DeRham \otimes \W^\DeRham \arrow[rru, hook] & &
    \end{tikzcd}\]
    It is immediate that the left and right diagrams commute.
    The lower diagram commutes thanks to Theorem \ref{thm:DeltaWDR_DeltaODR}.
    Thanks to Corollary \ref{cor:DeltaOBfil}, the upper diagram is obtained by applying the $\gr$ functor to the commutative diagram of Theorem \ref{thm:DeltaWB_DeltaOB}, proving its commutativity.
    By applying the functor $\K{\text-}\BFS_\gr \to \mathsf{Mor}(\K{\text-}\alg_\gr)$ to the isomorphism $\gr(\mathcal{O}^\Betti_{\mat, \fil}) \simeq \mathcal{O}^\DeRham_\mat$ from Proposition \ref{prop:iso_grOB_ODR} then using the equality $\gr(\Delta_{\mathcal{O}^\Betti_{\mat, \fil}}) = \Delta_{\gr(\mathcal{O}^\Betti_{\mat, \fil})}$ obtained from from the right hand side of the commutative diagram of functors \eqref{eq:func_diag}, we deduce that the diagram in the back is commutative.
    Finally, this collection of commutativities enables us to deduce that the front diagram commutes, thus proving the result.
\end{proof}
    \section{The objects \texorpdfstring{$\mathcal{O}^\Betti$}{OB}, \texorpdfstring{$\mathcal{O}^\Betti_\fil$}{OBfil} and \texorpdfstring{$\mathcal{O}^\DeRham$}{ODR}} 

In this section, we present an alternative construction of the bimodules from Sec. \ref{sec:OBmat_OBmatfil_ODRmat}. This construction is geometric, based on the groups and Lie algebras corresponding to braids (see \cite{EF1}). It gives rise to a graded bimodule $\bM^\DeRham$ (Proposition-Definition \ref{propdef:bimod_M}) and a filtered bimodule $\bM^\Betti$ (Proposition-Definition \ref{propdef:bimod_uM}), which are equipped with a bimodule morphism $\bM^\DeRham \to \gr(\bM^\Betti)$ (see \eqref{eq:morph_MDR_grMB}).
The bimodules $\bM^\Betti$ and $\bM^\DeRham$ are respectively related by isomorphisms to the filtered and graded bimodules of Sec. \ref{sec:OBmat_OBmatfil_ODRmat} (see Theorems \ref{thm:betti_bimod_iso_fil} and \ref{thm:derham_bimod_iso}), which enables one to equip them with factorization structures by pullback (Proposition-Definitions \ref{propdef:facto_struct_MB} and \ref{propdef:facto_struct_MDR}) and to prove that $\bM^\DeRham \to \gr(\bM^\Betti)$ is an isomorphism (Theorem \ref{thm:grbimodB_iso_bimodDR}).

\subsection{Betti geometric material} \label{sec:betti_geom}
Let $K_4$ be the braid group with four strands, that is, the group presented by generators $x_{12}$, $x_{13}$, $x_{14}$, $x_{23}$, $x_{24}$ and $x_{34}$ which satisfy the following relations \cite{Art}:
\begin{enumerate}[label=(\roman*), leftmargin=*]
    \item $(x_{ij} x_{ik} x_{jk}, x_{ij}) = (x_{ij} x_{ik} x_{jk}, x_{ik}) = (x_{ij} x_{ik} x_{jk}, x_{jk}) = 1$ for $i<j<k \in \llbracket 1, 4 \rrbracket$.
    \item $(x_{12}, x_{34}) = (x_{13}, x_{12}^{-1} x_{24} x_{12}) = (x_{14}, x_{23}) = 1$.
\end{enumerate}

\noindent Let $\omega_4 := x_{12} x_{13} x_{23} x_{14} x_{24} x_{34}$. One checks that $\omega_4$ is a generator of the group $Z(K_4)$. One then defines
\[
    P_5^{\ast} := K_4 / Z(K_4) = K_4 / \langle\omega_4\rangle.
\]
We shall abusively use the same notation the generators of $K_4$ and their classes in $P_5^{\ast}$.

\begin{lemma}[{\cite[Lemma 7.6]{EF1}}] \label{lem:pr}
    \begin{enumerate}[label=(\alph*), leftmargin=*]
        \item There are group morphisms $\underline{\pr}_1, \underline{\pr}_2, \underline{\pr}_5 : P_5^{\ast} \to F_2$ given by
        \[ \def\arraystretch{1.4}
        \begin{array}{|c|c|c|c|c|c|c|}
            \hline
            x & x_{12} & x_{13} & x_{14} & x_{23} & x_{24} & x_{34} \\
            \hline
            \underline{\pr}_1(x) & 1 & 1 & 1 & X_0 & (X_1 X_0)^{-1} & X_1 \\ 
            \hline
            \underline{\pr}_2(x) & 1 & (X_0 X_1)^{-1} & X_0 & 1 & 1 & X_1 \\
            \hline
            \underline{\pr}_5(x) & X_1 & (X_0 X_1)^{-1} & X_0 & X_0 & (X_1 X_0)^{-1} & X_1 \\
            \hline
        \end{array}\]
        \item There is a group morphism $\underline{\ell} : F_2 \to P_5^{\ast}$ given by $X_0 \mapsto x_{23}$ and $X_1 \mapsto x_{12}$. It is such that $\underline{\pr}_5 \circ \underline{\ell} = \mathrm{id}_{F_2}$.
    \end{enumerate}
\end{lemma}

\begin{definition}
    Define $\underline{\pr}_{12} : P_5^{\ast} \to F_2^{2}$ to be the group morphism \(p \mapsto (\underline{\pr}_1(p), \underline{\pr}_2(p))\). We assign to this group morphism and to each group morphism of Lemma \ref{lem:pr} the following morphisms of $\K{\text-}\alg$:
    \begin{enumerate}[label=(\alph*), leftmargin=*]
        \begin{multicols}{2}
        \item $\K\underline{\pr}_j : \K P_5^\ast \to \V^\Betti$, for $j \in \{1, 2, 5\}$;
        \item $\K\underline{\pr}_{12} : \K P_5^\ast \to \V^\Betti \otimes \V^\Betti$;    
        \end{multicols}
        \item $\K\underline{\ell} : \V^\Betti \to \K P_5^\ast$.
    \end{enumerate}
\end{definition}

\noindent Recall that $\ker(\K\underline{\pr}_5)$ is a two-sided ideal of $\K P_5^\ast$. This induces, in particular, a natural right $\K P_5^\ast$-module structure on $\ker(\K\underline{\pr}_5)$. 
\begin{lemma}
    \label{iso_P5_kerkpr5}
    The map $(\K P_5^\ast)^{\oplus 3} \to \ker(\K\underline{\pr}_5)$ given by
   \begin{equation*}
        (\underline{p}_1, \underline{p}_2, \underline{p}_3) \mapsto (x_{15} - 1) \underline{p}_1 + (x_{25} - 1) \underline{p}_2 + (x_{35} - 1) \underline{p}_3
   \end{equation*}
    is a right $\K P_5^\ast$-module isomorphism, where $x_{15}$, $x_{25}$ and $x_{35}$, given by $x_{15}:= (x_{12} x_{13} x_{14})^{-1}$, $x_{25}:= (x_{12} x_{23} x_{24})^{-1}$ and $x_{35}:= (x_{13} x_{23} x_{34})^{-1}$, freely generate $\ker(\K\underline{\pr}_5)$.
\end{lemma}
\begin{proof}
    This is analogous to \cite[Lemma 7.10]{EF1}. The two-sided ideal nature of $\ker(\K\underline{\pr}_5)$ enables one to replace left actions by right actions in this proof.
\end{proof}

\subsection{A Betti bimodule \texorpdfstring{$(\V^\Betti, \bM^\Betti, \V^\Betti \otimes \V^\Betti, \mathsf{r}\underline{\boldsymbol{\rho}})$}{(VB,MB,VBxVB,rubrho)}}

\begin{lemma} \label{bimod:kerupr5}
    The algebra morphism $\K\underline{\ell} : \V^\Betti \to \K P_5^\ast$ equips the $\K$-submodule $\ker(\K\underline{\pr}_5)$ of $\K P_5^\ast$ with a $(\V^\Betti, \K P_5^\ast)$-bimodule structure.
\end{lemma}
\begin{proof}
    Recall that the $\K$-submodule $\ker(\K\underline{\pr}_5)$ of $\K P_5^\ast$ is a two-sided ideal of $\K P_5^\ast$. This naturally equips it with a $(\K P_5^\ast, \K P_5^\ast)$-bimodule structure. Pulling back this structure by the algebra morphism $\K\underline{\ell} : \V^\Betti \to \K P_5^\ast$, one equips  $\ker(\K\underline{\pr}_5)$ with the structure of a $(\V^\Betti, \K P_5^\ast)$-bimodule. Its right module structure is that of $\ker(\K\underline{\pr}_5)$, and its compatible left module structure is given by
    \[
        \V^\Betti \xrightarrow{\K\underline{\ell}} \K P_5^\ast \to \mathrm{End}_{\K P_5^\ast{\text-}\rmod}(\ker(\K\underline{\pr}_5)),
    \]
    where the second arrow is given by the left $\K P_5^\ast$-module structure of $\ker(\K\underline{\pr}_5)$.
\end{proof}

\begin{lemma} \label{bimod:VB2}
    The algebra morphism $\K\underline{\pr}_{12} : \K P_5^\ast \to \V^\Betti \otimes \V^\Betti$ equips the $\K$-module $\V^\Betti \otimes \V^\Betti$ with a $(\K P_5^\ast, \V^\Betti \otimes \V^\Betti)$-bimodule structure.
\end{lemma}
\begin{proof}
    It is immediate that the algebra $\V^\Betti \otimes \V^\Betti$ is naturally a right $(\V^\Betti \otimes \V^\Betti)$-module. On the other hand, let us consider the composition
    \[
        \K P_5^\ast \xrightarrow{\K\underline{\pr}_{12}} \V^\Betti \otimes \V^\Betti \simeq \mathrm{End}_{(\V^\Betti \otimes \V^\Betti){\text-}\rmod}(\V^\Betti \otimes \V^\Betti),
    \]
    where the $\mathrm{End}_{(\V^\Betti \otimes \V^\Betti){\text-}\rmod}(\V^\Betti \otimes \V^\Betti) \simeq \V^\Betti \otimes \V^\Betti$ is the algebra isomorphism given by the evaluation at $1$ as in Proposition-Definition \ref{Delta}. This equips $\V^\Betti \otimes \V^\Betti$ with a compatible left $\K P_5^\ast$-module structure.
\end{proof}

\begin{propdef}
    \label{propdef:bimod_uM}
    The $(\V^\Betti, \K P_5^\ast)$-bimodule $\ker(\K\underline{\pr}_5)$ and the $(\K P_5^\ast, \V^\Betti \otimes \V^\Betti)$-bimodule $\V^\Betti \otimes \V^\Betti$ define a $(\V^\Betti, \V^\Betti \otimes \V^\Betti)$-bimodule 
    \[
        \bM^\Betti := \ker(\K\underline{\pr}_5) \otimes_{\K P_5^\ast} (\V^\Betti \otimes \V^\Betti).
    \]
    More explicitly, the left $\V^\Betti$-module structure on $\bM^\Betti$ is given by
    \[
        \mathsf{r}\underline{\boldsymbol{\rho}} : \V^\Betti \to \mathrm{End}_{(\V^\Betti \otimes \V^\Betti){\text-}\rmod}(\bM^\Betti), \quad \underline{v} \mapsto \left(\underline{p} \otimes \underline{w} \mapsto \K\underline{\ell}(\underline{v})\underline{p} \otimes \underline{w}\right). 
    \]
\end{propdef}
\begin{proof}
    This follows from Lemmas \ref{bimod:kerupr5} and \ref{bimod:VB2} and Proposition \ref{prop:bimodtens}.
\end{proof}

\noindent The remainder of this paragraph is dedicated to the proof of the following result:

\begin{proposition}
    \label{prop:betti_bimod_iso}
    The bimodules
    \[
        (\V^\Betti, \bM^\Betti, \V^\Betti \otimes \V^\Betti, \mathsf{r}\underline{\boldsymbol{\rho}}) \text{ and } (\V^\Betti, (\V^\Betti \otimes \V^\Betti)^{\oplus 3}, \V^\Betti \otimes \V^\Betti, \mathsf{r}\underline{\rho})
    \]
    are isomorphic.
\end{proposition}

\noindent To prove the proposition, we first establish that the algebra morphism $\mathsf{r}\underline{\rho} : \V^\Betti \to \M_3(\V^\Betti \otimes \V^\Betti)$ is related to the geometric material introduced in Sec. \ref{sec:betti_geom}.
\begin{lemma}[{\cite[Lemma 4.1 and Sec. 7.2.2]{EF1}}]
    \begin{enumerate}[label=(\alph*), leftmargin=*]
        \item For any $\underline p\in \K P_5^\ast$, there is a unique matrix $(\underline{a}_{ij}(\underline{p}))_{i,j \in \llbracket 1, 3\rrbracket}\in \M_3(\K P_5^\ast)$ such that
        \begin{equation} \label{matrix_ua}
            (x_{i5} - 1) \ \underline{p} = \sum_{i=1}^3 \underline{a}_{ij}(\underline{p}) \ (x_{j5} - 1).
        \end{equation}
        \item The map
        \[
            \underline{\varpi} : \K P_5^\ast \to \M_3(\K P_5^\ast), \quad \underline{p} \mapsto \left(\underline{a}_{ij}(\underline{p})\right)_{i,j \in \llbracket 1, 3\rrbracket}  
        \]
        is an algebra morphism. 
    \end{enumerate}
    \label{lem:intro_uvarpi}  
\end{lemma}

\begin{lemma}
    \begin{enumerate}[label=(\alph*), leftmargin=*]
        \item \label{matrix_ub} For any $\underline p\in \K P_5^\ast$, there is a unique matrix $(\underline{b}_{ij}(\underline{p}))_{i,j \in \llbracket 1, 3\rrbracket}\in \M_3(\K P_5^\ast)$ such that
        \[
            \underline{p} \ (x_{j5} - 1) = \sum_{i=1}^3 (x_{i5} - 1) \ \underline{b}_{ij}(\underline{p}).
        \]
        \item \label{ruvarpi} The map
        \[
            \mathsf{r}\underline{\varpi} : \K P_5^\ast \to \M_3(\K P_5^\ast), \quad \underline{p} \mapsto \left(\underline{b}_{ij}(\underline{p})\right)_{i,j \in \llbracket 1, 3\rrbracket}  
        \]
        is an algebra morphism. 
    \end{enumerate}
    \label{lem:intro_ruvarpi}  
\end{lemma}

\begin{proof}
    Both \ref{matrix_ub} and \ref{ruvarpi} are right-analogues of Lemma \ref{lem:intro_uvarpi}.
\end{proof}

\begin{lemma} \label{lem:ruvarpi_uvarpi}
    We have (equality of algebra morphisms $\K{P_5}^\ast \to \M_3(\K{P_5}^\ast)$)
    \[
        \mathsf{r}\underline{\varpi} = \Ad_{\diag(x_{15}, x_{25}, x_{35})^{-1}} \circ \M_3(\mathrm{op}_{P_5^\ast}) \circ {^t}(-) \circ \underline{\varpi} \circ \mathrm{op}_{P_5^\ast},
    \]
    where $\mathrm{op}_{P_5^\ast}$ and $\M_3(\mathrm{op}_{P_5^\ast})$ are given by Definitions \ref{def:opG} and \ref{def:Mtsf} respectively.
\end{lemma}
\begin{proof}
    Let $\underline{p} \in \K P_5^\ast$ and $i \in \llbracket 1, 3\rrbracket$. Applying the antimorphism $\mathrm{op}_{P_5^\ast}$ to the equality \eqref{matrix_ua} enables us to obtain
    \[
        \mathrm{op}_{P_5^\ast}(\underline{p}) \ (x_{i5}^{-1} - 1) = \sum_{j=1}^3 (x_{j5}^{-1} - 1) \ \mathrm{op}_{P_5^\ast}\left(\underline{a}_{ij}(\underline{p})\right).
    \]
    Set $\underline{q} := \mathrm{op}_{P_5^\ast}(\underline{p}) (-x_{i5}^{-1})$. We then have
    \begin{align}
        \label{uvarpiop}
        \underline{q} \ (x_{i5} - 1) & = \sum_{j=1}^3 (x_{j5}^{-1} - 1) \ \mathrm{op}_{P_5^\ast}\left(\underline{a}_{ij}\left((-x_{i5}^{-1}) \mathrm{op}_{P_5^\ast}(\underline{q})\right)\right) \\
        & = \sum_{j=1}^3 (x_{j5} - 1) \ (-x_{i5}^{-1}) \ \mathrm{op}_{P_5^\ast}\left(\underline{a}_{ij}\left((-x_{i5}^{-1}) \mathrm{op}_{P_5^\ast}(\underline{q})\right)\right). \notag
    \end{align}
    Let us evaluate $\underline{b}_{ji} := (-x_{i5}^{-1}) \ \mathrm{op}_{P_5^\ast}\left(\underline{a}_{ij}\left((-x_{i5}^{-1}) \mathrm{op}_{P_5^\ast}(\underline{q})\right)\right)$ for any $j \in \llbracket 1, 3\rrbracket$. We have
    \begin{align*}
        \underline{b}_{ji} & = (-x_{i5}^{-1}) \ \mathrm{op}_{P_5^\ast}\left(\underline{a}_{ij}\left((-x_{i5}^{-1}) \mathrm{op}_{P_5^\ast}(\underline{q})\right)\right) = (-x_{i5}^{-1}) \mathrm{op}_{P^\ast_5}\left(\sum_{k=1}^3 \underline{a}_{ik}(-x_{i5}^{-1}) \underline{a}_{kj}(\mathrm{op}_{P^\ast_5}(\underline{q}))\right) \\
        & = (-x_{i5}^{-1}) \mathrm{op}_{P^\ast_5}\left(-x_{i5}^{-1} \ \underline{a}_{ij}(\mathrm{op}_{P^\ast_5}(\underline{q})) \right) = x_{i5}^{-1} \ \mathrm{op}_{P^\ast_5}(\underline{a}_{ij}(\mathrm{op}_{P^\ast_5}(\underline{q}))) \ x_{i5}, 
    \end{align*}
    where the first equality comes from the fact that $\underline{\varpi}$ is an algebra morphism and the second equality follows 
    from the identity $\underline{a}_{ik}(-x_{i5}^{-1}) = -x_{i5}^{-1} \delta_{ik}$, for any $k \in \llbracket 1, 3\rrbracket$. Therefore, in equality \eqref{uvarpiop}, we have
    \[
        \underline{q} \ (x_{i5} - 1) = \sum_{j=1}^3 (x_{j5} - 1) \ x_{i5}^{-1} \ \mathrm{op}_{P^\ast_5}(\underline{a}_{ij}(\mathrm{op}_{P^\ast_5}(\underline{q}))) \ x_{i5}. 
    \]
    This implies the equality 
    \[
        \mathsf{r}\underline{\varpi}(\underline{q}) = \diag(x_{15}, x_{25}, x_{35})^{-1} \ \M_3(\mathrm{op}_{P_5^\ast}) \left({^t}(\underline{\varpi}(\mathrm{op}_{P_5^\ast}(\underline{q})))\right) \ \diag(x_{15}, x_{25}, x_{35}),
    \]
    from which one immediately deduces the equality for $\underline{p}$ thanks to the bijectivity of the map $\underline{p} \mapsto \mathrm{op}_{P_5^\ast}(\underline{p}) (-x_{i5}^{-1})=\underline{q}$, for any $i \in \llbracket 1, 3\rrbracket$.
\end{proof}

\begin{lemma} \label{lem:rurho_ruvarpi}
    We have (equality of algebra morphisms $\V^\Betti \to \M_3(\V^\Betti \otimes \V^\Betti)$)
    \[
        \mathsf{r}\underline{\rho} = \M_3(\K\underline{\pr}_{12}) \circ \mathsf{r}\underline{\varpi} \circ \K\underline{\ell}.
    \]
\end{lemma}
\begin{proof}
    Recall from \cite[(7.2.1)]{EF1} the equality of algebra morphisms $\V^\Betti \to \M_3(\V^\Betti \otimes \V^\Betti)$: 
    \[
        \underline{\rho} = \M_3(\K\underline{\pr}_{12}) \circ \underline{\varpi} \circ \K\underline{\ell}.
    \]
    We have
    \begin{align*}
        \mathsf{r}\underline{\rho} & = \Ad_{\diag(Y_1, X_1, (X_0X_1)^{-1} Y_1^{-1} Y_0 Y_1)^{-1}} \circ \M_3(\mathrm{op}_{F_2^2}) \circ {^t}(-) \circ \underline{\rho} \circ \mathrm{op}_{F_2} \\
        & = \Ad_{\diag(Y_1, X_1, (X_0X_1)^{-1} Y_1^{-1} Y_0 Y_1)^{-1}} \circ \M_3(\mathrm{op}_{F_2^2}) \circ {^t}(-) \circ \M_3(\K\underline{\pr}_{12}) \circ \underline{\varpi} \circ \K\underline{\ell} \circ \mathrm{op}_{F_2} \\
        & = \Ad_{\diag(Y_1, X_1, (X_0X_1)^{-1} Y_1^{-1} Y_0 Y_1)^{-1}} \circ \M_3(\mathrm{op}_{F_2^2}) \circ \M_3(\K\underline{\pr}_{12}) \circ {^t}(-) \circ \underline{\varpi} \circ \K\underline{\ell} \circ \mathrm{op}_{F_2} \\
        & = \Ad_{\diag(\underline{\pr}_{12}(x_{15}), \underline{\pr}_{12}(x_{25}), \underline{\pr}_{12}(x_{35}))^{-1}} \circ \M_3(\K\underline{\pr}_{12}) \circ \M_3(\mathrm{op}_{P_5^\ast}) \circ {^t}(-) \circ \underline{\varpi} \circ \mathrm{op}_{P_5^\ast} \circ \K\underline{\ell} \\
        & = \M_3(\K\underline{\pr}_{12}) \circ \Ad_{\diag(x_{15}, x_{25}, x_{35})^{-1}} \circ \M_3(\mathrm{op}_{P_5^\ast}) \circ {^t}(-) \circ \underline{\varpi} \circ \mathrm{op}_{P_5^\ast} \circ \K\underline{\ell} \\
        & = \M_3(\K\underline{\pr}_{12}) \circ \mathsf{r}\underline{\varpi} \circ \K\underline{\ell},
    \end{align*}
    where the third equality comes from identity \eqref{tMM_MtM}, the fourth one from Lemma \ref{lem:commut_op} applied to both group morphisms $\underline{\pr}_{12}$ and $\underline{\ell}$ and the sixth one from Lemma \ref{lem:ruvarpi_uvarpi}.  
\end{proof}

\begin{lemma}
    The map $(\V^\Betti \otimes \V^\Betti)^{\oplus 3} \to \bM^\Betti$ given by 
    \begin{equation}
        \label{iso:uA3_uM}
        (a_1, a_2, a_3) \mapsto \sum_{i=1}^3 (x_{i5} -1 ) \otimes a_i
    \end{equation}
    is a right $(\V^\Betti \otimes \V^\Betti)$-module isomorphism.
\end{lemma}
\begin{proof}
    Recall from Lemma \ref{iso_P5_kerkpr5} the right $\K P_5^\ast$-module isomorphism $(\K P_5^\ast)^{\oplus 3} \to \ker(\K\underline{\pr}_5)$. The left $\K P_5^\ast$-module structure on $\V^\Betti \otimes \V^\Betti$ given by Lemma \ref{bimod:VB2} enables us to apply the functor $- \otimes_{\K P_5^\ast} (\V^\Betti \otimes \V^\Betti)$ to this isomorphism. This induces a right $\V^\Betti \otimes \V^\Betti$-module isomorphism, which is given by the announced formula.
\end{proof}

\begin{proof}[Proof of Proposition \ref{prop:betti_bimod_iso}]
    It follows from Lemma \ref{lem:intro_ruvarpi} that $\bM^\Betti$ is a $(\K P_5^\ast,\V^\Betti \otimes \V^\Betti)$-bimodule for the left action given by $\M_3(\K\underline{\pr}_{12}) \circ \mathsf{r}\underline{\varpi}$, and that the map \eqref{iso:uA3_uM} is an isomorphism of $(\K P_5^\ast,\V^\Betti \otimes \V^\Betti)$-bimodules. Applying the pull-back by the morphism $\K\underline{\ell} : \V^\Betti \to \K P_5^\ast$ it follows that $\bM^\Betti$ is a $(\V^\Betti, \V^\Betti \otimes \V^\Betti)$-bimodule for the left action given by $\M_3(\K\underline{\pr}_{12}) \circ \mathsf{r}\underline{\varpi} \circ \K\underline{\ell}$, and that the map \eqref{iso:uA3_uM} is an isomorphism of $(\V^\Betti,\V^\Betti \otimes \V^\Betti)$-bimodules. This proves the wanted result since $\mathsf{r}\underline{\rho} = \M_3(\K\underline{\pr}_{12}) \circ \mathsf{r}\underline{\varpi} \circ \K\underline{\ell}$, thanks to Lemma \ref{lem:rurho_ruvarpi}.
\end{proof}

\subsection{A filtration on \texorpdfstring{$(\V^\Betti, \bM^\Betti, \V^\Betti \otimes \V^\Betti, \mathsf{r}\underline{\boldsymbol{\rho}})$}{(VB,MB,VBxVB,rubrho)}}

Recall that the group algebra $\K{P_5^\ast}$ is naturally equipped with a filtration
\begin{equation*}
    \F^0\K{P_5^\ast} = \K{P_5^\ast} \text{ and for } n \geq 1, \ \F^n\K{P_5^\ast} = I_{P_5^\ast}^n,
\end{equation*}
where $I_{P_5^\ast}$ denotes the augmentation ideal of the group algebra $\K{P_5^\ast}$, which is the $\K$-submodule of $\K{P_5^\ast}$ generated by the elements $p - 1$, where $p \in P_5^\ast$.
The pair $(\K{P_5^\ast}, (\F^n\K{P_5^\ast})_{n \in \Z})$ is an object of $\K{\text-}\alg_\fil$.
Additionally, recall that the $\K$-module $\ker(\K\underline{\pr}_5 : \K P_5^*\to \K F_2)$ is a two-sided ideal $\K P_5^*$, and therefore, a $(\K P_5^*, \K P_5^*)$-bimodule.  

\begin{lemma} \label{lem:fil_pr12}
    For any $n \geq 0$, the morphism
    \[
        \F^n \K{\underline{\pr}_{12}} : \F^n\K P_5^* \to \F^n(\V^\Betti \otimes \V^\Betti)
    \]
    is surjective.
\end{lemma}
\begin{proof}
    The group morphism $\underline{\pr}_{12} : P_5^*\to F_2^2$ is surjective, since
    \[
        x_{23} \mapsto X_0, x_{14} \mapsto Y_0, x_{34} x_{13} x_{14} \mapsto X_1 \text{ and } x_{23} x_{24} x_{34} \mapsto Y_1.    
    \]
    This implies the surjectivity of the morphism $\K\underline{\pr}_{12} : \K{P_5^*} \to \K{F_2^2}$ and then of the induced morphism $I_{P_5^*} \to I_{F_2^2}$, which in turn implies the surjectivity of the morphism $I_{P_5^*}^n \to I_{F_2^2}^n$ for any $n \geq 0$.
\end{proof}

\begin{propdef} \label{kerp5:bimod:fil}
    A filtration $(\F^n \ker(\K\underline{\pr}_5))_{n \in \Z}$ of the $\K$-module $\ker(\K\underline{\pr}_5)$ is given by 
    \[
        \F^n\ker(\K\underline{\pr}_5):=\F^n \K{P_5^\ast} \cap \ker(\K\underline{\pr}_5),
    \]
    for $n \geq 0$.
    Equipped with this filtration, $\ker(\underline{\pr}_5)$ is a filtered $(\mathbf kP_5^*, \mathbf kP_5^*)$-bimodule. 
\end{propdef}
\begin{proof}
    The first statement follows from the fact that $(\F^n\K P_5^*)_{n \in \Z}$ is a filtration of $\K P_5^*$.
    The second statement follows from the fact that $(\F^n\K{P_5^\ast})_{n \in \Z}$ is a filtration of $\K{P_5^\ast}$ as a $(\K{P_5^\ast}, \K{P_5^\ast})$-bimodule. 
\end{proof}

\noindent Recall from Proposition-Definition \ref{propdef:bimod_uM} the $(\V^\Betti, \V^\Betti \otimes \V^\Betti)$-bimodule 
\[
    \bM^\Betti = \ker(\K\underline{\pr}_5) \otimes_{\K P_5^\ast} (\V^\Betti \otimes \V^\Betti).
\]
\begin{propdef} \label{propdef:fil:MM}
    For $n \geq 0$, define
    \[
        \F^n\bM^\Betti := \im\left(\F^n \ker(\K\underline{\pr}_5) \to \ker(\K\underline{\pr}_5) \otimes_{\K P_5^\ast} (\V^\Betti \otimes \V^\Betti), x \mapsto x \otimes 1\right).
    \]
    Then, the pair $(\bM^\Betti, (\F^n\bM^\Betti)_{n \in \Z})$ is a filtered bimodule over the pair formed by the filtered algebras $(\V^\Betti, (\F^n\V^\Betti)_{n \in \Z})$ and $(\V^\Betti \otimes \V^\Betti, (\F^n(\V^\Betti \otimes \V^\Betti))_{n \in \Z})$.
\end{propdef}
\begin{proof}
    Let $m, a, b \geq 0$ and $\mu \in \F^m\bM^\Betti$, $\alpha \in \F^a(\V^\Betti \otimes \V^\Betti)$, $\beta \in \F^b\V^\Betti$. \newline
    By definition of $\F^m\bM^\Betti$, there exists $\widetilde{\mu} \in \F^m\ker(\K\underline{\pr}_5)$ such that $\mu = \widetilde{\mu} \otimes 1$. Then
    \[
        \beta \cdot \mu = \K\underline\ell(\beta) \widetilde{\mu} \otimes 1.
    \]
    Since $\K\underline\ell$ is compatible with filtrations, $\K\underline\ell(\beta) \in \F^b\K P_5^*$, which by Proposition-Definition \ref{kerp5:bimod:fil} implies $\K \underline\ell(\beta) \widetilde{\mu} \in \F^{b+m}\ker(\K\underline{\pr}_5)$, hence $\beta \cdot \mu \in \F^{b+m}\bM^\Betti$. \newline
    Thanks to Lemma \ref{lem:fil_pr12}, the morphism $\F^a \K{\underline{\pr}_{12}} : \F^a\K P_5^* \to \F^a(\V^\Betti \otimes \V^\Betti)$ is surjective. This surjectivity implies the existence of $\widetilde{\alpha} \in \F^a \K P_5^*$ such that $\K\underline{\pr}_{12}(\widetilde{\alpha}) = \alpha$. Then
    \[
        \mu \cdot \alpha = \widetilde{\mu} \otimes \alpha = \widetilde{\mu} \widetilde{\alpha} \otimes 1.
    \]
    Finally, Proposition-Definition \ref{kerp5:bimod:fil} implies that $\widetilde{\mu} \widetilde{\alpha} \in \F^{m+a}\ker(\K\underline{\pr}_5)$, which implies that $\mu \cdot \alpha \in \F^{m+a}\bM^\Betti$.  
\end{proof}

\begin{theorem} \label{thm:betti_bimod_iso_fil}
    The right $(\V^\Betti \otimes \V^\Betti)$-module isomorphism $(\V^\Betti \otimes \V^\Betti)^{\oplus 3} \to \bM^\Betti$ given in \eqref{iso:uA3_uM} is an isomorphism of filtered modules with respect to the filtrations given in Lemma \ref{lem:filVB23} and Proposition-Definition \ref{propdef:fil:MM}.
\end{theorem}
\begin{proof}
    Let us show that for any $a \geq 0$, the isomorphism $(\V^\Betti \otimes \V^\Betti)^{\oplus 3} \to \bM^\Betti$ induces a bijection
    \[
        \F^{a-1} (\V^\Betti \otimes \V^\Betti)^{\oplus 3} \simeq \F^a \bM^\Betti.
    \]
    For this, we fix $a \geq 0$ and we follow these steps:
    \begin{steplist}
        \item \emph{Let us show that the isomorphism $(\V^\Betti \otimes \V^\Betti)^{\oplus 3} \to \bM^\Betti$ induces an injection}
        \begin{equation} \label{eq:inj_FaVB3_FaMB}
            \F^{a-1} (\V^\Betti \otimes \V^\Betti)^{\oplus 3} \hookrightarrow \F^a \bM^\Betti.
        \end{equation}
        Let $(\alpha_1, \alpha_2, \alpha_3) \in \F^{a-1} (\V^\Betti \otimes \V^\Betti)^{\oplus 3}$. Thanks to Lemma \ref{lem:fil_pr12}, the morphism $\F^{a-1} \K{\underline{\pr}_{12}} : \F^{a-1}\K{P_5^\ast} \to \F^{a-1}(\V^\Betti \otimes \V^\Betti)$ is surjective. This surjectivity implies the existence of \linebreak $(\widetilde{\alpha}_1, \widetilde{\alpha}_2, \widetilde{\alpha}_3) \in (\F^{a-1}\K{P_5^\ast})^{\oplus 3}$ such that $\F^{a-1} \K{\underline{\pr}_{12}}(\widetilde{\alpha}_i) = \alpha_i$ for any $i \in \{1, 2, 3\}$. Therefore, the image of $(\alpha_1, \alpha_2, \alpha_3)$ by the isomorphism $(\V^\Betti \otimes \V^\Betti)^{\oplus 3} \to \bM^\Betti$ is given by
        \[
            \sum_{i=1}^3 (x_{i5} - 1) \otimes \alpha_i = \sum_{i=1}^3 (x_{i5} - 1) \widetilde{\alpha}_i \otimes 1. 
        \]
        By construction, we have 
        \[
            \sum_{i=1}^3 (x_{i5} - 1) \widetilde{\alpha}_i \in \ker(\K\underline{\pr}_5) \cap \F^a\K{P_5^\ast} = \F^a \ker(\K\underline{\pr}_5).  
        \]
        The statement then follows.
        \item Denote by $F_3$ the free group with three generators. \emph{Let us show that there exists a $\K$-module isomorphism}
        \begin{equation} \label{eq:isoPhi_P5}
            \K{F_3} \otimes \K{F_2} \simeq \K{P_5^\ast}.
        \end{equation}
        Thanks to \cite[Lemma 7.9 1)]{EF1}, we have an injective group morphism $F_3 \to P_5^\ast$ such that we identify the generators of $F_3$ with $x_{15}, x_{25}$ and $x_{35}$. We then have the following split exact sequence of groups
        \[
            \{1\} \to F_3 \to P_5^\ast \xrightarrow{\pr_5} F_2 \to \{1\},
        \]
        with splitting $\underline{\ell} : F_2 \to P_5^\ast$. Applying Proposition \ref{prop:semidirect_augmentation}\ref{isoPhi} to this sequence we obtain a $\K$-module isomorphism $\K{F_3} \otimes \K{F_2} \simeq \K{P_5^\ast}$ induced by the bijection\footnote{One may also refer to \cite[Lemma 7.9. 2)]{EF1}.} $F_3 \times F_2 \to P_5^\ast$, $(u,u^\prime) \mapsto u \cdot \underline{\ell}(u^\prime)$.
        \item Let $\Theta : F_2 \to \Aut(F_3)$ be the group morphism given by $h \mapsto \left(\Theta_h : x \mapsto \underline{\ell}(h) \ x \ \underline{\ell}(h)^{-1}\right)$. \emph{Let us show that}
        \begin{equation} \label{eq:Thetahab_idF3ab}
            \Theta_h^\mathrm{ab} = \mathrm{id}_{F_3^\mathrm{ab}}, \forall h \in F_2.
        \end{equation}
        Since $\Theta$ is a group morphism, it suffices to show identity \eqref{eq:Thetahab_idF3ab} for $h=X_0$ and $h=X_1$ and since $\Theta_{X_i}$ ($i\in\{0,1\}$) is a group morphism, it suffices to evaluate at the generators of $F_3$. One has
        \[
            \Theta_{X_0}(x_{j5}) = x_{23} x_{j5} x_{23}^{-1} \text{ and } \Theta_{X_1}(x_{j5}) = x_{12} x_{j5} x_{12}^{-1}
        \]
        for $j \in \{1,2,3\}$. The injection $F_3 \hookrightarrow P_5^\ast$ enables us to evaluate these identities in $P_5^\ast$. Furthermore, thanks to \cite[Lemma 7.3]{EF1}, one may compute these identities in $K_5$. We will abusively use the same notations for the generators of $P_5^\ast$ and $K_5$. One has
        \[
            \Theta_{X_0}(x_{15}) = x_{23} x_{15} x_{23}^{-1} = x_{15}, 
        \]
        where the last equality follows from the fact that one has from \cite[(7.1.3)]{EF1} that $(x_{15}, x_{23}) = 1$ in $K_5$. Next, one has
        \[
            \Theta_{X_0}(x_{25}) = x_{23} x_{25} x_{23}^{-1} = x_{23} x_{25} x_{35}^{-1} x_{35} x_{23}^{-1} = x_{35}^{-1} x_{23} x_{25} x_{35} x_{23}^{-1} = x_{35}^{-1} x_{25} x_{35} x_{23} x_{23}^{-1} = x_{35}^{-1} x_{25} x_{35},
        \]
        where the third equality follows from the fact that $(x_{23} x_{25}, x_{35}^{-1}) = 1$ and the fourth one from $(x_{23}, x_{25} x_{35}) = 1$; with both identities in $K_5$ being a consequence of \cite[(7.1.2)]{EF1}. Next, one has
        \begin{align*}
            \Theta_{X_0}(x_{35}) & = x_{23} x_{35} x_{23}^{-1} = x_{23} (x_{25} x_{35})^{-1} x_{25} x_{35} x_{35} x_{23}^{-1} = (x_{25} x_{35})^{-1} x_{23} x_{25} x_{35} x_{35} x_{23}^{-1} \\
            & = (x_{25} x_{35})^{-1} x_{35} x_{23} x_{25} x_{35} x_{23}^{-1} = (x_{25} x_{35})^{-1} x_{35} x_{25} x_{35} x_{23} x_{23}^{-1} = (x_{25} x_{35})^{-1} x_{35} x_{25} x_{35},
        \end{align*}
        where the third equality follows from the fact that $(x_{23}, (x_{25} x_{35})^{-1}) = 1$, the fourth one from $(x_{23} x_{25} x_{35}, x_{35}) = 1$ and the fifth one from $(x_{23}, x_{25} x_{35}) = 1$; with both identities in $K_5$ being a consequence of \cite[(7.1.2)]{EF1}. Next, one has
        \[
            \Theta_{X_1}(x_{15}) = x_{12} x_{15} x_{12}^{-1} = x_{12} x_{15} x_{25}^{-1} x_{25} x_{12}^{-1} = x_{25}^{-1} x_{12} x_{15} x_{25} x_{12}^{-1} = x_{25}^{-1} x_{15} x_{25} x_{12} x_{12}^{-1} = x_{25}^{-1} x_{15} x_{25}, 
        \]
        where the third equality follows from the fact that $(x_{12} x_{15}, x_{25}^{-1}) = 1$ and the fourth one from $(x_{12}, x_{15} x_{25}) = 1$; with both identities in $K_5$ being a consequence of \cite[(7.1.2)]{EF1}. Next, one has
        \begin{align*}
            \Theta_{X_1}(x_{25}) & = x_{12} x_{25} x_{12}^{-1} = x_{12} (x_{15} x_{25})^{-1} x_{15} x_{25} x_{25} x_{12}^{-1} = (x_{15} x_{25})^{-1} x_{12} x_{15} x_{25} x_{25} x_{12}^{-1} \\
            & = (x_{15} x_{25})^{-1} x_{25} x_{12} x_{15} x_{25} x_{12}^{-1} = (x_{15} x_{25})^{-1} x_{25} x_{15} x_{25} x_{12} x_{12}^{-1} = (x_{15} x_{25})^{-1} x_{25} x_{15} x_{25},
        \end{align*}
        where the third equality follows from the fact that $(x_{12}, (x_{15} x_{25})^{-1}) = 1$, the fourth one from $(x_{12} x_{15} x_{25}, x_{25}) = 1$ and the fifth one from $(x_{12}, x_{15} x_{25}) = 1$; with both identities in $K_5$ being a consequence of \cite[(7.1.2)]{EF1}. Next, one has
        \[
            \Theta_{X_1}(x_{35}) = x_{12} x_{35} x_{12}^{-1} = x_{35},
        \]
        where the last equality follows from the fact that one has from \cite[(7.1.3)]{EF1} that $(x_{12}, x_{35}) = 1$ in $K_5$. This implies that
        \begin{align*}
            \Theta_{X_0}^\mathrm{ab}(x_{15}^\mathrm{ab}) & = x_{15}^\mathrm{ab} & \Theta_{X_1}^\mathrm{ab}(x_{15}^\mathrm{ab}) & = (x_{25}^{-1} x_{15} x_{25})^\mathrm{ab} = x_{15}^\mathrm{ab} \\
            \Theta_{X_0}^\mathrm{ab}(x_{25}^\mathrm{ab}) & = (x_{35}^{-1} x_{25} x_{35})^\mathrm{ab} = x_{25}^\mathrm{ab} & \Theta_{X_1}^\mathrm{ab}(x_{25}^\mathrm{ab}) & = ((x_{15} x_{25})^{-1} x_{25} x_{15} x_{25})^\mathrm{ab} = x_{25}^\mathrm{ab}\\
            \Theta_{X_0}^\mathrm{ab}(x_{35}^\mathrm{ab}) & = ((x_{25} x_{35})^{-1} x_{35} x_{25} x_{35})^\mathrm{ab} = x_{35}^\mathrm{ab} & \Theta_{X_1}^\mathrm{ab}(x_{35}^\mathrm{ab}) & = x_{35}^\mathrm{ab}. 
        \end{align*}
        This proves that $\Theta_{X_i}^\mathrm{ab} = \mathrm{id}_{F_3^\mathrm{ab}}$ for $i \in \{0, 1\}$, which establishes statement \eqref{eq:Thetahab_idF3ab}.
        \item \emph{Let us show the equalities}
        \begin{equation} \label{eq:lhs_Fa_ker}
            \sum_{\substack{b+c=a \\ b>0}} \F^b \ \K{F_3} \otimes \F^c \ \K{F_2} = \left(\sum_{b+c=a} I_{F_3}^b \otimes I_{F_2}^c\right) \cap \left(I_{F_3} \otimes \K{F_2}\right) 
        \end{equation}
        \emph{and}
        \begin{equation} \label{eq:rhs_Fa_ker}
            \F^a \ker(\K\underline{\pr}_5) = \F^a \K{P_5^\ast} \cap \ker(\K\underline{\pr}_5) = I_{P_5^\ast}^a \cap \ker(\K\underline{\pr}_5). 
        \end{equation}
        Equality \eqref{eq:rhs_Fa_ker} is immediate. Let us prove equality \eqref{eq:lhs_Fa_ker}. The inclusion ($\subset$) being immediate, let us prove the converse. Indeed, let $\displaystyle x = \sum_{b+c=a} x_{b,c} \in I_{F_3} \otimes \K{F_2}$ with $x_{b,c} \in I_{F_3}^b \otimes I_{F_2}^c$. Since $x \in I_{F_3} \otimes \K{F_2}$, it follows that $\varepsilon_{F_3} \otimes \mathrm{id}_{\K{F_2}}(x) = 0$. Therefore,
        \[
            \sum_{b+c=a} \varepsilon_{F_3} \otimes \mathrm{id}_{\K{F_2}}(x_{b,c}) = 0.
        \]
        Since $\varepsilon_{F_3} \otimes \mathrm{id}_{\K{F_2}}(x_{b,c}) = 0$ for $b>0$, it follows from this equality that $\varepsilon_{F_3} \otimes \mathrm{id}_{\K{F_2}}(x_{0,a}) = 0$, which implies that $x_{0,a} \in I_{F_3} \otimes I_{F_2}^a$, thus
        \[
            x = \sum_{b+c=a} x_{b,c} \in \sum_{\substack{b+c=a\\b>0}} I_{F_3}^b \otimes I_{F_2}^c.
        \]
        This concludes the proof of equality \eqref{eq:lhs_Fa_ker}.
        \item \emph{Let us show that the isomorphism \eqref{eq:isoPhi_P5} induces the following isomorphism}
        \begin{equation} \label{eq:Fa_ker}
            \sum_{\substack{b+c=a \\ b>0}} \F^b \ \K{F_3} \otimes \F^c \ \K{F_2} \simeq \F^a \ker(\K\underline{\pr}_5).
        \end{equation}
        First, as in the proof of \cite[Lemma 7.10]{EF1}, one checks that the commutativity of the diagram
        \[\begin{tikzcd}
            \K{F_3} \otimes \K{F_2} \ar[rr, "\eqref{eq:isoPhi_P5}"] \ar[rrd, "\varepsilon_{F_3} \otimes \mathrm{id}_{\K{F_2}}"'] && \K{P_5^\ast} \ar[d, "\K{\underline{\pr}_5}"] \\
            && \K{F_2}
        \end{tikzcd}\]
        implies that the isomorphism \eqref{eq:isoPhi_P5} induces an isomorphism
        \begin{equation} \label{eq:kF2_IF3_kerpr5}
            I_{F_3} \otimes \K{F_2} \simeq \ker(\K\underline{\pr}_5).
        \end{equation}
        Second, identity \eqref{eq:Thetahab_idF3ab} enables us to apply Proposition \ref{prop:semidirect_augmentation}\ref{isosumIFaxIHbIGn}. From this, one deduces that the isomorphism \eqref{eq:isoPhi_P5} gives rise to an isomorphism
        \begin{equation} \label{eq:IF2_IF3_IP5}
            \sum_{b+c=a} I_{F_3}^b \otimes I_{F_2}^c \simeq I_{P_5^\ast}^a.
        \end{equation}
        It follows from \eqref{eq:kF2_IF3_kerpr5} and \eqref{eq:IF2_IF3_IP5} that the isomorphism \eqref{eq:isoPhi_P5} induces an isomorphism between the intersection of the left-hand sides and the right-hand sides of these isomorphisms. These intersections are respectively given by \eqref{eq:lhs_Fa_ker} and \eqref{eq:rhs_Fa_ker}, which implies the announced statement.
        \item \emph{Let us show that}
        \begin{equation} \label{eq:kerfildec}
            \F^a \ker(\K\underline{\pr}_5) = \sum_{i=1}^3 (x_{i5} - 1) \F^{a-1}\K{P_5^\ast}. 
        \end{equation}
        Recall that $\displaystyle I_{F_3}^b = \sum_{i=1}^3 (x_{i5} - 1) I_{F_3}^{b-1}$. This implies
        \begin{equation} \label{eq:sumFakP5}
            \sum_{\substack{b+c=a \\ b>0}} \F^b \ \K{F_3} \otimes \F^c \ \K{F_2} = \sum_{i=1}^3 (x_{i5} - 1) \sum_{b+c=a-1} \F^b \ \K{F_3} \otimes \F^c \ \K{F_2} \simeq \sum_{i=1}^3 (x_{i5} - 1) \F^{a-1} \K{P_5^\ast},
        \end{equation}
        where the isomorphism is obtained from \eqref{eq:IF2_IF3_IP5} by replacing $a$ with $a-1$. Finally, equality \eqref{eq:kerfildec} follows from \eqref{eq:Fa_ker} and \eqref{eq:sumFakP5}.
    \end{steplist}
    Finally, one deduces that the map $\F^{a-1} (\V^\Betti \otimes \V^\Betti)^{\oplus 3} \to \F^a \bM^\Betti$ is injective from \eqref{eq:inj_FaVB3_FaMB} and surjective from \eqref{eq:kerfildec}, thus proving the theorem.
\end{proof}

\subsection{De Rham geometric material} \label{sec:derham_geom}
Let $\mathfrak{t}_4$ be the infinitesimal braid Lie algebra with four strands, that is, the Lie algebra presented by generators $t_{12}$, $t_{13}$, $t_{14}$, $t_{23}$, $t_{24}$ and $t_{34}$ which satisfy the following relations \cite{Fur, EF1}:
\begin{enumerate}[label=(\roman*), leftmargin=*]
    \item $[t_{ij}, t_{ik} + t_{jk}] = 0$ for $i<j<k \in \llbracket 1, 4 \rrbracket$.
    \item $[t_{ij}, t_{kl}] = 0$ for $i<j, k<l \in \llbracket 1, 4 \rrbracket$ such that $\{i,j\}\cap\{k,l\}=\varnothing$.
\end{enumerate}

\noindent Let $z_4 := t_{12} + t_{13} + t_{23} + t_{14} + t_{24} + t_{34}$. One checks that $z_4$ is a generator of the Lie algebra $Z(\mathfrak{t}_4)$. One then defines
\[
    \mathfrak{p}_5 := \mathfrak{t}_4 / Z(\mathfrak{t}_4) = \mathfrak{t}_4 / (z_4).
\]
We shall abusively use the same notation the generators of $\mathfrak{t}_4$ and their classes in $\mathfrak{p}_5$.

\begin{lemma}[{\cite[Sec 5.1.2]{EF1}}] \label{lem:prLie}
    \begin{enumerate}[label=(\alph*), leftmargin=*]
        \item There are Lie algebra morphisms $\pr_1, \pr_2, \pr_5 : \mathfrak{p}_5 \to \mathfrak{f}_2$ given by
        \[ \def\arraystretch{1.4}
        \begin{array}{|c|c|c|c|c|c|c|}
            \hline
            t & t_{12} & t_{13} & t_{14} & t_{23} & t_{24} & t_{34} \\
            \hline
            \underline{\pr}_1(t) & 0 & 0 & 0 & e_0 & e_\infty & e_1 \\ 
            \hline
            \underline{\pr}_2(t) & 0 & e_\infty & e_0 & 0 & 0 & e_1 \\
            \hline
            \underline{\pr}_5(t) & e_1 & e_\infty & e_0 & e_0 & e_\infty & e_1 \\
            \hline
        \end{array}\]
        \item There is a Lie algebra morphism $\ell : \mathfrak{f}_2 \to \mathfrak{p}_5$ given by $e_0 \mapsto t_{23}$ and $e_1 \mapsto t_{12}$. It is such that $\pr_5 \circ \ell = \mathrm{id}_{\mathfrak{f}_2}$.
    \end{enumerate}
\end{lemma}

\begin{definition}
    Define $\pr_{12} : \mathfrak{p}_5 \to \mathfrak{f}_2^{\oplus 2}$ to be the Lie algebra morphism \(p \mapsto (\pr_1(p), \pr_2(p))\).
    We assign to this Lie algebra morphism and to each Lie algebra morphism of Lemma \ref{lem:prLie} the following morphisms of $\K{\text-}\alg$:
    \begin{enumerate}[label=(\alph*), leftmargin=*]
        \begin{multicols}{2}
        \item For $j \in \{1, 2, 5\}$, $U(\pr_j) : U(\mathfrak{p}_5) \to \V^\DeRham$;
        \item $U(\pr_{12}) : U(\mathfrak{p}_5) \to (\V^\DeRham)^{\otimes 2}$;    
        \end{multicols}
        \item $U(\ell) : \V^\DeRham \to U(\mathfrak{p}_5)$.
    \end{enumerate}
\end{definition}

\noindent Recall that $\ker(U(\pr_5))$ is a two-sided ideal of $U(\mathfrak{p}_5)$. This induces, in particular, a natural right $U(\mathfrak{p}_5)$-module structure on $\ker(U(\pr_5))$. 
\begin{lemma}
    \label{iso_p5_kerUpr5}
    The map $U(\mathfrak{p}_5)^{\oplus 3} \to \ker(U(\pr_5))$ given by $(p_1, p_2, p_3) \mapsto t_{15} p_1 + t_{25} p_2 + t_{35} p_3$ is a right $U(\mathfrak{p}_5)$-module isomorphism, where $t_{15}$, $t_{25}$ and $t_{35}$ given by $t_{15}:= -t_{12} - t_{13} - t_{14}$, $t_{25}:= -t_{12} - t_{23} - t_{24}$ and $t_{35}:= -t_{12} - t_{13} - t_{14}$ freely generate $\ker(U(\pr_5))$.
\end{lemma}
\begin{proof}
    This is analogous to \cite[Lemma 5.5]{EF1}. The two-sided ideal nature of $\ker(U(\pr_5))$ enables one to replace left actions by right actions in this proof.
\end{proof}

\subsection{A De Rham graded bimodule \texorpdfstring{$(\V^\DeRham, \bM^\DeRham, \V^\DeRham \otimes \V^\DeRham, \mathsf{r}\boldsymbol{\rho})$}{(VDRxVDR,VDR,MDR,rbrho)}}
\begin{lemma} \label{bimod:kerpr5}
    The graded algebra morphism $U(\ell) : \V^\DeRham \to U(\mathfrak{p}_5)$ equips the graded $\K$-submodule $\ker(U(\pr_5))$ of $U(\mathfrak{p}_5)$ with a graded $(\V^\DeRham, U(\mathfrak{p}_5))$-bimodule structure.
\end{lemma}
\begin{proof}
    Recall that the graded $\K$-submodule $\ker(U(\pr_5))$ of $U(\mathfrak{p}_5)$ is a two-sided ideal of $U(\mathfrak{p}_5)$. This naturally equips it with a graded $(U(\mathfrak{p}_5), U(\mathfrak{p}_5))$-bimodule structure. Pulling back this structure by the graded algebra morphism $U(\ell) : \V^\DeRham \to U(\mathfrak{p}_5)$, one equips  $\ker(U(\pr_5))$ with the structure of a graded $(\V^\DeRham, U(\mathfrak{p}_5))$-bimodule. Its graded right module structure is that of $\ker(U(\pr_5))$, and its compatible graded left module structure is given by
    \[
        \V^\DeRham \xrightarrow{U(\ell)} U(\mathfrak{p}_5) \to \mathrm{End}_{U(\mathfrak{p}_5)\ast{\text-}\rmod}(\ker(U(\pr_5))),
    \]
    where the second arrow is given by the graded left $U(\mathfrak{p}_5)$-module structure of $\ker(U(\pr_5))$.
\end{proof}

\begin{lemma} \label{bimod:VDR2}
    The graded algebra morphism $U(\pr_{12}) : U(\mathfrak{p}_5) \to \V^\DeRham \otimes \V^\DeRham$ equips the graded $\K$-module $\V^\DeRham \otimes \V^\DeRham$ with a graded $(U(\mathfrak{p}_5), \V^\DeRham \otimes \V^\DeRham)$-bimodule structure.
\end{lemma}
\begin{proof}
    It is immediate that the graded algebra $\V^\DeRham \otimes \V^\DeRham$ is naturally a graded right \linebreak $(\V^\DeRham \otimes \V^\DeRham)$-module. On the other hand, let us consider the composition
    \[
        U(\mathfrak{p}_5) \xrightarrow{U(\pr_{12})} \V^\DeRham \otimes \V^\DeRham \simeq \mathrm{End}_{(\V^\DeRham \otimes \V^\DeRham){\text-}\rmod}(\V^\DeRham \otimes \V^\DeRham),
    \]
    where the $\mathrm{End}_{(\V^\DeRham \otimes \V^\DeRham){\text-}\rmod}(\V^\DeRham \otimes \V^\DeRham) \simeq \V^\DeRham \otimes \V^\DeRham$ is the graded algebra isomorphism given by the evaluation at $1$ as in Proposition-Definition \ref{Delta}. This equips $\V^\DeRham \otimes \V^\DeRham$ with a compatible graded left $U(\mathfrak{p}_5)$-module structure.
\end{proof}

\begin{propdef} \label{propdef:bimod_M}
    The graded $(\V^\DeRham, U(\mathfrak{p}_5))$-bimodule $\ker(U(\pr_5))$ and the graded $(U(\mathfrak{p}_5), \V^\DeRham \otimes \V^\DeRham)$-bimodule $\V^\DeRham \otimes \V^\DeRham$ define a graded $(\V^\DeRham, \V^\DeRham \otimes \V^\DeRham)$-bimodule 
    \[
        \bM^\DeRham := \ker(U(\pr_5)) \otimes_{U(\mathfrak{p}_5)} (\V^\DeRham \otimes \V^\DeRham).
    \]
    More explicitly, the graded left $\V^\DeRham$-module structure on $\bM^\DeRham$ is given by
    \[
        \mathsf{r}\boldsymbol{\rho} : \V^\DeRham \to \mathrm{End}_{(\V^\DeRham \otimes \V^\DeRham){\text-}\rmod}(\bM^\DeRham), \quad v \mapsto \left(p \otimes w \mapsto U(\ell)(v) p \otimes w\right). 
    \]
\end{propdef}
\begin{proof}
    This follows from Lemmas \ref{bimod:kerpr5} and \ref{bimod:VDR2} and Proposition \ref{prop:bimodtens}.
\end{proof}

\noindent The remainder of this paragraph is dedicated to the proof of the following result:
\begin{theorem}
    \label{thm:derham_bimod_iso}
    The graded bimodules
    \[
        (\V^\DeRham, \bM^\DeRham, \V^\DeRham \otimes \V^\DeRham, \mathsf{r}\boldsymbol{\rho}) \text{ and } (\V^\DeRham, (\V^\DeRham \otimes \V^\DeRham)^{\oplus 3}, \V^\DeRham \otimes \V^\DeRham, \mathsf{r}\rho)
    \]
    are isomorphic.
\end{theorem}

\noindent To prove the theorem, we first establish that the graded algebra morphism
\[
    \mathsf{r}\rho : \V^\DeRham \to \M_3(\V^\DeRham \otimes \V^\DeRham)
\]
is related to the geometric material introduced in Sec. \ref{sec:derham_geom}.

\begin{lemma}[{\cite[Lemma 4.1 and Sec. 5.2.2]{EF1}}]
    \begin{enumerate}[label=(\alph*), leftmargin=*]
        \item For any $p \in U(\mathfrak{p}_5)$, there is a unique matrix $(a_{ij}(p))_{i,j \in \llbracket 1, 3\rrbracket}\in \M_3(U(\mathfrak{p}_5))$ such that
        \begin{equation} \label{matrix_a}
            t_{i5} \ p = \sum_{i=1}^3 a_{ij}(p) \ t_{j5}.
        \end{equation}
        \item The map
        \[
            \varpi : U(\mathfrak{p}_5) \to \M_3(U(\mathfrak{p}_5)), \quad p \mapsto \left(a_{ij}(p)\right)_{i,j \in \llbracket 1, 3\rrbracket}  
        \]
        is a graded algebra morphism. 
    \end{enumerate}
    \label{lem:intro_varpi}
\end{lemma}

\begin{lemma}
    \begin{enumerate}[label=(\alph*), leftmargin=*]
        \item \label{matrix_b} For any $p \in U(\mathfrak{p}_5)$, there is a unique matrix $(b_{ij}(p))_{i,j \in \llbracket 1, 3\rrbracket}\in \M_3(U(\mathfrak{p}_5))$ such that
        \[
            p \ t_{j5} = \sum_{i=1}^3 t_{i5} \ b_{ij}(p).
        \]
        \item \label{rvarpi} The map
        \[
            \mathsf{r}\varpi : U(\mathfrak{p}_5) \to \M_3(U(\mathfrak{p}_5)), \quad p \mapsto \left(b_{ij}(p)\right)_{i,j \in \llbracket 1, 3\rrbracket}  
        \]
        is a graded algebra morphism. 
    \end{enumerate}
    \label{lem:intro_rvarpi}  
\end{lemma}

\begin{proof}
    Both \ref{matrix_b} and \ref{rvarpi} are right-analogues of Lemma \ref{lem:intro_varpi}.
\end{proof}

\begin{lemma} \label{lem:rvarpi_varpi}
    We have (equality of graded algebra morphisms $U(\mathfrak{p}_5) \to \M_3(U(\mathfrak{p}_5))$)
    \[
        \mathsf{r}\varpi = \M_3(S_{\mathfrak{p}_5}) \circ {^t}(-) \circ \varpi \circ S_{\mathfrak{p}_5},
    \]
    where $S_{\mathfrak{p}_5}$ and $\M_3(S_{\mathfrak{p}_5})$ are given by Definitions \ref{def:Sg} and \ref{def:Mtsf} respectively.
\end{lemma}
\begin{proof}
    Let $p \in U(\mathfrak{p}_5)$ and $i \in \llbracket 1, 3\rrbracket$. Applying the antimorphism $\mathrm{S}_{\mathfrak{p}_5}$ to the equality \eqref{matrix_a} enables us to obtain
    \[
        \mathrm{S}_{\mathfrak{p}_5}(p) \ t_{i5} = \sum_{j=1}^3 t_{j5} \ \mathrm{S}_{\mathfrak{p}_5}\left(a_{ij}(p)\right).
    \]
    Set $q := \mathrm{S}_{\mathfrak{p}_5}(p)$. Since $\mathrm{S}_{\mathfrak{p}_5}$ is an involution, it follows that $\mathrm{S}_{\mathfrak{p}_5}(q) = p$. We then have
    \[
        q \ t_{i5} = \sum_{j=1}^3 t_{j5} \ \mathrm{S}_{\mathfrak{p}_5}\left(a_{ij}(\mathrm{S}_{\mathfrak{p}_5}(q))\right).
    \]
    Setting $b_{ji} = \mathrm{S}_{\mathfrak{p}_5}\left(a_{ij}(\mathrm{S}_{\mathfrak{p}_5}(q))\right)$, this implies the equality 
    \[
        \mathsf{r}\varpi(q) = \M_3(\mathrm{S}_{\mathfrak{p}_5}) \left({^t}(\varpi(\mathrm{S}_{\mathfrak{p}_5}(q)))\right),
    \]
    from which one immediately deduces the equality for $p$ thanks to the bijectivity of $\mathrm{S}_{\mathfrak{p}_5}$.
\end{proof}

\begin{lemma} \label{lem:rrho_rvarpi}
    We have (equality of graded algebra morphisms $\V^\DeRham \to \M_3(\V^\DeRham \otimes \V^\DeRham)$)
    \[
        \mathsf{r}\rho = \M_3(U(\pr_{12})) \circ \mathsf{r}\varpi \circ U(\ell).
    \]
\end{lemma}
\begin{proof}
    Recall from \cite[(5.2.5)]{EF1} that we have the following equality of graded algebra morphisms $\V^\DeRham \to \M_3(\V^\DeRham \otimes \V^\DeRham)$:
    \[
        \rho = \M_3(U(\pr_{12})) \circ \varpi \circ U(\ell).
    \]
    Therefore,
        \begin{align*}
        \mathsf{r}\rho & = \M_3(\mathrm{S}_{\mathfrak{f}_2^{\oplus 2}}) \circ {^t}(-) \circ \rho \circ \mathrm{S}_{\mathfrak{f}_2} = \M_3(\mathrm{S}_{\mathfrak{f}_2^{\oplus 2}}) \circ {^t}(-) \circ \M_3(U(\pr_{12})) \circ \varpi \circ U(\ell) \circ \mathrm{S}_{\mathfrak{f}_2} \\
        & = \M_3(\mathrm{S}_{\mathfrak{f}_2^{\oplus 2}}) \circ \M_3(U(\pr_{12})) \circ {^t}(-) \circ \varpi \circ U(\ell) \circ \mathrm{S}_{\mathfrak{f}_2} \\
        & = \M_3(U(\pr_{12})) \circ \M_3(\mathrm{S}_{\mathfrak{p}_5}) \circ {^t}(-) \circ \varpi \circ \mathrm{S}_{\mathfrak{p}_5} \circ U(\ell) \\
        & = \M_3(U(\pr_{12})) \circ \mathsf{r}\varpi \circ U(\ell),
    \end{align*}
    where the third equality comes from identity \eqref{tMM_MtM}, the fourth one from Lemma \ref{lem:commut_S} applied to both Lie algebra morphisms $\pr_{12}$ and $\ell$ and the last one from Lemma \ref{lem:rvarpi_varpi}.  
\end{proof}

\begin{lemma}
    The map $(\V^\DeRham \otimes \V^\DeRham)^{\oplus 3} \to \bM^\DeRham$ given by 
    \begin{equation}
        \label{iso:A3_M}
        (a_1, a_2, a_3) \mapsto \sum_{i=1}^3 t_{i5} \otimes a_i
    \end{equation}
    is a graded right $(\V^\DeRham \otimes \V^\DeRham)$-module isomorphism.
\end{lemma}
\begin{proof}
    Consider the graded right $U(\mathfrak{p}_5)$-module isomorphism $U(\mathfrak{p}_5)^{\oplus 3} \to \ker(U(\pr_5))$ of Lemma \ref{iso_p5_kerUpr5}. The graded left $U(\mathfrak{p}_5)$-module structure on $\V^\DeRham \otimes \V^\DeRham$ given by Lemma \ref{bimod:VDR2} enables us to apply the functor $- \otimes_{U(\mathfrak{p}_5)} (\V^\DeRham \otimes \V^\DeRham)$ to this isomorphism. This induces a graded right $(\V^\DeRham \otimes \V^\DeRham)$-module isomorphism, which is given by the announced formula.
\end{proof}

\begin{proof}[Proof of Theorem \ref{thm:derham_bimod_iso}]
    It follows from Lemma \ref{lem:intro_rvarpi} that
    \[
        \bM^\DeRham = \ker(U(\pr_5)) \otimes_{U(\mathfrak{p}_5)} (\V^\DeRham \otimes \V^\DeRham)
    \]
    is a graded $(U(\mathfrak{p}_5), \V^\DeRham \otimes \V^\DeRham)$-bimodule for the left action given by $\M_3(U(\pr_{12})) \circ \mathsf{r}\varpi$, and that the map \eqref{iso:A3_M} is an isomorphism of graded $(U(\mathfrak{p}_5), \V^\DeRham \otimes \V^\DeRham)$-bimodules.
    Applying the pullback by the graded algebra morphism $U(\ell) : \V^\DeRham \to U(\mathfrak{p}_5)$ it follows that $\bM^\DeRham$ is a graded $(\V^\DeRham, \V^\DeRham \otimes \V^\DeRham)$-bimodule for the left action given by $\M_3(U(\pr_{12})) \circ \mathsf{r}\varpi \circ U(\ell)$, and that the map \eqref{iso:A3_M} is an isomorphism of graded $(\V^\DeRham, \V^\DeRham \otimes \V^\DeRham)$-bimodules. This proves the wanted result since $\mathsf{r}\rho = \M_3(U(\pr_{12})) \circ \mathsf{r}\varpi \circ U(\ell)$, thanks to Lemma \ref{lem:rrho_rvarpi}.
\end{proof}

\subsection{The isomorphism of bimodules \texorpdfstring{$\gr(\bM^\Betti)\simeq\bM^\DeRham$}{grbMB=bMDR}}
Recall from Proposition-Definition \ref{propdef:fil:MM} that $\bM^\Betti$ is equipped with a filtered $\K$-module structure. This defines the associated graded $\K$-module $\gr(\bM^\Betti)$. We construct a graded $\K$-module morphism $\bM^\DeRham \to \bM^\Betti$. Let us start with the following lemmas:

\begin{lemma} \label{lem:grker_kergr_pr}
    For $n \in \Z$, we have
    \begin{enumerate}[label=(\alph*), leftmargin=*]
        \item $\gr_n \ker(\K\underline{\pr}_5) \simeq \ker(\gr_n(\K\underline{\pr}_5))$;
        \item $\gr_n \ker(\K\underline{\pr}_{12}) \simeq \ker(\gr_n(\K\underline{\pr}_{12}))$. 
    \end{enumerate}
\end{lemma}
\begin{proof}
    This follows immediately from Lemma \ref{lem:grker_kergr} applied to the group morphisms $\underline{\pr}_5 : P_5^\ast \to F_2$ and $\underline{\pr}_{12} : P_5^\ast \to F_2^2$ respectively. 
\end{proof}

\begin{lemma} \label{lem:bimod_and_kerker}
    \begin{enumerate}[label=(\alph*), leftmargin=*]
        \item There exists a left $U(\mathfrak{p}_5)$-module isomorphism 
        \[
            \bM^\DeRham \simeq \ker(U(\pr_5)) \big/ \ker(U(\pr_5)) \cdot \ker(U(\pr_{12}));
        \]
        \item \label{bimod_and_kerker_Betti} There exists a left $\K{P_5^\ast}$-module isomorphism 
        \[
            \bM^\Betti \simeq \ker(\K{\underline{\pr}_5}) \big/ \ker(\K{\underline{\pr}_5}) \cdot \ker(\K{\underline{\pr}_{12}}).
        \]
    \end{enumerate}
\end{lemma}
\begin{proof}
    This follows from Proposition \ref{prop:bimods_as_ker} with
    \[
        (\B, \A, \boldsymbol{\varphi}, \bM) = (U(\mathfrak{p}_5), \V^\DeRham \otimes \V^\DeRham, U(\pr_{12}), \ker(U(\pr_5)))
    \]
    (resp. $(\B, \A, \boldsymbol{\varphi}, \bM) = (\K{P_5^\ast}, \V^\Betti \otimes \V^\Betti, \K{\underline{\pr}_{12}}, \ker(\K{\underline{\pr}_5}))$).
\end{proof}

\begin{proposition}
    There exists a $\K$-module morphism $\bM^\DeRham \to \gr(\bM^\Betti)$ such that
    \[
        e_{i5} \otimes 1 \mapsto [x_{i5} - 1]_1 \otimes 1, \text{ for } i \in \{1,2,3\},
    \]
    and is compatible with the right actions and the $\K$-algebra morphism $\mbox{\small$\V^\DeRham \otimes \V^\DeRham \to \gr(\V^\Betti \otimes \V^\Betti)$}$.
\end{proposition}
\begin{proof}
    We will follow these steps:
    \begin{steplist}
        \item \label{ker_to_grMB}\emph{Construction of a $\K$-module morphism}
        \begin{equation} \label{eq:ker_to_grMB}
            \bigoplus_{n\in\Z} \frac{\F^n\ker(\K{\underline{\pr}_5})}{\displaystyle\F^{n+1}\ker(\K{\underline{\pr}_5})+\sum_{a+b=n} \F^a\ker(\K{\underline{\pr}_5}) \cdot \F^b\ker(\K{\underline{\pr}_{12}})} \longrightarrow \gr(\bM^\Betti),
        \end{equation}
        \emph{with compatible right actions}. \newline
        Applying the functor $\gr$ to Lemma \ref{lem:bimod_and_kerker} \ref{bimod_and_kerker_Betti}, we obtain
        \begin{equation} \label{eq:grMB_grker}
            \gr\left(\bM^\Betti\right) \simeq \gr\left(\ker(\K{\underline{\pr}_5}) \big/ \ker(\K{\underline{\pr}_5}) \cdot \ker(\K{\underline{\pr}_{12}})\right).
        \end{equation}
        Moreover, we have the equality
        \[
            \mbox{\small$\displaystyle\gr\left(\ker(\K{\underline{\pr}_5}) \big/ \ker(\K{\underline{\pr}_5}) \cdot \ker(\K{\underline{\pr}_{12}})\right) = \bigoplus_{n\in\Z} \frac{\F^n\ker(\K{\underline{\pr}_5})}{\F^{n+1}\ker(\K{\underline{\pr}_5}) + \F^n\ker(\K{\underline{\pr}_5}) \cap \left(\ker(\K{\underline{\pr}_5}) \cdot \ker(\K{\underline{\pr}_{12}})\right)}$}.
        \]
        One therefore obtains a $\K$-module morphism
        \begin{equation} \label{eq:Fker_and_gr}
            \mbox{\small$\displaystyle\bigoplus_{n\in\Z} \frac{\F^n\ker(\K{\underline{\pr}_5})}{\displaystyle\F^{n+1}\ker(\K{\underline{\pr}_5})+\sum_{a+b=n} \F^a\ker(\K{\underline{\pr}_5}) \cdot \F^b\ker(\K{\underline{\pr}_{12}})} \to \gr\left(\ker(\K{\underline{\pr}_5}) \big/ \ker(\K{\underline{\pr}_5}) \cdot \ker(\K{\underline{\pr}_{12}})\right)$} 
        \end{equation}
        given by taking the class of an element of $\F^n\ker(\K{\underline{\pr}_5})$ in the source module to the class of the same element in the target module. The morphism \eqref{eq:ker_to_grMB} is then constructed by composition of \eqref{eq:Fker_and_gr} and \eqref{eq:grMB_grker}. \newline
        The target of the $\K$-module morphism \eqref{eq:ker_to_grMB} is a right module over $\gr\left(\K{P_5^\ast}/\ker(\K{\underline{\pr}_{12}})\right)$, while the source is a right module over
        \[
            \bigoplus_{m\in\Z} \frac{\F^m \K{P_5^\ast}}{\F^m\ker(\K{\underline{\pr}_{12}}) + \F^{m+1}\K{P_5^\ast}} = \gr\left(\K{P_5^\ast}\right) \Big/ \gr\left(\ker(\K{\underline{\pr}_{12}})\right).
        \]
        The $\K$-module morphism \eqref{eq:ker_to_grMB} is then compatible with these right actions and the morphism
        \begin{equation} \label{eq:grgrker_grker}
            \gr\left(\K{P_5^\ast}\right) \Big/ \gr\left(\ker(\K{\underline{\pr}_{12}})\right) \to \gr\left(\K{P_5^\ast}/\ker(\K{\underline{\pr}_{12}})\right),
        \end{equation}
        which is an isomorphism since the filtration of $\ker(\K{\underline{\pr}_{12}})$ is induced by that of $\K{P_5^\ast}$.
        \item \label{ker_iso_grMDR} \emph{Construction of a $\K$-module isomorphism}
        \begin{equation} \label{eq:ker_iso_grMDR}
            \bigoplus_{n\in\Z} \frac{\F^n\ker(\K{\underline{\pr}_5})}{\displaystyle\F^{n+1}\ker(\K{\underline{\pr}_5})+\sum_{a+b=n} \F^a\ker(\K{\underline{\pr}_5}) \cdot \F^b\ker(\K{\underline{\pr}_{12}})} \simeq \bM^\DeRham,
        \end{equation}
        \emph{with compatible right actions}. \newline
        Thanks to Lemma \ref{lem:Fil_gr_ker_quotient} applied to $\varphi = \underline{\pr}_5$ and $\psi = \underline{\pr}_{12}$, and to Lemma \ref{lem:grker_kergr_pr}, we have
        \begin{equation} \label{eq:fil_gr_ker_pr}
            \mbox{\small$\displaystyle\bigoplus_{n\in\Z} \frac{\F^n\ker(\K{\underline{\pr}_5})}{\displaystyle\F^{n+1}\ker(\K{\underline{\pr}_5})+\sum_{a+b=n} \F^a\ker(\K{\underline{\pr}_5}) \cdot \F^b\ker(\K{\underline{\pr}_{12}})} \simeq \bigoplus_{n\in\Z} \frac{\ker(\gr_n(\K{\underline{\pr}_5}))}{\displaystyle\sum_{a+b=n} \ker(\gr_a(\K{\underline{\pr}_5})) \cdot \ker(\gr_b(\K{\underline{\pr}_{12}}))}$}.
        \end{equation}
        The source (resp. target) is a right module over the $\K$-algebra $\gr\left(\K{P_5^\ast}\right) \Big/ \gr\left(\ker(\K{\underline{\pr}_{12}})\right)$ (resp. $\gr\left(\K{P_5^\ast}\right) \Big/ \ker\left(\gr(\K{\underline{\pr}_{12}})\right)$) and the isomorphism \eqref{eq:fil_gr_ker_pr} is compatible with the $\K$-algebra morphism
        \begin{equation} \label{eq:grgrker_grkergr}
            \gr\left(\K{P_5^\ast}\right) \Big/ \gr\left(\ker(\K{\underline{\pr}_{12}})\right) \to \gr\left(\K{P_5^\ast}\right) \Big/ \ker\left(\gr(\K{\underline{\pr}_{12}})\right),
        \end{equation}
        which is an isomorphism by Lemma \ref{lem:grker_kergr}. \newline
        Next, we have the equality
        \[
            \bigoplus_{n\in\Z} \frac{\ker(\gr_n(\K{\underline{\pr}_5}))}{\displaystyle\sum_{a+b=n} \ker(\gr_a(\K{\underline{\pr}_5})) \cdot \ker(\gr_b(\K{\underline{\pr}_{12}}))} = \frac{\ker(\gr(\K{\underline{\pr}_5}))}{ \ker(\gr(\K{\underline{\pr}_5})) \cdot \ker(\gr(\K{\underline{\pr}_{12}}))}.
        \]
        Thanks to Proposition \ref{prop:bimods_as_ker} applied to $\bM = \ker(\gr(\K{\underline{\pr}_5}))$, $\B = \gr(\K{P_5^\ast})$ and $\A = \gr(\K{F_2^2})$, using the surjectivity of $\boldsymbol{\varphi} = \gr(\K{\underline{\pr}_{12}}) : \gr(\K{P_5^\ast}) \to \gr(\K{F_2^2})$, which follows from surjectivity of $\pr_{12} : P_5^\ast \to F_2^2$; it follows that
        \begin{equation} \label{eq:bimod_kerker_pr5}
            \frac{\ker(\gr(\K{\underline{\pr}_5}))}{ \ker(\gr(\K{\underline{\pr}_5})) \cdot \ker(\gr(\K{\underline{\pr}_{12}}))} \simeq \ker(\gr(\K{\underline{\pr}_5})) \otimes_{\gr(\K{P_5^\ast})} \gr(\K{F_2^2}).  
        \end{equation}
        This $\K$-module isomorphism is compatible with the right actions and $\K$-algebra morphism
        \begin{equation} \label{eq:grkergr_gr}
            \gr\left(\K{P_5^\ast}\right) \Big/ \ker\left(\gr(\K{\underline{\pr}_{12}})\right) \to \gr(\K{F_2^2}),
        \end{equation}
        which is an isomorphism by the surjectivity of $\gr(\K{\underline{\pr}_{12}})$. \newline
        Finally, thanks to the graded $\K$-algebra isomorphisms $U(\mathfrak{p}_5) \simeq \gr(\K{P_5^\ast})$ and $U(\mathfrak{f}_2) \simeq \gr(\K{F_2})$, and the commutativity of the diagrams
        \[\begin{tikzcd}
            \gr(\K{P_5^\ast}) \ar[rr, "\gr(\K{\underline{\pr}_5})"] \ar[d, "\simeq"'] && \gr(\K{F_2}) \ar[d, "\simeq"] \\
            U(\mathfrak{p}_5) \ar[rr, "U(\pr_5)"] && U(\mathfrak{f}_2)
        \end{tikzcd}
        \text{ \quad and \quad}
        \begin{tikzcd}
            \gr(\K{P_5^\ast}) \ar[rr, "\gr(\K{\underline{\pr}_{12}})"] \ar[d, "\simeq"'] && \gr(\K{F_2^2}) \ar[d, "\simeq"] \\
            U(\mathfrak{p}_5) \ar[rr, "U(\pr_{12})"] && U(\mathfrak{f}_2^{\oplus 2})
        \end{tikzcd}
        \]
        we obtain the following graded $\K$-module isomorphism
        \begin{equation} \label{eq:bimodgr_bimodDR_iso}
            \ker(\gr(\K{\underline{\pr}_5})) \otimes_{\gr(\K{P_5^\ast})} \gr(\K{F_2^2}) \simeq \ker(U(\underline{\pr}_5)) \otimes_{U(\mathfrak{p}_5)} U(\mathfrak{f}_2^{\oplus 2}) = \bM^\DeRham, 
        \end{equation}
        which is compatible with the right actions and the algebra isomorphism
        \begin{equation} \label{eq:gr_U}
            \gr(\K{F_2^2}) \to U(\mathfrak{f}_2^{\oplus 2}).
        \end{equation}
        The announced isomorphism \eqref{eq:ker_iso_grMDR} is then constructed by composition of \eqref{eq:fil_gr_ker_pr}, \eqref{eq:bimod_kerker_pr5} and \eqref{eq:bimodgr_bimodDR_iso}. It is compatible with the right actions and the $\K$-algebra isomorphism
        \begin{equation} \label{eq:grgrker_U}
            \gr\left(\K{P_5^\ast}\right) \Big/ \gr\left(\ker(\K{\underline{\pr}_{12}})\right) \to U(\mathfrak{f}_2^{\oplus 2})
        \end{equation}
        obtained by the composition of \eqref{eq:grgrker_grkergr}, \eqref{eq:grkergr_gr} and \eqref{eq:gr_U}.
        \item \emph{Conclusion}.\newline
        Composing the morphisms \eqref{eq:ker_to_grMB} and \ref{eq:ker_iso_grMDR} from \ref{ker_to_grMB} and \ref{ker_iso_grMDR} respectively we obtain a $\K$-module morphism
        \begin{equation} \label{eq:morph_MDR_grMB}
            \bM^\DeRham \simeq \bigoplus_{n\in\Z} \frac{\F^n\ker(\K{\underline{\pr}_5})}{\displaystyle\F^{n+1}\ker(\K{\underline{\pr}_5})+\sum_{a+b=n} \F^a\ker(\K{\underline{\pr}_5}) \cdot \F^b\ker(\K{\underline{\pr}_{12}})} \to \gr(\bM^\Betti).
        \end{equation}
        Finally, one may check that this morphism sends the element $e_{i5} \otimes 1 \in \bM^\DeRham$ to the element $[x_{i5} - 1]_1 \otimes 1 \in \gr(\bM^\Betti)$, for $i \in \{1,2,3\}$; and that it is also compatible with the right actions and the $\K$-algebra morphism
        \begin{equation*} 
            \V^\DeRham \otimes \V^\DeRham \to \gr(\V^\Betti \otimes \V^\Betti),
        \end{equation*}
        which coincides with the composition of morphisms \eqref{eq:grgrker_grker} and \eqref{eq:grgrker_U}.
    \end{steplist}
\end{proof}

\begin{theorem} \label{thm:grbimodB_iso_bimodDR}
    The graded module morphism $\bM^\DeRham \to \gr(\bM^\Betti)$ induces a bimodule isomorphism
    \[
        (\V^\DeRham, \bM^\DeRham, \V^\DeRham \otimes \V^\DeRham, \mathsf{r}\boldsymbol{\rho}) \simeq ( \gr \V^\Betti, \gr(\bM^\Betti), \gr(\V^\Betti \otimes \V^\Betti), \gr(\mathsf{r}\underline{\boldsymbol{\rho}})).
    \]
\end{theorem}
\begin{proof}
    Recall the following isomorphisms:
    \begin{itemize}[leftmargin=*]
        \item $(\V^\DeRham \otimes \V^\DeRham)^{\oplus 3} \simeq \bM^\DeRham$ is a right $(\V^\DeRham \otimes \V^\DeRham)$-module isomorphism and is given by Theorem \ref{thm:derham_bimod_iso};
        \item $\gr(\V^\Betti \otimes \V^\Betti)^{\oplus 3} \simeq \gr(\bM^\Betti)$ is a right $\gr(\V^\Betti \otimes \V^\Betti)$-module isomorphism and is given by the filtered bimodule isomorphism of Proposition \ref{prop:betti_bimod_iso} to which one applies the functor $\gr$;
        \item $(\V^\DeRham \otimes \V^\DeRham)^{\oplus 3} \simeq \gr(\V^\Betti \otimes \V^\Betti)^{\oplus 3}$ is a right module isomorphism over the isomorphism $\gr(\V^\Betti) \otimes \gr(\V^\Betti) \simeq \V^\DeRham \otimes \V^\DeRham$ and is given by Proposition \ref{prop:iso_grOB_ODR}.
    \end{itemize}
    Let us show that the following diagram
    \begin{equation} \label{diag:MDR_grMB}
        \begin{tikzcd}
            (\V^\DeRham \otimes \V^\DeRham)^{\oplus 3} \ar[rr, "\simeq"] \ar[d, "\simeq"'] && \gr(\V^\Betti \otimes \V^\Betti)^{\oplus 3} \ar[d, "\simeq"] \\
            \bM^\DeRham \ar[rr] && \gr(\bM^\Betti)
        \end{tikzcd}
    \end{equation}
    commutes. To do so it is sufficient to show that the morphisms
    \[
        (\V^\DeRham \otimes \V^\DeRham)^{\oplus 3} \to \gr(\V^\Betti \otimes \V^\Betti)^{\oplus 3} \to \gr(\bM^\Betti) \text{ and } (\V^\DeRham \otimes \V^\DeRham)^{\oplus 3} \to \bM^\DeRham \to \gr(\bM^\Betti)
    \]
    are equal as right module morphisms over the isomorphism $\gr(\V^\Betti \otimes \V^\Betti) \simeq \V^\DeRham \otimes \V^\DeRham$. Therefore, it suffices to check this equality on the generators of $(\V^\DeRham \otimes \V^\DeRham)^{\oplus 3}$ as a right $(\V^\DeRham \otimes \V^\DeRham)$-module.
    Denote by $u_1, u_2, u_3 \in (\V^\DeRham \otimes \V^\DeRham)^{\oplus 3}$ the canonical generators. For $i \in \{1,2,3\}$, the images of these elements by the above compositions are respectively given by
    \[
        u_i \mapsto \tilde{u}_i \mapsto [x_{i5} - 1]_1 \otimes 1 \text{ and } u_i \mapsto e_{i5} \otimes 1 \mapsto [x_{i5} - 1]_1 \otimes 1,
    \]
    where $\tilde{u}_1, \tilde{u}_2, \tilde{u}_3$ are the canonical generators of the right $\gr(\V^\Betti \otimes \V^\Betti)$-module $\gr(\V^\Betti \otimes \V^\Betti)^{\oplus 3}$.
    Finally, the morphism $\bM^\DeRham \to \gr(\bM^\Betti)$ is indeed an isomorphism since all other arrows of the commutative diagram \eqref{diag:MDR_grMB} are isomorphisms.
\end{proof}

\subsection{Geometric construction of the bimodules with factorization structures \texorpdfstring{$\mathcal{O}^\Betti$}{OB}, \texorpdfstring{$\mathcal{O}^\Betti_\fil$}{OBfil} and \texorpdfstring{$\mathcal{O}^\DeRham$}{ODR}}
\begin{propdef} \label{propdef:facto_struct_MB}
    Define the compositions
    \[
        \underline{\boldsymbol{r}} : \bM^\Betti \simeq (\V^\Betti \otimes \V^\Betti)^{\oplus 3} \xrightarrow{\mathsf{r}\underline{\row}} \V^\Betti \otimes \V^\Betti \text{ and } \underline{\boldsymbol{c}} : \V^\Betti \otimes \V^\Betti \xrightarrow{\mathsf{r}\underline{\col}} (\V^\Betti \otimes \V^\Betti)^{\oplus 3} \simeq \bM^\Betti, 
    \]
    where $(\V^\Betti \otimes \V^\Betti)^{\oplus 3} \simeq \bM^\Betti$ is the right $(\V^\Betti \otimes \V^\Betti)$-module isomorphism given in \eqref{iso:uA3_uM}. Then the tuple
    \begin{equation*} 
        \mathcal{O}^\Betti := (\V^\Betti, \bM^\Betti, \V^\Betti \otimes \V^\Betti, \mathsf{r}\underline{\boldsymbol{\rho}}, X_1 - 1, \underline{\boldsymbol{r}}, \underline{\boldsymbol{c}})
    \end{equation*}
    is an object of $\K{\text-}\BFS$.
\end{propdef}
\begin{proof}
    Thanks to Corollary \ref{cor:tuple_OBmat} and Proposition \ref{prop:betti_bimod_iso}, the result follows from Proposition \ref{prop:pullback} applied to
    \[
        (\B, \bM, \A, \rho, e, \boldsymbol{r}, \boldsymbol{c}) = (\V^\Betti, (\V^\Betti \otimes \V^\Betti)^{\oplus 3}, \V^\Betti \otimes \V^\Betti, \underline{\mathsf{r}\rho}, X_1-1, \mathsf{r}\underline{\row}, \mathsf{r}\underline{\col})
    \]
    and
    \[
        (\B^\prime, \bM^\prime, \A^\prime, \rho^\prime) = (\V^\Betti, \bM^\Betti, \V^\Betti \otimes \V^\Betti, \mathsf{r}\underline{\boldsymbol{\rho}}).
    \]
\end{proof}

\noindent Recall from \eqref{eq:filVB} and Lemma \ref{lem:filVB2} that $\V^\Betti$ and $\V^\Betti \otimes \V^\Betti$ are objects of $\K{\text-}\alg_\fil$.
\begin{corollary} \label{cor:OBfil_contruction}
    The filtration given in Proposition-Definition \ref{propdef:fil:MM} define a filtered structure on $\mathcal{O}^\Betti$, which defines an object $\mathcal{O}^\Betti_\fil$ of the category $\K{\text-}\BFS_\fil$.
\end{corollary}
\begin{proof}
    Thanks to Corollary \ref{cor:tuple_OBmatfil} and Theorem \ref{thm:betti_bimod_iso_fil}, the result follows from Proposition \ref{prop:pullback_fil} applied to $\mathcal{O}^\Betti_{\mat, \fil}$ and $(\V^\Betti, \bM^\Betti, \V^\Betti \otimes \V^\Betti, \mathsf{r}\underline{\boldsymbol{\rho}})$.
\end{proof}

\noindent Recall from Sec \ref{sec:ODRmat_grOBmat} that $\V^\DeRham$ and $\V^\DeRham \otimes \V^\DeRham$ are objects of $\K{\text-}\alg_\gr$.
\begin{propdef} \label{propdef:facto_struct_MDR}
    Define the compositions
    \[
        \boldsymbol{r} : \bM^\DeRham \simeq (\V^\DeRham \otimes \V^\DeRham)^{\oplus 3} \xrightarrow{\mathsf{r}\row} \V^\DeRham \otimes \V^\DeRham \text{ and } \boldsymbol{c} : \V^\DeRham \otimes \V^\DeRham \xrightarrow{\mathsf{r}\col} (\V^\DeRham \otimes \V^\DeRham)^{\oplus 3} \simeq \bM^\DeRham, 
    \]
    where $(\V^\DeRham \otimes \V^\DeRham)^{\oplus 3} \simeq \bM^\DeRham$ is the right $(\V^\DeRham \otimes \V^\DeRham)$-module isomorphism given in \eqref{iso:A3_M}. The tuple
    \begin{equation*} 
        \mathcal{O}^\DeRham := (\V^\DeRham, \bM^\DeRham, \V^\DeRham \otimes \V^\DeRham, \mathsf{r}\boldsymbol{\rho}, e_1, \boldsymbol{r}, \boldsymbol{c})
    \end{equation*}
    is an object of $\K{\text-}\BFS_\gr$.
\end{propdef}
\begin{proof}
    Thanks to Corollary \ref{cor:tuple_ODRmat} and Theorem \ref{thm:derham_bimod_iso}, the result follows from Proposition \ref{prop:pullback_gr} applied to
    \[
        (\B, \bM, \A, \rho, e, \boldsymbol{r}, \boldsymbol{c}) = (\V^\DeRham, (\V^\DeRham \otimes \V^\DeRham)^{\oplus 3}, \V^\DeRham \otimes \V^\DeRham, \mathsf{r}\rho, e_1, \mathsf{r}\row, \mathsf{r}\col)
    \]
    and
    \[
        (\B^\prime, \bM^\prime, \A^\prime, \rho^\prime) = (\V^\DeRham, \bM^\DeRham, \V^\DeRham \otimes \V^\DeRham, \mathsf{r}\boldsymbol{\rho}).
    \]
\end{proof}
    \appendix
\section{The morphisms \texorpdfstring{$\mathrm{op}_G$}{opG}, \texorpdfstring{$\mathrm{S}_\mathfrak{g}$}{Sg} and the functor \texorpdfstring{$\M_{t,s}$}{Mts}}
\begin{defn} \label{def:opG}
    For a group $G$, define $\mathrm{op}_G$ to be the group algebra antiautomorphism of $\K{G}$ given by $g \mapsto g^{-1}$ for any $g \in G$.    
\end{defn}

\begin{lem} \label{lem:commut_op}
    If $\varphi : G \to H$ is a group morphism, then the algebra morphisms $\K\varphi : \K{G} \to \K{H}$ satisfies (equality of algebra antimorphisms $\K{G} \to \K{H}$)
    \[
        \K\varphi \circ \mathrm{op}_G = \mathrm{op}_H \circ \K\varphi.
    \]
\end{lem}
\begin{proof}
    It suffices to check the equality for any $g \in G$. We have
    \[
        \K\varphi \circ \mathrm{op}_G(g) = \K\varphi(g^{-1}) = \varphi(g)^{-1} = \mathrm{op}_H(\varphi(g))= \mathrm{op}_H \circ \K\varphi(g),
    \]
    where the middle equality follows from the group morphism status of $\varphi$.
\end{proof}

\begin{defn} \label{def:Sg}
    For a $\K$-Lie algebra $\mathfrak{g}$, define $\mathrm{S}_\mathfrak{g}$ to be the $\K$-algebra antiautomorphism of the universal enveloping algebra $U(\mathfrak{g})$ given by $x \mapsto -x$ for any $x \in \mathfrak{g}$.    
\end{defn}

\begin{lem} \label{lem:commut_S}
    If $\phi : \mathfrak{g} \to \mathfrak{h}$ is a $\K$-Lie algebra morphism, then the $\K$-algebra morphisms $U(\phi) : U(\mathfrak{g}) \to U(\mathfrak{h})$ satisfies (equality of algebra antimorphisms $U(\mathfrak{g}) \to U(\mathfrak{h})$)
    \[
        U(\phi) \circ \mathrm{S}_\mathfrak{g} = \mathrm{S}_\mathfrak{h} \circ U(\phi).
    \]
\end{lem}
\begin{proof}
    It suffices to check the equality for any $x \in \mathfrak{g}$. We have
    \[
        U(\phi) \circ \mathrm{S}_\mathfrak{g}(x) = U(\phi)(-x) = -\phi(x) = \mathrm{S}_\mathfrak{h}(\phi(x))= \mathrm{S}_\mathfrak{h} \circ U(\phi)(x).
    \]
\end{proof}

\begin{defn} \label{def:Mtsf}
    Let $f : E \to F$ be a $\K$-module morphism. For $s, t \in \Z_{> 0}$, define
    \begin{equation*}
        \M_{t,s}(f) : \M_{t,s}(E) \to \M_{t,s}(F)
    \end{equation*}
    to be the $\K$-module morphism that transforms matrices over $E$ to matrices over $F$ by applying $f$ element-wise.
\end{defn}
\noindent Let $f : E \to F$ be a $\K$-module morphism. For $s, t \in \Z_{> 0}$ and $M \in \M_{t,s}(E)$ it is immediate that
\begin{equation}
    \label{tMM_MtM}
    ^t\M_{t,s}(f)(M) = \M_{s,t}(f)({}^tM), 
\end{equation}
where ${}^t(-)$ denotes the transposition of matrices. Assuming that the map $f$ is an algebra antimorphism, one checks that for $s, n, t \in \Z_{> 0}$ and $M \in \M_{t,n}(E)$, $M^\prime \in \M_{n,s}(E)$ we have
\begin{equation}
    \label{tMf_identity}
     ^t\left(\M_{t,s}(f)(M M^\prime)\right) = {}^t\left(\M_{n,s}(f)(M^\prime)\right) {}^t\left(\M_{t,n}(f)(M)\right),
\end{equation}

\section{A lemma on cokernels}
\begin{lem} \label{lem:iso_and_cokers}
    Let $\alpha : A \twoheadrightarrow C$ and $\beta : B \twoheadrightarrow D$ be two surjective $\K$-module morphisms such that there exists a pair of $\K$-module morphisms $\mathrm{f} : A \to B$ and $\mathrm{g} : C \to D$ such that the following diagram
    \begin{equation} \label{diag:iso_and_cokers}
    \begin{tikzcd}
        A \ar[rr, "\alpha", two heads] \ar[d, "\mathrm{f}"'] && C \ar[d, "\mathrm{g}"] \\
        B \ar[rr, "\beta", two heads] && D
    \end{tikzcd}\end{equation}
    commutes. Then, the morphism $\beta$ induces a $\K$-module isomorphism
    \[
        \coker(\mathrm{f}) / \ker(\beta) \simeq \coker(\mathrm{g}).
    \]
\end{lem}
\begin{proof}
    Taking the vertical cokernels of diagram \eqref{diag:iso_and_cokers}, it follows that there exists a unique $\K$-module morphism $\coker(\mathrm{f}) \to \coker(\mathrm{g})$ such that the following diagram
    \begin{equation}\label{diag:vertical_coker}
        \begin{tikzcd}
            B \ar[rr, "\beta", two heads] \ar[d, two heads] && D \ar[d, two heads] \\
            \coker(\mathrm{f}) \ar[rr]  && \coker(\mathrm{g}) 
        \end{tikzcd}
    \end{equation}
    commutes, where the vertical maps are the canonical projections. On the other hand, the map $\ker(\beta) \to \coker(\mathrm{f})$ obtained by the composition
    \[
        \ker(\beta) \hookrightarrow B \twoheadrightarrow \coker(\mathrm{f})
    \]
    is such that the composition
    \[
        \ker(\beta) \to \coker(\mathrm{f}) \to \coker(\mathrm{g})
    \]
    is zero. Indeed, thanks to diagram \eqref{diag:vertical_coker}, this map is equal to the composition
    \[
        \ker(\beta) \hookrightarrow B \xrightarrow{\beta} D \twoheadrightarrow \coker(\mathrm{g}),
    \]
    which is in fact zero. The wanted statement is equivalent to the exactness of the the sequence
    \[
        \ker(\beta) \to \coker(\mathrm{f}) \to \coker(\mathrm{g}) \to \{0\},
    \]
    which we now prove. The surjectivity of $\coker(\mathrm{f}) \to \coker(\mathrm{g})$ immediately follows from the surjectivity of $\beta$ and the commutativity of diagram \eqref{diag:vertical_coker}. Let us now prove that
    \[
        \im(\ker(\beta) \to \coker(\mathrm{f})) = \ker(\coker(\mathrm{f}) \to \coker(\mathrm{g})).
    \]
    We have
    \begin{equation} \label{eq:im_of_exact}
        \im(\ker(\beta) \to \coker(\mathrm{f})) = \frac{\ker(\beta) + \im(\mathrm{f})}{\im(\mathrm{f})} \simeq \frac{\ker(\beta)}{\ker(\beta) \cap \im(\mathrm{f})} 
    \end{equation}
    and
    \begin{equation} \label{eq:ker_of_exact}
        \ker(\coker(\mathrm{f}) \to \coker(\mathrm{g})) = \frac{\{b \in B \mid \beta(b) \in \im(\mathrm{g})\}}{\im(\mathrm{f})}.
    \end{equation}
    The commutative square of inclusions
    \[\begin{tikzcd}
        \ker(\beta) \cap \im(\mathrm{f}) \ar[rr, hook] \ar[d, hook] && \im(\mathrm{f}) \ar[d, hook'] \\
        \ker(\beta) \ar[rr, hook] && \{b \in B \mid \beta(b) \in \im(\mathrm{g})\}
    \end{tikzcd}\]
    gives rise to an injection
    \begin{equation} \label{eq:injection_of_exact}
        \frac{\ker(\beta)}{\ker(\beta) \cap \im(\mathrm{f})} \hookrightarrow \frac{\{b \in B \mid \beta(b) \in \im(\mathrm{g})\}}{\im(\mathrm{f})}
    \end{equation}
    Let us prove its surjectivity. Let $b \in B$ such that $\beta(b) \in \im(\mathrm{g})$. Let then $c \in C$ be such that $\beta(b) = \mathrm{g}(c)$. Since $\alpha : A \to C$ is surjective, there exists $a \in A$ such that $c=\alpha(a)$. It follows that
    \[
        \beta(b) = \mathrm{g} \circ \alpha(a) = \beta \circ \mathrm{f} (a),
    \]
    where the last equality follows from the commutativity of diagram \eqref{diag:iso_and_cokers}. We then have
    \[
        b - \mathrm{f}(a) \in \ker(\beta).
    \]
    The image of the class of $b - \mathrm{f}(a)$ under \eqref{eq:injection_of_exact} is the class of $b - \mathrm{f}(a)$ in the target of \eqref{eq:injection_of_exact}, which is equal to the class of $b$, thus proving the surjectivity of \eqref{eq:injection_of_exact}. Finally, it follows from \eqref{eq:im_of_exact} and \eqref{eq:ker_of_exact} that the map $\im(\ker(\beta) \to \coker(\mathrm{f})) = \ker(\coker(\mathrm{f}) \to \coker(\mathrm{g}))$ is an isomorphism.   
\end{proof}
\section{Some results on group algebras}
\begin{prop} \label{prop:semidirect_augmentation}
    Let
    \[
        \{1\} \to F \xrightarrow{\iota} G \xrightarrow{\pi} H \to \{1\}
    \]
    be a split exact sequence of groups with splitting $\sigma : H \to G$. Denote by $\Theta : H \to \Aut(F)$ the group morphism given by $H\ni h \mapsto \Theta_h \in \Aut(F)$ where
    \[
        \Theta_h : x \mapsto \iota^{-1}\left(\sigma(h) \iota(x) \sigma(h)^{-1}\right)
    \]
    Then
    \begin{enumerate}[label=(\alph*), leftmargin=*]
        \item \label{isoPhi} The composition $\Phi : \K{F} \otimes \K{H} \xrightarrow{\K{\iota} \otimes \K{\sigma}} \K{G} \otimes \K{G} \to \K{G}$ is an isomorphism of left $\K{F}$-modules, where $\K{G} \otimes \K{G} \to \K{G}$ is the product of the group algebra $\K{G}$.
        \item \label{isosumIFaxIHbIGn} Assume that for any $h \in H$, $\Theta_h^\mathrm{ab}=\mathrm{id}_{F^\mathrm{ab}}$ (where $(-)^\mathrm{ab}$ is the abelianization functor). Then, for any $n \geq 0$, the morphism $\Phi : \K{F} \otimes \K{H} \to \K{G}$ induces an isomorphism of left $\K{F}$-modules
        \[
            \sum_{a+b=n} I_F^a \otimes I_H^b \simeq I_G^n.
        \]
    \end{enumerate}
\end{prop}
\begin{proof}
    \begin{enumerate}[label=(\alph*), leftmargin=*]
        \item This follows from the fact that the map $F \times H \to G$ given by $(x,h) \mapsto \iota(x) \sigma(h)$ is a bijection with reciprocal given by $g \mapsto \left(\iota^{-1}\left(g \ \sigma(\pi(g))^{-1}\right), \pi(g)\right)$.
        \item Let $n \geq 0$. Let us show that $\displaystyle \Phi\left(\sum_{a+b=n} I_F^a \otimes I_H^b\right) = I_G^n$ by following these steps:
        \begin{steplist}
            \item \label{step:inclusion} \emph{Let us show that}
            \[
                \Phi\left(\sum_{a+b=n} I_F^a \otimes I_H^b\right) \subset I_G^n.
            \]
            Indeed, for any $a, b \geq 0$ such that $a + b = n$, we have
            \[
                \Phi(I_F^a \otimes I_H^b) \subset \K{\iota}(I_F^a) \K{\sigma}(I_H^b) = \K{\iota}(I_F)^a \K{\sigma}(I_H)^b \subset I_G^a I_G^b = I_G^n,  
            \]
            where the last inclusion follows from the fact group algebra morphisms preserve augmentation ideals. Therefore, we obtain the announced inclusion.
            \item \label{step:grakTheta} \emph{Let $h \in H$. Let us show that for any $a \geq 0$ we have}
            \begin{equation*}
                \gr_a(\K{\Theta_h}) = \mathrm{id}_{I_F^a/I_F^{a+1}}.
            \end{equation*}
            By assumption on $\Theta_h$, for $x \in F$, we have that $\pi_F^\mathrm{ab}(\Theta_h(x)) = \pi_F^\mathrm{ab}(x)$, where $\pi_F^\mathrm{ab} : F \twoheadrightarrow F^\mathrm{ab}$ is the canonical projection. Therefore, there exists $u \in (F,F)$ such that $\Theta_h(x) = x u$, then (equality in $I_F$)
            \[
                \K\Theta_h(x-1) = (x-1) + (u-1) + (x-1)(u-1)
            \]
            It is immediate that $(x-1)(u-1) \in I_F^2$ and thanks to \cite[Exercise 6.1.4]{Wei} we also have $u-1 \in I_F^2$ since $u \in (F,F)$. Hence, (equality in $\gr_1(\K{F}) = I_F / I_F^2$)
            \begin{equation} \label{gr1kThetah_is_id}
                \gr_1(\K{\Theta_h})([x-1]_1) = [x-1]_1.
            \end{equation}
            Recall that the $\K$-module morphism $\K{\Theta_h}$ is a filtered algebra automorphism of $\K{F}$ which implies that $\gr(\K{\Theta_h})$ is a graded algebra automorphism of $\gr(\K{F})$.
            Since the algebra $\gr(\K{F})$ is generated by $\gr_1(\K{F})$, equality \ref{gr1kThetah_is_id} implies that
            \[
                \gr(\K{\Theta_h}) = \mathrm{id}_{\gr(\K{F})},
            \]
            thus proving the wanted identity.
            \item \label{step:g1Phi} Let $a,b \geq 0$ such that $a + b = n$. Let $g \in G$. 
            \begin{enumerate}[label=(\roman*)]
                \item Thanks to the proof of \ref{isoPhi}, there exists a unique $(x, h) \in F \times H$ such that $g = \iota(x) \sigma(h)$.
                \item Thanks to the proof of \ref{step:inclusion}, we define the $\K$-module morphism $\Phi_{a,b} : I_F^a \otimes I_H^b \to I_G^n$ to be the restriction of $\Phi$ to $I_F^a \otimes I_H^b$.
                \item Thanks to \ref{step:grakTheta}, for $u \in I_F^a$, we have that $\K{\Theta_h}(u)-u \in I_F^{a+1}$. Therefore, we may define the $\K$-module morphism
                \[\begin{array}{ccccc}
                    \nu_{a,b}^g & : & I_F^a \otimes I_H^b & \to & I_F^{a+1} \otimes I_H^b \\
                    & & u \otimes v & \mapsto & (x-1) \K{\Theta_h}(u) \otimes v +  (\K{\Theta_h}(u)-u) \otimes v  
                \end{array}\]
                \item Define the $\K$-module morphism
                \[\begin{array}{ccccc}
                    \upsilon_{a,b}^g & : & I_F^a \otimes I_H^b & \to & I_F^a \otimes I_H^{b+1} \\
                    & & u \otimes v & \mapsto & x \ \K{\Theta_h}(u) \otimes (h-1) v  
                \end{array}\]
            \end{enumerate}
            \emph{Let us show that for any $t_{a,b} \in I_F^a \otimes I_H^b$ we have (equality in $I_G^{n+1}$)}
            \begin{equation*}
                (g-1) \Phi_{a,b}(t_{a,b}) = \Phi_{a+1, b} \circ \nu_{a,b}^g(t_{a,b}) + \Phi_{a, b+1} \circ \upsilon_{a,b}^g(t_{a,b}).
            \end{equation*}
            Indeed, by linearity, it suffices to show this equality for $u \otimes v \in I_F^a \otimes I_H^b$. We have
            \begin{align*}
                & \Phi_{a+1, b} \circ \nu_{a,b}^g(u \otimes v) + \Phi_{a, b+1} \circ \upsilon_{a,b}^g(u \otimes v) \\
                & = \K\iota\big((x - 1) \K\Theta_h(u)\big) \K\sigma(v) + \K\iota\big(\K\Theta_h(u) - u\big) \K\sigma(v) + \K\iota\big(x \ \K\Theta_h(u)\big) \K\sigma((h-1) v) \\
                & = \K\iota\big(x \ \K\Theta_h(u)\big) \K\sigma(h v) - \K\iota(u) \K\sigma(v) = \K\iota(x) (\K\iota \circ \K\Theta_h)(u) \K\sigma(h) \K\sigma(v) - \K\iota(u) \K\sigma(v) \\
                & = \K\iota(x) \K\sigma(h) \K\iota(u) \K\sigma(v) - \K\iota(u) \K\sigma(v) = \Big(\K\iota(x) \K\sigma(h) - 1\Big) \K\iota(u) \K\sigma(v) \\
                & = (g-1) \Phi_{a,b}(u \otimes v), 
            \end{align*}
            where the third equality from the fact that both $\K{\iota}$ and $\K{\sigma}$ are algebra morphisms, the fourth one from the definition of $\Theta_h$ and the last one from the definition of $\Phi$ and from $g = \iota(x) \sigma(h)$.
            \item \emph{Let us show by induction on $n$ that}
            \[
                \Phi\left(\sum_{a+b=n} I_F^a \otimes I_H^b\right) \supset I_G^n.
            \]
            First, for $n=0$, this follows from the surjectivity of $\Phi$, thanks to \ref{isoPhi}. Next, since
            \[
                I_G^{n+1} = \sum_{g \in G} (g-1) I_G^n,
            \]
            it suffices to show that for any $g \in G$ and any $z \in I_G^n$ we have
            \[
                (g - 1) z \in \Phi\left(\sum_{a+b=n+1} I_F^a \otimes I_H^b\right).
            \]
            By induction hypothesis, there exists $\displaystyle (t_{a,b})_{a+b=n} \in \bigoplus_{a+b=n} I_F^a \otimes I_H^b$ such that
            \[
                z = \Phi\left(\sum_{a+b=n} t_{a,b}\right) = \sum_{a+b=n} \Phi_{a,b}(t_{a,b}).
            \]
            It follows that
            \begin{align*}
                (g-1) z & = \sum_{a+b=n} (g-1) \Phi_{a,b}(t_{a,b}) = \sum_{a+b=n} \left(\Phi_{a+1, b} \circ \nu_{a,b}^g(t_{a,b}) + \Phi_{a, b+1} \circ \upsilon_{a,b}^g(t_{a,b})\right) \\
                & = \sum_{a+b=n+1} \Phi_{a,b}\left(\nu_{a-1,b}^g(t_{a-1,b}) + \upsilon_{a,b-1}^g(t_{a,b-1})\right) \in \sum_{a+b=n+1} \Phi_{a,b}\left(I_F^a \otimes I_H^b\right), 
            \end{align*}
            where the second equality follows from \ref{step:g1Phi}.
        \end{steplist}
    \end{enumerate}
\end{proof}

\begin{lem} \label{lem:grker_kergr}
    Let $\varphi : G \to H$ be a group morphism and $n \in \Z$. We have
    \begin{enumerate}[label=(\alph*), leftmargin=*]
        \item $\gr_n\ker(\K{\varphi}) \subset \ker(\gr_n \K{\varphi})$;
        \item if, moreover, $\varphi$ is surjective, then $\gr_n\ker(\K{\varphi}) = \ker(\gr_n \K{\varphi})$. 
    \end{enumerate}
\end{lem}
\begin{proof}
    \begin{enumerate}[label=(\alph*), leftmargin=*]
        \item We have
        \begin{align}
            \gr_n\ker(\K{\varphi}) & = \frac{\ker(\K{\varphi}) \cap \F^{n}\K{G}}{\ker(\K{\varphi}) \cap \F^{n+1}\K{G}} \notag \\
            & \subset \frac{\left\{ x \in \F^{n}\K{G} \mid \K{\varphi}(x) \in \F^{n+1}\K{H} \right\}}{\F^{n+1}\K{G}}  \label{grnker_subset_kergrn} \\ 
            & = \ker\left( \frac{\F^{n} \K{G}}{\F^{n+1} \K{G}} \to \frac{\F^{n} \K{H}}{\F^{n+1} \K{H}}\right) = \ker(\gr_n \K{\varphi}). \notag
        \end{align}
        \item Let us now assume that $\varphi : G \to H$ is surjective. It suffices to prove that the inclusion in \eqref{grnker_subset_kergrn} is in fact an equality. Indeed, let $x \in \F^n \K{G}$ such that $\K{\varphi}(x) \in \F^{n+1} \K{H}$. 
        The surjectivity of $\varphi : G \to H$ implies the surjectivity of $\F^{n+1} \K{\varphi} : \F^{n+1} \K{G} \to \F^{n+1} \K{H}$. Therefore, there exists $y \in \F^{n+1} \K{G}$ such that $\K{\varphi}(y) = \K{\varphi}(x) $. It follows that
        \[
            x-y \in \ker(\K{\varphi}) \cap \F^n \K{G},
        \]
        and the image by the inclusion of the class of $x-y$ in the source of the inclusion is the class of $x$ in the target of the inclusion, which implies that the inclusion in \eqref{grnker_subset_kergrn} is an equality.
    \end{enumerate}
\end{proof}

\begin{lem} \label{lem:Fil_gr_ker_quotient}
    Let $\varphi : G \to H$ and $\psi : G \to K$ be two group morphisms. We have (isomorphism of $\K$-modules)
    \[
        \frac{\F^n\ker(\K{\varphi})}{\displaystyle\F^{n+1}\ker(\K{\varphi}) + \sum_{a+b=n} \F^a \ker(\K{\varphi}) \cdot \F^b \ker(\K{\psi})} \simeq \frac{\gr_n(\ker(\K{\varphi}))}{\displaystyle\sum_{a+b=n} \gr_a(\ker(\K{\varphi})) \cdot \gr_b(\ker(\K{\psi}))}.
    \]
\end{lem}
\begin{proof}
    Let $a,b \geq 0$ such that $a+b=n$. We have
    \[
        \F^a\ker(\K{\varphi}) \cdot \F^b\ker(\K{\psi}) \subset \F^a\K{G} \cdot \F^b\K{G} \subset \F^n\K{G},
    \]
    where the second inclusion follows from the fact that $(\F^m\K{G})_{m\in\Z}$ is an algebra filtration. On the other hand, recall that $\ker(\K{\varphi})$ is an ideal of $\K{G}$, therefore, $\ker(\K{\varphi}) \cdot \ker(\K{\psi}) \subset \ker(\K{\varphi})$. This implies that
    \[
        \F^a\ker(\K{\varphi}) \cdot \F^b\ker(\K{\psi}) \subset \F^n\ker(\K{\varphi}).
    \]
    Therefore the product on $\K{G}$ induces a $\K$-module morphism
    \[
        \bigoplus_{a+b=n} \F^a \ker(\K{\varphi}) \otimes \F^b \ker(\K{\psi}) \to \F^n \ker(\varphi). 
    \]
    On the other hand, the fact that $\gr \ker(\K{\varphi})$ is a graded ideal of $\gr \K{G}$ enables us to define a $\K$-module morphism
    \[
        \bigoplus_{a+b=n} \gr_a \ker(\K{\varphi}) \otimes \gr_b \ker(\K{\psi}) \to \gr_n \ker(\K{\varphi}). 
    \]
    Both maps fit in the following diagram
    \begin{equation} \label{diag:Fil_gr_ker}
        \begin{tikzcd}
            \displaystyle \bigoplus_{a+b=n} \F^a \ker(\K{\varphi}) \otimes \F^b \ker(\K{\psi}) \ar[rr, two heads] \ar[d] && \displaystyle \bigoplus_{a+b=n} \gr_a \ker(\K{\varphi}) \otimes \gr_b \ker(\K{\psi}) \ar[d] \\
            \F^n \ker(\K{\varphi}) \ar[rr, two heads] && \gr_n \ker(\K{\varphi})
        \end{tikzcd}
    \end{equation}
    where the horizontal maps are the canonical projection. This diagram is commutative since for any $a,b \geq 0$ such that $a+b=n$ we have for $g \in \F^a \ker(\K{\varphi})$ and $h \in \F^b \ker(\K{\psi})$ that $[g]_a \cdot [h]_b = [gh]_n$. \newline
    Finally, recall that $\ker(\F^n \ker(\K{\varphi}) \to \gr_n \ker(\K{\varphi})) = \F^{n+1} \ker(\K{\varphi})$. The result then follows by applying Lemma \ref{lem:iso_and_cokers} to the commutative diagram \eqref{diag:Fil_gr_ker}.
\end{proof}
    \bibliographystyle{abstract}
    \bibliography{main}
\end{document}